\documentclass[a4paper]{amsart}
\usepackage{amsmath,amssymb,times, amscd, graphicx, xspace,mathrsfs,a4wide,hyperref,url,stmaryrd}
\usepackage{chngcntr}
\usepackage{marginnote}
\usepackage{tikz-cd}
\usepackage[title]{appendix}
\hypersetup{
  colorlinks   = true, 
  urlcolor     = blue, 
  linkcolor    = blue, 
  citecolor   = blue 
}

\newcommand{\Z}{\ensuremath{\mathbb{Z}}\xspace}
\newcommand{\Q}{\ensuremath{\mathbb{Q}}\xspace}
\newcommand{\R}{\ensuremath{\mathbb{R}}\xspace}
\newcommand{\C}{\ensuremath{\mathbb{C}}\xspace}
\newcommand{\A}{\ensuremath{\mathbb{A}}\xspace}
\newcommand{\F}{\ensuremath{\mathbb{F}}\xspace}
\newcommand{\Qp}{\ensuremath{\mathbb{Q}_{p}}\xspace}
\newcommand{\Zp}{\ensuremath{\mathbb{Z}_{p}}\xspace}

\newcommand{\D}{\mathcal{D}}

\newcommand{\m}{\ensuremath{\mathfrak{m}}\xspace}

\newcommand{\n}{\ensuremath{\mathfrak{n}}\xspace}

\newcommand{\p}{\ensuremath{\mathfrak{p}}\xspace}

\newcommand{\Frob}{\mathrm{Frob}\xspace}

\newcommand{\comment}[1]{}

\DeclareMathOperator{\Gal}{Gal}
\DeclareMathOperator{\End}{End}
\DeclareMathOperator{\Hom}{Hom}

\DeclareMathOperator{\Sym}{Sym}
\DeclareMathOperator{\Spec}{Spec}
\DeclareMathOperator{\Spf}{Spf}

\DeclareMathOperator{\Spa}{Spa}
\DeclareMathOperator{\Tor}{Tor}

\DeclareMathOperator{\Ker}{Ker}

\DeclareMathOperator{\Ind}{Ind}
\DeclareMathOperator{\Ext}{Ext}
\newcommand{\GL}{\ensuremath{\mathrm{GL}}\xspace}
\newcommand{\SL}{\ensuremath{\mathrm{SL}}\xspace}

\newcommand{\mbf}{\mathbf}
\newcommand{\mb}{\mathbb}
\newcommand{\mc}{\mathcal}
\newcommand{\ms}{\mathscr}
\newcommand{\mf}{\mathfrak}
\newcommand{\vp}{\varpi}
\newcommand{\ra}{\rightarrow}
\newcommand{\ctens}{\widehat{\otimes}}
\newcommand{\sub}{\subseteq}
\newcommand{\oo}{\mathcal{O}}
\newcommand{\ol}{\overline}

\newcommand{\bu}{\bullet}
\newcommand{\ka}{\kappa}

\newcommand{\wh}{\widehat}
\newcommand{\wt}{\widetilde}

\newcommand{\G}{\Gamma}

\newcommand{\SO}{\mathrm{SO}}
\newcommand{\PGL}{\mathrm{PGL}}
\newcommand{\fX}{{\mathfrak{X}}}

\newcommand{\hH}{\widehat{H}}
\newcommand{\Art}{\mathbf{Art}}
\newcommand{\CNL}{\mathbf{CNL}}
\newcommand{\Set}{\mathbf{Set}}
\newcommand{\Grp}{\mathbf{Grp}}
\newcommand{\E}{\mathcal{E}}
\newcommand{\T}{\mathbf{T}}
\newcommand{\olQp}{\overline{\mathbb{Q}}_p}
\newcommand{\olqp}{\overline{\mathbb{Q}}_p}
\newcommand{\s}{\mathbf{S}}
\newcommand{\ad}{\mathrm{ad}}
\newcommand{\Reff}{\mathrm{Ref}}
\newcommand{\tr}{\mathrm{tr}}

\newtheorem{thmx}{Theorem}
\newtheorem{theorem}{Theorem}[subsection]
\newtheorem{proposition}[theorem]{Proposition}
\newtheorem{corollary}[theorem]{Corollary}
\newtheorem{lemma}[theorem]{Lemma}

\newtheorem{questionx}{Question}
\theoremstyle{definition}
\newtheorem{definition}[theorem]{Definition}

\newtheorem{remark}[theorem]{Remark}

\mathchardef\mhyphen="2D
\title{}
\author{}

\input xy
\xyoption{all}
\usepackage{amscd}

\newcommand{\Qbar}{\overline{\mathbb{Q}}}

\usepackage{color}

\begin{document}

\title{Endoscopy on $\SL_2$-eigenvarieties}
\author{Christian Johansson and Judith Ludwig}

\address{Department of Mathematical Sciences, Chalmers University of Technology and the University of Gothenburg, 412 96 Gothenburg, Sweden}
\email{chrjohv@chalmers.se}

\address{IWR, University of Heidelberg, Im Neuenheimer Feld 205, 69120 Heidelberg, Germany}
\email{judith.ludwig@iwr.uni-heidelberg.de}

\maketitle

\begin{abstract}
In this paper, we study $p$-adic endoscopy on eigenvarieties for $\SL_2$ over totally real fields, taking a geometric perspective. We show that non-automorphic members of endoscopic $L$-packets of regular weight contribute eigenvectors to overconvergent cohomology at critically refined endoscopic points on the eigenvariety, and we precisely quantify this contribution. This gives a new perspective on and generalizes previous work of the second author. Our methods are geometric, and are based on showing that the $\SL_2$-eigenvariety is locally a quotient of an eigenvariety for $\GL_2$, which allows us to explicitly describe the local geometry of the $\SL_2$-eigenvariety. In particular, we show that it often fails to be Gorenstein.
\end{abstract}
\tableofcontents
\counterwithin{equation}{subsection}

\section{Introduction}

A central question in the theory of $p$-adic automorphic forms since its infancy has been to determine when $p$-adic automorphic forms are classical. The answers one might expect will depend on the nature of the theory of $p$-adic automorphic forms that one is using. In this paper, we will be using the overconvergent cohomology of Ash--Stevens, but to begin with we keep the discussion at the level of any `overconvergent' (or `finite slope') theory. In these settings, the original result of this nature is Coleman's result that any overconvergent modular $U_p$-eigenform of weight $k \geq 2$ and slope less than $k-1$ is classical, which has since been generalized to near-complete generality. Going further, one might ask the following question:

\begin{questionx}\label{question}
Assume that $f$ is an overconvergent eigenform for all unramified Hecke operators and the dilating\footnote{Also sometimes called the Atkin--Lehner algebra.} Iwahori--Hecke algebra at $p$, and assume that the system of Hecke eigenvalues of $f$ is classical of classical algebraic weight. Is $f$ itself then classical?
\end{questionx}

Note that the question makes sense: There is a well defined subspace of classical forms inside the space of overconvergent forms, so a priori there can be non-classical eigenforms with the same eigenvalues as a classical form. Question \ref{question} is naturally thought about in terms of eigenvarieties. By its very construction, an eigenvariety $\E$ parametrizes the systems of Hecke eigenvalues of overconvergent automorphic forms, and it carries a coherent sheaf $\mc{M}$ on it whose fibres are the corresponding eigenspaces\footnote{Or typically their duals, suitably interpreted, but we will not worry about this distinction in the introduction.}. Let us assume for simplicity that we are in a `defect zero'-situation where classical points are Zariski dense in the eigenvariety. By Coleman's theorem and its generalizations, all eigenforms are classical for classical eigensystems of regular weight and small slope, and these conditions are generic among all eigensystems of classical weight. To study Question \ref{question} beyond the case of small slopes, one is led to studying the variation of the classical subspace and the local geometry of eigenvarieties near classical points. For $\GL_{2/\Q}$, one can show that the dimension of the classical eigenspace is locally constant (cf. Proposition \ref{constancy of classical subspace} of this paper), and thus the answer to Question \ref{question} is affirmative if and only if $\mc{M}$ is locally free at the point. This local freeness is in turn implied by smoothness of the eigenvariety, which was proved in essentially all cuspidal cases of weight $\geq 2$ by Bella\"iche \cite{bellaiche-crit}. Thus the answer to Question \ref{question} is essentially affirmative for $\GL_{2/\Q}$. In the context of higher rank unitary groups, Question \ref{question} has also been answered affirmatively in many (non-endoscopic) cases by Breuil--Hellmann--Schraen \cite{bhs2,bhs3}.

\medskip

The goal of this paper is to investigate Question \ref{question} for $\SL_{2/F}$, where $F$ is a totally real number field. For $\SL_{2/\Q}$ and its inner forms, one of us (J.L.) has shown that the answer to Question \ref{question} is negative for critically refined endoscopic points, under certain technical assumptions \cite{L2,L1}. This was achieved by showing, using the theory of automorphic representations for these groups \cite{labesse-langlands,lansky-raghuram}, that the dimension of the classical eigenspace drops at such points. On the other hand, the dimension of the fibre of $\mc{M}$ cannot drop by upper semicontinuity. It was speculated in \cite{L1} that the extra non-classical eigenforms should be related to non-automorphic members of the (classical) $L$-packet of the endoscopic point. In this paper, we make this idea precise by computing the whole eigenspace explicitly, including its decomposition into a classical and a non-classical part. At the same time, we substantially generalize the results of \cite{L2,L1}, removing the technical hypotheses of those works and allowing an arbitrary totally real field $F$. This requires a very different approach to understanding spaces of overconvergent automorphic forms and eigenvarieties for $\SL_2$, motivated by ideas from the geometrization of Langlands correspondences.

\subsection{Heuristics on how to compute eigenspaces of overconvergent automorphic forms} \label{subsec: working model}

Before detailing the results of this paper, we wish to give a sketch of how we believe that the structure of these eigenspaces can be understood. The overall picture should be the same regardless of which theory of overconvergent forms we use, so for the purpose of this discussion we give ourselves an eigenvariety $\E$ with its coherent sheaf\footnote{To obtain an object with better structural properties, one should take a suitable direct limit of the sheaves $\mc{M}$ over all tame levels; this will play an important technical role in the paper but we will not elaborate on it here.} $\mc{M}$, for some split (for simplicity) reductive group $G/F$, with $F$ a number field, which we assume has defect zero (i.e. $G(F_v)$ has essentially discrete series for all $v\mid \infty$).  As is now well established in Langlands philosophy, spectra of Hecke algebras should be close to coarse moduli spaces of Galois representations. In our situation, there should be (very roughly speaking) a moduli stack $\mf{X}$ of totally odd, trianguline Galois representations of $G_F := \Gal(\ol{F}/F)$ valued in the dual group\footnote{One should really consider the $C$-group here, but we ignore this point as it does not play a role in the situations we consider in the paper.} $\wh{G}$, with coarse moduli space $X$. There should be maps
\[
\mf{X} \to X \to \E,
\]
and the map $X \to \E$ should (roughly speaking) be an isomorphism when $\wh{G}$ is acceptable, in the sense that its invariant theory is generated by unary invariants, cf. \cite{weidner}. The stack $\mf{X}$ maps to the corresponding stacks $\mf{X}_v$ of $\wh{G}$-valued representations of $G_{F_v} := \Gal(\ol{F}_v/F_v)$ (taking a suitable stack of trianguline, or refined, representations when $v\mid p$), giving us
\[
g = (g_v)_v : \mf{X} \to \prod_v \mf{X}_v.
\]
The sheaf $\mc{M}$ on $\E$ should arise as follows: For $v\nmid p$, there should be a universal family of smooth $G(F_v)$-representations $\mc{F}_v$ on $\mf{X}_v$ interpolating the local Langlands correspondence, cf. \cite{emerton-helm,hellmann2021derived,fargues-scholze}. For $v\mid p$, we expect there to be a suitable $T(F_v)$-representation $\mc{F}_v$ on $\mf{X}_v$, where $T\sub G$ is a maximal torus. This gives us a sheaf
\begin{equation}\label{eq: sheaf local-global}
\mc{F} := \bigotimes_v^\prime g_v^\ast \mc{F}_v
\end{equation}
on $\mf{X}$. Writing $\pi$ for the map $\mf{X} \to \E$, we then expect that
\[
\mc{M} = \pi_\ast \mc{F}.
\]
To simplify the discussion, assume in addition that $G$ is semisimple and simply connected, and that $\wh{G}$ is acceptable (these assumptions hold for $G = \SL_2$). For the purpose of this discussion, this means that we can expect $X=\E$. Looking locally around an irreducible point $\rho \in X$, the map $\mf{X} \to X$ is often an isomorphism. Hypothetically, this can fail because $\rho$ has a non-trivial centralizer group $S_\rho$, or for reasons to do with triangulations/refinements. In this paper, we are mostly interested in the former situation. Locally around $\rho$ in $X$, the geometry would look like
\[
[Y/S_\rho^\prime] \to Y \sslash S_\rho^\prime,
\]
where $Y$ is an affinoid variety with an action of the subgroup $S_\rho^\prime \sub S_\rho$ which fixes the refinement, $[Y/S_\rho^\prime]$ is the stacky quotient and $Y \sslash S_\rho^\prime$ is the GIT quotient. The pushforward map then takes a sheaf on $[Y/S_\rho^\prime]$, viewed as $S_\rho^\prime$-equivariant sheaf on $Y$, to its $S_\rho^\prime$-invariants (as $S_\rho^\prime$ is finite and we are in characteristic $0$, there is no higher cohomology). 

\medskip

Before saying more, let us contrast this to what should be happening for classical spaces of (algebraic) automorphic forms. In that case the same type of picture should apply, but one would now look at the moduli of representations $G_F \to \wh{G}$ which are geometric with fixed Hodge--Tate weights (keeping all assumptions on $G$, for simplicity). The corresponding moduli stack should be zero-dimensional, and irreducible points are isolated, looking like $[\ast/S_\rho]$. The same type of picture should apply --- the sheaves $g_v^\ast \mc{F}_v$ restricted to $[\ast /S_\rho]$ are the local $L$-packets, the sheaf $\mc{F}$ is the global $L$-packet, and pushing forward along
\[
[\ast/S_\rho] \to \ast \sslash S_\rho = \ast
\]
singles out the automorphic part of the $L$-packet --- computing the action of $S_\rho$ on $\mc{F}$ and taking invariants should amount to Arthur's multiplicity formula. In the function field case, much of this picture has been realized in \cite{lafforgue-zhu}, which was a motivation for our work.

\medskip

The difference between the two settings is that the moduli stack $\mf{X}$ in the overconvergent setting is not $0$-dimensional at $\rho$, and the map $\mf{X} \to X$ may therefore be ramified at $\rho$. This induces a filtration on the fibre of $\mc{M}$ at $\rho\in X$, whose graded pieces are parts of the `overconvergent $p$-adic $L$-packet' (which should be taken to be the fibre of $\mc{F}$ at $\rho$, viewed as a point of $\mf{X}$), with multiplicities. The top quotient in this filtration is the (dual of the) classical eigenspace, but as soon as there are more steps in the filtration, we will see non-classical eigenforms. In particular, we expect that the notion of local-global compatibility in the $p$-adic setting is more subtle, with the link between $p$-adic $L$-packets and what is visible in spaces of $p$-adic automorphic forms being more complicated. We also expect that this picture, with straightforward modifications, applies to other theories of $p$-adic automorphic forms as well, such as completed cohomology\footnote{For $\SL_{2/\Q}$, one of us (C.J.) and James Newton have obtained results on local-global compatibility in completed cohomology which confirms this picture.}. We finish this subsection by noting that after the first version of this paper was made public, Emerton--Gee--Hellmann have announced a rather precise conjecture for computing spaces of overconvergent automorphic forms similar in spirit to what we have sketched above; see \cite[Conjecture 9.6.16]{egh}.

\subsection{Results and methods}

The picture that we have given above is of course highly conjectural. Our goal in this paper is, nevertheless, to implement it in spirit for $G=\SL_{2/F}$ with $F$ totally real. In our implementation, moduli spaces or stacks of $\wh{G}=\PGL_2$-representations will not feature\footnote{However, deformation spaces of $\PGL_2$-representations (without any triangulations or refinements) will feature at an important point in our argument.}. Instead we make the following ad hoc replacements (using the notation of \S \ref{subsec: working model}):
\begin{itemize}
\item We replace the global moduli stack $\mf{X}$ by the $\GL_{2/F}$-eigenvariety $\wt{\E}$ together with the action on it by twisting with characters;

\item We replace the sheaf $\mc{F}$ by the coherent sheaf $\wt{\mc{M}}$ on $\wt{\E}$, together with its twisting action.
\end{itemize}
No substitute for the local stacks $\mf{X}_v$ or sheaves $\mc{F}_v$ are needed, and instead a classicality result for $\wt{\mc{M}}$ together with classical local-global compatibility replaces equation (\ref{eq: sheaf local-global}). Using automorphic forms on $\GL_2$ to study automorphic forms on $\SL_2$ has a long history, and was done by Labesse--Langlands \cite{labesse-langlands} in the context of automorphic representation theory, but to the best of our knowledge our work is the first to implement this idea in the context of eigenvarieties. While the comparison between \emph{spaces} of overconvergent automorphic forms on $\GL_2$ and $\SL_2$ is relatively straightforward, the comparison between eigenvarieties (i.e. the analogue of proving $X=\E$) is much more difficult and takes up most of this paper. In some sense this difficulty is natural, since the comparison of eigenvarieties can be regarded as a $p$-adic analogue, in families, of the subtle problem of proving multiplicity $1$ for automorphic representations for $\SL_2$ \cite{ramakrishnan}.

\medskip

Let us now turn to the precise results. We work with the overconvergent cohomology of Ash--Stevens \cite{ash-stevens}, where the corresponding eigenvarieties for $\GL_{2/F}$ have been studied in great detail by Bergdall--Hansen \cite{bh}. In this introduction we start by assuming $F=\Q$, which greatly simplifies many technical aspects. Let $\wt{\E}$ denote the Coleman--Mazur eigencurve, constructed using overconvergent cohomology (see e.g. \cite[\S 6]{JN1}), with tame level being full level $N$ for some $N$ coprime to $p$. We recall that $\wt{\E}$ lives over a weight space $\mc{W}$, which is the moduli space of continuous characters of $\Z_p^\times$. Given an affinoid family $U$ of characters, we have a corresponding module $\D_U$ of locally analytic distributions, dual to the module of locally analytic $\oo(U)$-valued functions on $\Zp$. Overconvergent cohomology is then defined as
\[
H^1(Y_N,\D_U),
\] 
where $Y_N$ is the disconnected modular curve with Iwahori level at $p$ and full tame level $N$. There is an action of Hecke operators on $H^1(Y_N,\D_U)$. At $p$, we have the usual $U_p$- and $S_p$-operators, and also the $\SL_2$-Hecke operator $u_p = U_p^2S_p^{-1}$; both $U_p$ and $u_p$ are induced from compact operators at the level of complexes. To simplify the comparison with the $\SL_2$-theory, we will use $u_p$ instead of the more commonly used $U_p$ when talking about slopes\footnote{Recall the the slope of an eigenform is the $p$-adic valuation of its eigenvalue for a chosen compact Hecke operator at $p$, in this case $u_p$.}. When $H^1(Y_N,\D_U)$ carries a slope decomposition
\[
H^1(Y_N,\D_U) = H^1(Y_N,\D_U)_{\leq h} \oplus H^1(Y_N,\D_U)_{>h}
\]
for some $h \in \Q_{\geq 0}$, the pair $(U,h)$ is said to be a slope datum. The eigencurve $\wt{\E}$ is glued together from the local pieces $\wt{\E}_{U,h}= \Spa \T_{U,h}$, where $\T_{U,h} \sub \End_{\oo(U)}(H^1(Y_N,\D_U)_{\leq h})$ is the sub-$\oo(U)$-algebra generated by the Hecke operators $T_\ell$ for $\ell \nmid Np$ and $U_p$. The coherent sheaf $\wt{\mc{M}}$ on $\wt{\E}$ is defined by
\[
\wt{\mc{M}}(\wt{\E}_{U,h}) = H^1(Y_N,\D_U)_{\leq h}.
\]
The disconnected modular curve has a surjective map $Y_N \to (\Z/N)^\times$ with connected fibres, and the component $Y_N^1$ lying over $1 \in (\Z/N)^\times$ is the $\SL_2$-modular curve of full tame level $N$ and Iwahori level at~$p$. The weight space and the locally analytic distribution modules used to construct the $\SL_2$-eigencurve $\E$ are exactly the same as those used for $\GL_2$. In particular, we have the same slope data $(U,h)$ and the $\SL_2$-slope $\leq h$-overconvergent cohomology modules $H^1(Y_N^1 ,\D_U)_{\leq h}$ are direct summands of the corresponding $\GL_2$-modules $H^1(Y_N,\D_U)_{\leq h}$. Indeed, the map $Y_N \to (\Z/N)^\times$ induces an action of the dual group $H$ of $(\Z/N)^\times$ on $H^1(Y_N,\D_U)_{\leq h}$ (by twisting by characters) and we have
\[
H^1(Y^1_N,\D_U)_{\leq h} = \left( H^1(Y_N,\D_U)_{\leq h} \right)_H.
\]
The $\SL_2$-eigencurve $\E$ is then glued together from the local pieces $\E_{U,h} = \Spa \s_{U,h}$, where $\s_{U,h} \sub \End(H^1(Y^1_N,\D_U)_{\leq h})$ is the $\oo(U)$-subalgebra generated by the $\SL_2$-Hecke operators $t_\ell = T_{\ell}^2 S_{\ell}^{-1}$ for $\ell \nmid Np$ and $u_p$, and the coherent sheaf $\mc{M}$ is defined by
\[
\mc{M}(\E_{U,h}) = H^1(Y_N^1,\D_U)_{\leq h}.
\]
Using classical functoriality from $\GL_2$ to $\SL_2$ and a Chenevier-style interpolation theorem, we construct a surjective and finite morphism $\pi : \wt{\E} \to \E$. The $H$-action on $\wt{\mc{M}}$ induces an $H$-action on $\wt{\E}$ (by twisting), and one might reasonably think that the coarse quotient $\wt{\E} \sslash H$ is equal to $\E$. This, together with the fact that $\mc{M} = (\pi_\ast \wt{\mc{M}})_H$, would enable us to compute the fibres of $\mc{M}$ at classical points in terms of the fibres of $\wt{\mc{M}}$, and we know that the fibres of $\wt{\mc{M}}$ consist of classical eigenforms.

\medskip

As indicated above, the study of the map $\wt{\E} \sslash H \to \E$ requires new ideas. Since we are only interested in the local geometry, we do this locally around the points we are interested in\footnote{In fact, we only construct the map $\pi$ locally.}. The classical points on $\wt{\E}$ we are interested in correspond to pairs $x=(\pi,\alpha)$ where $\pi$ is a cuspidal cohomological automorphic representation of $\GL_2(\A)$ with $\pi^{K^p(N)I} \neq 0$, and $\alpha$ is a refinement (a choice of $U_p$-eigenvalue on $\pi_p^I$), where $I \sub \GL_2(\Zp)$ is the usual upper triangular Iwahori subgroup. Similarly, we will be interested in classical points on $\E$ corresponding to pairs $y=(\Pi,\gamma)$, where $\Pi$ is a global $L$-packet for $\SL_2(\A)$ which is cuspidal and cohomological, with $\SL_2(\A^\infty) \cap K^p(N)I$-fixed vectors, and $\gamma$ is a refinement (a choice of $u_p$-eigenvalue in the $\SL_2(\A^\infty) \cap K^p(N)I$-fixed vectors of $\Pi$). Assume that $x=(\pi,\alpha) \in \wt{\E}$ maps to $y = (\Pi,\gamma) \in \E$. The $L$-packet $\Pi$ is endoscopic if and only if $\pi$ has CM by an imaginary quadratic field $\wt{F}$. If $\rho = \rho_\pi : G_{\Q} \to \GL_2(\ol{\Q}_p)$ is the Galois representation attached to $\pi$, we let $S_\rho$ be the centralizer group of $\ad^0\rho$, viewed as a representation into $\PGL_2(\ol{\Q}_p)$. $S_\rho$ is non-trivial if and only if $\pi$ has CM, in which case it is isomorphic to $\Z/2$. It acts on $\wt{\E}$ by letting the non-trivial element act by the quadratic character corresponding to $\wt{F}$, which we may view as an element of $H$. We let $S_\rho^\prime \sub S_\rho$ denote the subgroup that fixes $x$; we have $S_\rho^\prime \neq S_\rho$ if and only if $\wt{F}$ is inert at $p$. We make one further technical assumption on $\pi$:

\begin{itemize}
\item All $U_p$-eigenvalues on $\pi_p^I$ have multiplicity one.
\end{itemize}
This is often referred to as ``regularity'' in the literature, and it is true for all $\pi$ conditional on the Tate conjecture and true for all $\pi$ corresponding to modular forms of weight $2$ unconditionally \cite{coleman-edixhoven}. It also holds in the CM case. We may now formulate our first main theorem, in the case $F =\Q$:

\begin{thmx}\label{thm A}
The following holds:
\begin{enumerate}
\item If $-\alpha$ is not a refinement of $\pi$, or if $-\alpha$ is a refinement of $\pi$ and $\pi$ has CM, then the natural map $\wh{\oo}_{\E,y} \to \wh{\oo}_{\wt{\E},x}$ of completed local rings induces an isomorphism $\wh{\oo}_{\E,y} \cong \wh{\oo}_{\wt{\E},x}^{S_\rho^\prime}$. In particular, $\E$ is smooth at $y$.

\item If $-\alpha$ is a refinement of $\pi$ and $\pi$ does not have CM, then there are precisely $2$ components locally going through $y$, and both are smooth. 
\end{enumerate}
\end{thmx}

When $S_\rho^\prime$ is non-trivial, i.e., when $\pi$ has CM by $\wt{F}$ with $p$ split in $\wt{F}$, then the action of $S_\rho^\prime$ on $\wh{\oo}_{\wt{\E},x}$ is non-trivial exactly when the refinement $\alpha$ is critical. In particular, the map $\wt{\E} \to \E$ is a local closed immersion at $x$ unless we are in the critically refined CM case, in which case it has ramification degree $2$. From this, we deduce our second main theorem, concerning classicality. Let us write $\mc{M}_y$ for the fibre of $\mc{M}$ at $y$; in the overconvergent setting it has a \emph{quotient} $\mc{M}_y \to \mc{M}_y^{cl}$ of classical forms.

\begin{thmx}\label{thm B}
If $\pi$ does not have CM or the refinement $\alpha$ is non-critical, then $\mc{M}_y \to \mc{M}_y^{cl}$ is an isomorphism. When $\pi$ has CM and $\alpha$ is critical, the kernel of $\mc{M}_y \to \mc{M}_y^{cl}$ is non-trivial, and the kernel may be identified with the direct sum of the $\gamma$-eigenspaces of $u_p$ in the $\SL_2(\A^\infty) \cap K^p(N)I$-fixed vectors of the non-automorphic members of $\Pi$.
\end{thmx}

The interpretation of the kernel in the critically refined CM case is not without subtleties; see Remarks \ref{interpretation} and \ref{interpretation quat}. We also discuss this more after Theorem \ref{thm E} below.

\medskip

Let us now discuss the proofs of Theorems \ref{thm A} and \ref{thm B} further. The key part is to prove Theorem \ref{thm A} and the remarks after it concerning the action of $S_\rho^\prime$; Theorem \ref{thm B} then follows from the identification $\mc{M} = (\pi_\ast \wt{\mc{M}})_H$ and the fact that $\wt{\mc{M}}$ is locally free at $x$. Looking at the map $\pi : \wt{\E} \to \E$, one may roughly divide the argument into two parts:

\begin{enumerate}
\item Show that the set-theoretic fibre $\pi^{-1}(y)$ consists of one or two $H$-orbits. The case of two orbits happens precisely when $-\alpha$ is a refinement of $\pi$ and $\pi$ does not have CM.

\smallskip

\item Show that the number of irreducible components of $\wh{\oo}_{\E,y}$ is equal to the number of orbits as in (1), and that if $\wh{\oo}_{\E,y} \twoheadrightarrow R$ is the irreducible component corresponding to an orbit with representative~$z$, then $R \to \wh{\oo}_{\wt{\E},z}^{S_\rho^\prime}$ is an isomorphism.
\end{enumerate}

For both parts, we make heavy use of Galois-theoretic methods. The proof of part (1) uses the family of Galois representations over $\wt{\E}$ and the fact that it is a refined, or trianguline, family to show that all elements in $\pi^{-1}(y)$ are twists by characters in $H$ of the points mentioned (here and elsewhere we view $(\Z/N)^\times$ as $\Gal(\Q(\zeta_N)/\Q)$). A subtle issue is that the points of $\pi^{-1}(y)$ are not \emph{a priori} classical; to show this requires the full force of global triangulation theory.

\medskip

The proof of (2) is the most novel part of our argument. To simplify the exposition, let us restrict to the critically refined CM case, which is in some sense the most interesting (at least a posteriori, in light of Theorem \ref{thm B}). Part (1) then establishes that $\wh{\oo}_{\E,y}$ is a domain. Bella\"iche \cite{bellaiche-crit} established a local $R=T$ theorem
\[
R_\rho^\alpha \cong \wh{\oo}_{\wt{\E},x}
\]
between $\wh{\oo}_{\wt{\E},x}$ and a deformation ring $R_\rho^\alpha$ parametrizing deformations of $\rho$ that are weakly refined with respect to $\alpha$, and showed that both rings are formally smooth. Using Bella\"iche's description of the tangent space of $R_\rho^\alpha$, we compute the action of $S_\rho$ on $\wh{\oo}_{\wt{\E},x}$ and show that it is non-trivial (this relies on $\alpha$ being critical). We now consider the universal unrestricted deformation ring $R_\rho$ of $\rho$, which surjects onto $R_\rho^\alpha$, and look at the maps
\[
R_\rho \twoheadrightarrow R_\rho^\alpha \cong \wh{\oo}_{\wt{\E},x}.
\]  
Taking invariants for the action of $S_\rho$, we get a surjection
\[
R_\rho^{S_\rho} \twoheadrightarrow \wh{\oo}_{\wt{\E},x}^{S_\rho}.
\]
Let $\rho^{univ}$ be the universal deformation of $\rho$. For $\ell \nmid Np$, $R_\rho^{S_\rho}$ contains the element $\tr(\Frob_\ell,\ad^0 \rho^{univ})$, and this maps to $\ell^{-1}(t_\ell +1) \in \wh{\oo}_{\E,y} \sub \wh{\oo}_{\wt{\E},x}^{S_\rho}$, where $t_\ell$ is the usual $\SL_2$-Hecke operator. To complete the proof of part (2), it suffices to show that the functions $\tr(\Frob_\ell,\ad^0 \rho^{univ})$ generate $R_\rho^{S_\rho}$ topologically. To prove this, we use the theory of $\PGL_2$-valued pseudorepresentations due to Lafforgue \cite{lafforgue}, crucially relying on the result of Emerson \cite{emerson} that the moduli space of pseudorepresentations is equal to the character variety in characteristic $0$, together with results of Procesi on the invariant theory of $\PGL_2$. However, one cannot apply Emerson's result directly, since $G_{\Q}$ is not discrete. To get around this issue requires some extra tricks, where we approximate a suitable quotient of $G_{\Q}$ by a discrete group. 

\medskip

This concludes our sketch of the case $F=\Q$. We will use notation that is similar to above also in the general case; we hope that our notation below is self-explanatory. The case of general totally real $F$ follows the same rough ideas as for $\Q$, but is considerably more technical, and the results are in many ways more interesting and illuminating. One problem is that overconvergent cohomology can now appear in multiple degrees, and as a result the eigenvarieties will not be equidimensional. While it is known from \cite{JN1} that there is a universal family of Galois representations\footnote{Or, more accurately, pseudorepresentations.} over (the nilreduction of) $\wt{\E}$, this family is not known to be trianguline away from the irreducible components of maximal dimension. For this reason, we make the following assumptions on $\pi$:

\begin{enumerate}

\item If $F \neq \Q$, then the reduction $\ol{\rho}_\pi$ is irreducible and generic, in the sense of \cite{caraiani-tamiozzo};

\item For every $v\mid p$, the $U_v$-eigenvalues on the Iwahori-fixed vectors $\pi_v^{I_v}$ have multiplicity $1$;

\item If $F \neq \Q$ and $\pi$ has complex multiplication by a quadratic CM extension $\wt{F}/F$, then $\wt{F} \not\subseteq F(\zeta_{p^\infty})$.

\end{enumerate}

The first assumption guarantees that the relevant overconvergent cohomology groups vanish outside the middle degree, using recent work of Caraiani--Tamiozzo \cite{caraiani-tamiozzo}. The second assumption is the same regular assumption that we had when $F=\Q$ (but, for general $F$, it does not always hold). The third assumption ensures the vanishing of the Bloch--Kato Selmer group $H^1_f(G_F,\ad \, \rho)$ by recent work of Newton--Thorne \cite{newton-thorne}. This vanishing is known in general when $F=\Q$, and is key to Bella\"iche's $R=T$-theorem and its recent generalization to general totally real $F$ by Bergdall--Hansen \cite{bh}. All three assumptions are rather mild; for a discussion of the closely related condition of decency see \cite[\S 1.6]{bh}. 

\medskip

Let us now state our generalizations of Theorems \ref{thm A} and \ref{thm B}. We let $\pi$ be a cuspidal cohomological automorphic representation of $\GL_2(\A_F)$ satisfying the above conditions, we let $\alpha = (\alpha_v)_{v\mid p}$ be a refinement, and we write $x=(\pi,\alpha)$ for the corresponding point on $\wt{\E}$. It induces a classical point $y=(\Pi,\gamma)$ on the $\SL_2$-eigenvariety. We let $r$ denote the number of places $v\mid p$ for which $-\alpha_v$ is also a refinement of $\pi_v$. If $\pi$ has CM by a quadratic extension $\wt{F}/F$, then we let $S_{in}$ denote the set of places $v \mid p$ which are inert in~$\wt{F}$.

\begin{thmx}\label{thm C} The following holds:

\begin{enumerate}
\item Assume that $\pi$ does not have CM. Then $\wh{\oo}_{\E,y}$ has $2^r$ irreducible components, all of which are smooth.

\item Assume that $\pi$ has CM and that $S_{in}\neq \emptyset$. Then $\wh{\oo}_{\E,y}$ has $2^{r-1}$ irreducible components. If, additionally, for all $v\mid p$ which split in $\wt{F}$, $\alpha_v$ is non-critical for all embeddings $\sigma$ inducing $v$ and $-\alpha_v$ is not a refinement of $\pi_v$, then all irreducible components of $\wh{\oo}_{\E,y}$ are smooth.

\item Assume that $\pi$ has CM and that $S_{in} = \emptyset$. Then $\wh{\oo}_{\E,y}$ has $2^r$ irreducible components and we can compute the precise geometry of the components. In particular, when $r=0$, $\wh{\oo}_{\E,y} = \wh{\oo}_{\wt{\E},x}^{S_\rho}$.
\end{enumerate}

\end{thmx}

In the cases excluded by Theorem \ref{thm C}, our methods cannot at present say if the components are smooth or not. Let us now assume that $\pi$ has CM and that $S_{in}=\emptyset$ in order to explicate Theorem \ref{thm C}(3). We focus on the case $r=0$ (for general $r$ each component will look like in Theorem \ref{thm D} below). Since $\wh{\oo}_{\E,y} = \wh{\oo}_{\wt{\E},x}^{S_\rho}$ and $\wh{\oo}_{\wt{\E},x}$ is smooth \cite{bh}, we see that $\wh{\oo}_{\E,y}$ is a particular type of quotient singularity, whose geometric properties can be summarized as follows:

\begin{thmx}\label{thm D}
Assume that $\pi$ has CM and that $r = 0$. Let $n$ be the number of embeddings $\sigma : F \to \ol{\Q}_p$ such that $\alpha$ is critical at $\sigma$. Then $\wh{\oo}_{\E,y}$ is normal and Cohen-Macaulay. Moreover:
\begin{enumerate}
\item If $n=0$ or $n=1$, then $\wh{\oo}_{\E,y}$ is regular.

\item If $n=2$, then $\wh{\oo}_{\E,y}$ is a complete intersection but not regular.

\item If $n$ is even and $\geq 4$, then $\wh{\oo}_{\E,y}$ is Gorenstein but not a complete intersection.

\item If $n$ is odd and $\geq 3$, then $\wh{\oo}_{\E,y}$ is $2$-Gorenstein but not Gorenstein.
\end{enumerate}

\end{thmx}

To the best of our knowledge, these are the first higher dimensional examples of points on eigenvarieties which are singular and the structure of the singularity has been worked out explicitly; we comment more on this in \S \ref{subsec: connections}.  Finally, Theorem \ref{thm C} leads to the following generalization of Theorem \ref{thm B}:

\begin{thmx}\label{thm E} Consider the surjection $\mc{M}_y \to \mc{M}_y^{cl}$ from the fibre of $\mc{M}$ at $y$ to its quotient of classical forms. Then the following holds:

\begin{enumerate}
\item Assume either that $\pi$ does not have CM, or that $\pi$ has CM and for all $v\mid p$ which split in $\wt{F}$, $\alpha_v$ is non-critical for all embeddings $\sigma$ inducing $v$ and $-\alpha_v$ is not a refinement of $\pi_v$. Then $\mc{M}_y \to \mc{M}_y^{cl}$ is an isomorphism.

\item Assume that $\pi$ has CM and $S_{in} = \emptyset$. Then the kernel may be identified with the direct sum of copies of the $\gamma = (\gamma_v)_{v \mid p}$-eigenspaces of $(u_v)_{v\mid p}$ in the $\SL_2(\A_F^\infty) \cap K^p(N)I$-fixed vectors of the non-automorphic members of $\Pi$.

\end{enumerate}

\end{thmx}

In case (2), we can compute the multiplicity precisely: When $r=0$ it is the quantity $n$ appearing in Theorem \ref{thm D}, and in general it is a sum of similar quantities over representatives of the $H$-orbits of $\pi^{-1}(y)$. In particular, the multiplicity \emph{need not be $1$}, showing that the seeming symmetry between the automorphic and non-automorphic members in Theorem \ref{thm B} is rather particular to that situation.

\medskip

Our arguments can also be applied in the setting of inner forms of $\GL_{2/F}$ and $\SL_{2/F}$, coming from quaternion algebras $B/F$. The case when $[F:\Q]$ is even and $B$ is ramified exactly at the infinite places is particularly interesting, as all cuspidal cohomological automorphic representations $\pi$ of $\GL_{2/F}$ can be transferred to $B^\times$ via the Jacquet--Langlands correspondence. The resulting eigenvarieties are equidimensional and reduced, so one may prove Theorems \ref{thm C}, \ref{thm D} and \ref{thm E} in this context without the assumption that $\ol{\rho}$ is generic and irreducible. Moreover, the interpretation of the non-classical forms as coming from the non-automorphic members of the $L$-packet is more transparent in this case; see Remark \ref{interpretation quat}.

\begin{remark} We end this discussion of the paper with two remarks.

\begin{enumerate}
\item It is well known that the $p$-adic \emph{generalized} eigenspace at critically refined CM points contains non-classical forms, cf.\ \cite{bellaiche-crit,bh}. Through the method of relating $\GL_2$ and $\SL_2$, these generalized eigenforms are related to our non-classical eigenforms. Indeed, one might describe our result as saying that certain of these generalized eigenforms are eigenforms for the $\SL_2$-Hecke algebra.

\medskip

\item We expect that the method can be generalized to other situations where a larger, isogenous group can be used to understand endoscopy (depending on the availability of the various technical results that underpin it).
\end{enumerate}
\end{remark}

\subsection{Previous results}\label{subsec: connections}

Let us briefly compare our results to some previous work on singularities and non-classical $p$-adic eigenforms, focusing on higher-dimensional situations and cohomological weights.

\subsubsection{Failure of classicality and $p$-adic endoscopy}
Beyond cases of small slope, works such as \cite{bellaiche-crit}, \cite{bh} and \cite{bhs2,bhs3} have established the link between classicality and the local geometry of the eigenvariety. Notably, \cite{bhs2,bhs3} give an affirmative answer to Question \ref{question} for certain higher rank unitary groups under technical conditions (which, in particular, force the points to be non-endoscopic), the most important being that the Weyl group element attached to the triangulation is a product of distinct simple reflections. When this latter condition fails, Hellmann--Hernandez--Schraen have recently announced that the eigenspace should contain non-classical forms, and that they can prove unconditional results in this direction for $U(3)$; see discussion after \cite[Conjecture 9.6.37]{egh}. We note that whereas the ``obstruction to classicality'' in our paper is global, coming from endoscopy, their obstruction is local at $p$.

\medskip

The phenomenon of endoscopy in the global $p$-adic Langlands program has, to the best of our knowledege, hardly been studied before beyond the works \cite{L1,L2} of the second author. We hope that our work will provide a good starting point for further study.

\subsubsection{Eigenvarieties at endoscopic points, and singularities}
In some situations, the local geometry of eigenvarieties at endoscopic points has been studied before. From the perspective of this paper, we highlight some results of Bella\"iche. In \cite{bellaiche} he studied the local geometry of eigenvarieties for a unitary group $U(3)$ at endoscopic points and showed that the eigenvariety is \emph{non-smooth but irreducible}  at certain classical points that arise from endoscopic automorphic representations; this was later generalized to $U(n)$ in \cite{bellaiche-isobaric}. Bella\"iche proves this by studying the family of pseudocharacters $T$ on the eigenvariety and showing that its schematic reducibility locus at his endoscopic point is exactly the point itself. From there, he proves an abstract result on such families which allows him to conclude non-smoothness, but does not allow him to say much more about the geometry. In particular, his method is very different from the one used in this paper, even though the endoscopic points he considers have some similarities with the ones we study here.

\subsection{Outline of the paper}

Let us briefly outline the contents of this paper. Section \ref{sec: aut and gal reps} starts by briefly recalling material on locally symmetric spaces and automorphic representations for $\GL_2(\A_F)$ and $\SL_2(\A_F)$ that we will need. The rest of the section is devoted to the theory of pseudorepresentations of Lafforgue and Galois deformation, primarily for $\PGL_2$-valued representations (or, equivalently, $\SO_3$-valued representations). Here we prove the key result, in the notation above, that $R_\rho^{S_\rho}$ is topologically generated by the elements $\tr(\Frob_v,\ad^0\rho)$ (despite the formulation here, this is naturally a result about $\PGL_2$-valued representations). 

\medskip

Section \ref{sec: eigenvarieties} recalls the theory of overconvergent cohomology for $\GL_{2/F}$ and $\SL_{2/F}$. We use the conventions of \cite{bh}, which we reference heavily. The key results proven are the existence and finer study of the $p$-adic functoriality map $\pi: \wt{\E} \to \E$ (which we only need, and therefore only prove, locally on the weight space) and the study of the fibre $\pi^{-1}(y)$. The proof of the former uses the general Chenevier-style interpolation theorems proved in \cite{JN2} together with some techniques that are special to our situation, and the proof of the latter uses the family of Galois representations on $\wt{\E}$ and the global triangulation theory of \cite{kpx}. 

\medskip

Section \ref{sec: local geometry} then proves Theorems \ref{thm C} and \ref{thm D}. The key here, besides the results of the previous sections, is the local $R=T$-theorem of Bergdall--Hansen \cite{bh} and a detailed calculation of the action of $S_\rho$ on the tangent spaces of the relevant deformation functors. Section \ref{sec: endoscopic forms} then proves Theorem \ref{thm E}, using the results on the local geometry and the formula for the coherent sheaf on $\E$ in terms of the coherent sheaf on $\wt{\E}$.

\medskip

Section \ref{sec: quat} then briefly discusses the analogues of Theorems \ref{thm C}, \ref{thm D} and \ref{thm E} for the inner forms of $\SL_{2/F}$ coming from quaternion algebras ramified exactly at the infinite places, and the paper concludes with an appendix discussing some properties of CM Hilbert modular forms, the quotient singularities appearing in Theorem \ref{thm D}, and some local $p$-adic geometry.

\subsection{Notation}\label{setup}

Throughout this paper, $p$ will denote a fixed prime number and $F$ will be a totally real number field of degree $d:= [F:\Q]$. For any place $v$ of $F$, we write $F_v$ for the completion of $F$ at $v$. If $w$ is a place of $\Q$, we write $\Sigma_w$ for the set of embeddings $F \hookrightarrow \ol{\Q}_w$. If $v$ is a place of $F$, lying over a place $w$ of $\Q$, we write $\Sigma_v$ for the set of embeddings of $F_v \hookrightarrow \ol{\Q}_w$. Moreover, we set 
\[
F_w := \prod_{v \mid w} F_v, \,\,\,\,\,\, \A_F^w := \prod^\prime_{v \nmid w} F_v.
\]
Various common versions of this notation, such as $\A_F^{p\infty} := \bigotimes^\prime_{v \nmid p\infty} F_v$, will also occur. Similarly, we let $\oo_v$ denote the ring of integers of $F_v$, and write 
\[
\oo_w := \prod_{v \mid w} \oo_v, \,\,\,\,\,\, \oo_F^w := \prod_{v \nmid w} \oo_v,
\]
etc. The residue field of $F_v$ will be denoted by $\F_v$, and the cardinality of $\F_v$ will be denoted by $q_v$. If $v$ lies over a place $w$ of $\Q$, we let $d_v := [F_v:\Q_w]$ denote the local degree. For simplicity, we fix an isomorphism $\C \cong \ol{\Q}_p$ throughout the paper, and use the induced identification $\Sigma_\infty \cong \Sigma_p$. 

\medskip

To simplify notation, we will fix uniformizers $\varpi_v$ of $F_v$ for all finite places $v$ of $F$. Nothing that we do will depend on these choices in a crucial way. We also set $\varpi_p = (\varpi_v)_{v\mid p}$, which is an element of $\oo_{p} \cong \prod_{v\mid p}\oo_{v}$.

\medskip

If $K$ is any field, we let $G_K$ denote the absolute Galois group of $K$ (with respect to some choice of separable closure of $K$). If $K$ is a number field and $S$ is a finite set of places containing the infinite places, we let $G_{K,S}$ denote the Galois group of the maximal extension of $K$ that is unramified outside $S$. 

\medskip

If $H(-,-)$ denotes a cohomology theory (typically singular cohomology in this paper), we let
\[
H^\ast(-,-) = \bigoplus_i H^i(-,-)
\]
denote the total cohomology, viewed as a \emph{graded} module. Any map, isomorphism, etc. of total cohomology groups is assumed to respect the grading.

\medskip

In this paper, we will do $p$-adic analytic geometry in terms of adic spaces, and we will use standard notation and terminology from the theory of adic spaces without further mention. We refer to \cite{huber-book,berkeley} for a thorough discussion. In particular, a rigid analytic variety (or rigid space) over a non-archimedean field $K$ will refer to an adic space over $\Spa(K,K^\circ)$ which is locally isomorphic to $\Spa(A,A^\circ)$ for some topologically finitely generated $K$-algebra $A$ (i.e. an affinoid $K$-algebra in the sense of Tate's rigid analytic geometry). For simplicity, we will write $\Spa(A)$ for $\Spa(A,A^\circ)$ (for any Huber ring), as we did in the introduction.

\subsection*{Acknowledgments}

This project has benefited greatly from ideas developed in a collaboration between C.J. and James Newton on $p$-adic local Langlands for $\SL_2(\Qp)$ and local-global compatibility for completed cohomology of $\SL_2(\Q)$, and we wish to thank him for this and for other useful conversations. We wish to thank John Bergdall for his comments on an earlier draft of this paper, in particular for spotting some mistakes and generously sharing his expertise on $(\varphi,\Gamma)$-modules over the Robba ring. We also wish to thank Bergdall and David Hansen for answering our questions on \cite{bh}, and more generally for stimulating conversations on the topic of this paper. Finally, we wish to sincerely thank an anonymous referee for their careful reading of and many comments on our paper, which greatly improved the paper and helped us catch some oversights in an earlier version.   

\medskip

During part of this project, C.J. has been supported by Vetenskapsr\r{a}det Grant 2020-05016, \textit{Geometric structures in the $p$-adic Langlands program}. J.L. acknowledges support from the Deutsche Forschungsgemeinschaft (DFG, German Research Foundation) through TRR 326 \textit{Geometry and Arithmetic of Uniformized Structures}, project number 444845124.

\section{Automorphic and Galois representations}\label{sec: aut and gal reps}

In this section we give some background on the locally symmetric spaces and automorphic representations. We then discuss Lafforgue's pseudorepresentations and Galois deformation theory, proving results that later will be used in a crucial way to identify $\SL_2$-Hecke algebras inside $\GL_2$-Hecke algebras with certain rings of invariants. 

\subsection{Locally symmetric spaces and local systems}\label{subsec: symmetric spaces}

We recall some locally symmetric spaces for $\GL_2$ and $\SL_2$ over $F$. Define
\[
D_\infty = \mf{h}^{\Sigma_\infty},
\]
where $\mf{h} \sub \C$ is the usual upper halfplane. It is acted on by $\GL_2(F_\infty)^\circ$ by M\"obius maps (where, if $G$ is a real Lie group, $G^\circ$ denotes the identity component), and hence by $\GL_2(F)^\circ := \GL_2(F) \cap \GL_2(F_\infty)^\circ$, the matrices in $\GL_2(F)$ with totally positive determinant. For any compact open subgroup $K \sub \GL_2(\A_F^\infty)$, we define
\[
Y_K := \GL_2(F) \backslash \GL_2(\A_F) / K K_\infty \cong \GL_2(F)^\circ \backslash D_\infty \times \GL_2(\A_F^\infty) /K,
\]
where $K_\infty = (Z(\R)\cdot\mathrm{SO}_2(\R))^{\Sigma_\infty}  \sub \GL_2(\R)^{\Sigma_\infty} = \GL_2(F_\infty)$ is a maximal compact connected subgroup modulo centre. The determinant map defines a continuous surjection
\[
Y_K \to Cl_{\det(K)}^+
\]
with connected fibres. Here, if $U \sub (\A_F^\infty)^\times$ is a compact open subgroup
\[
Cl_U^+ := F^\times \backslash \A_F^\times /( U\cdot (F_\infty^\times)^\circ ) = (F^\times)^\circ \backslash (\A_F^\infty)^\times / U
\]
denotes the ray class group of modulus $U\cdot(F_\infty^\times)^\circ$, which is a finite abelian group. In particular, we may identify $Cl_{\det(K)}^+$ with the set of connected components in $Y_K$. For $\SL_2$, we set
\[
Y_K^1 := \SL_2(F) \backslash D_\infty \times \SL_2(\A_F^\infty) /K
\]
for any compact open subgroup $K\sub \SL_2(\A_F^\infty)$. When $K \sub \GL_2(\A_F^\infty)$ is a compact open subgroup, the natural inclusion $\SL_2 \sub \GL_2$ gives a map $Y_{K\cap \SL_2(\A_F^\infty)}^1 \to Y_K$, which is a finite covering map of the component of $Y_K$ mapping to $1\in Cl_{\det(K)}^+$. This covering is explicitly given as
\[
(\SL_2(F) \cap K) \backslash D_\infty \to (\GL_2(F)^\circ \cap K) \backslash D_\infty.
\]
Let $\n \sub \oo_F$ be a non-zero ideal coprime to $p$. In this paper, we will need the level subgroup $K^p(\n) \sub \GL_2(\A_F^{p\infty})$, defined as those matrices in $\GL_2(\wh{\oo}_F^p)$ which are congruent to the identity modulo $\n$. We let $I \sub \GL_2(\oo_{F,p})$ be the Iwahori subgroup of matrices which are upper triangular modulo $\vp_p$. We will need the following observation.

\begin{lemma}\label{shrinking}
Let $K^p \sub \GL_2(\A_F^{p\infty})$ be a compact open subgroup. Then there exists a compact open subgroup $K_0^p \sub K^p$ such that the dashed arrow in the diagram 
\[
\begin{tikzcd}
(\SL_2(F) \cap K^p_0 I) \backslash D_\infty \ar{r} \ar{d} & (\GL_2(F)^\circ \cap K^p_0 I) \backslash D_\infty \ar{d} \arrow[ld,dashed] \\
(\SL_2(F) \cap K^p I) \backslash D_\infty \ar{r} & (\GL_2(F)^\circ \cap K^p I) \backslash D_\infty
\end{tikzcd}
\]
exists.
\end{lemma}

\begin{proof}
Without loss of generality we may assume that $K^p = K^p(\n)$ for some ideal $\n \sub \oo_F$ prime to $p$, and we may look for a $K^p_0$ of the form $K^p(\m)$ for some $\m \sub \n$ (also prime to $p$). If $\mf{a} \sub \oo_F$ is an ideal, we write $U(\mf{a})_+$ for the subgroup of $\oo_F^+$ of totally positive units congruent to $1$ modulo $\mf{a}$. We need to find $\m$ as in the statement such that the image of $\GL_2(F)^\circ \cap K^p(\m)I$ in $\PGL_2(F)$ is contained in the image of $\SL_2(F) \cap K^p(\n)I$, which is equivalent to 
\[
\GL_2(F)^\circ \cap K^p(\m)I \sub (\SL_2(F) \cap K^p(\n)I)\cdot Z(F).
\]
By \cite[Th\'eor\`eme 1]{chevalley}, there exists an ideal $\m$, prime to $p$, such that $U(\m)_+ \sub U(\n)_+^2$. Choose such an $\m$ which is also contained in $\n$. Then, if $g\in \GL_2(F)^\circ \cap K^p(\m)I$, its determinant is in $U(\m)_+$ and hence equal to $z^2$ for some $z\in U(\n)_+$. In particular, we see that
\[
g =  \left( g \cdot \begin{pmatrix} z^{-1} & 0 \\ 0 & z^{-1} \end{pmatrix} \right) \cdot \begin{pmatrix} z & 0 \\ 0 & z \end{pmatrix} \in (\SL_2(F) \cap K^p(\n)I)\cdot Z(F)
\]
as desired. This finishes the proof.
\end{proof}

Let us now move on to discuss local systems on $Y_K$ and $Y_K^1$, following \cite[\S 2.2]{bh} where it is described for general connected reductive groups. The procedure is the same in both cases, so we will just describe it in the case of $Y_K$. Consider the natural map
\[
D_\infty \times \GL_2(\A^\infty_F) \to Y_K,
\]
where $\GL_2(\A_F^\infty)$ carries the \emph{discrete} topology. Let $Z(K) = Z(F) \cap K$ (intersection taken inside $\GL_2(\A_F^\infty)$). Then the map above is a torsor for $(\GL_2(F)^\circ \times K)/Z(K)$ (viewed as a \emph{discrete} group), where $Z(K)$ is diagonally embedded in the product. Here, we let $(\gamma,k) \in \GL_2(F)^\circ \times K$ act on $(\tau,g)\in D_\infty \times \GL_2(\A_F^\infty)$ by $(\gamma,k) \cdot (\tau,g) = (\gamma\tau, \gamma gk^{-1})$. Thus, any left $(\GL_2(F)^\circ \times K)/Z(K)$-module gives a local system on $Y_K$ by descent. There is an exact sequence
\[
1 \to \GL_2(F)^\circ \to (\GL_2(F)^\circ \times K)/Z(K) \to K/Z(K) \to 1.
\]
In this paper, we will only (explicitly) need $(\GL_2(F)^\circ \times K)/Z(K)$-modules which are trivial on $\GL_2(F)^\circ$, i.e. local systems descended from the cover
\[
\mf{Y} := \GL_2(F)^\circ \backslash D_\infty \times \GL_2(\A^\infty_F) \to Y_K,
\]
which has covering group $K/Z(K)$ (naturally acting from the right). Any right $K/Z(K)$-module $N$ gives a local system on $Y_K$, which we will also denote by $N$. Let $C_\bu (\mf{Y})$ be the complex of singular chains on $\mf{Y}$ and let $N$ be a right $K$-module, trivial on $Z(K)$. Then the cohomology $H^\ast(Y_K,N)$ can be computed as the cohomology of the complex $\Hom_K(C_\bu(\mf{Y}),N)$. We will also need to consider the action of Hecke operators. Let $\Delta \sub \GL_2(\A_F^\infty)$ be a monoid containing a compact open subgroup $K\sub \GL_2(\A_F^\infty)$. Then, if $N$ is a \emph{left} $\Delta$-module with trivial $Z(K)$-action, we may define the action of Hecke operators $[K\delta K]$, $\delta \in \Delta$ on $H^\ast(Y_K,N)$, as in \cite[\S 2.2]{bh}. Note that $N$ is made into a right $K$-module (and hence a local system on $Y_K$) by inverting the left $K$-action. In general, the Hecke algebra of compactly supported bi-$K$-invariant locally constant functions on $\Delta$ (under convolution) will be denoted by $\mbf{T}(\Delta,K)$, whenever $K\sub G$ is a compact open subgroup of a locally profinite group $G$ and $K \sub \Delta \sub G$ is a monoid. 

\medskip

We finish by recording a consequence of Lemma \ref{shrinking} that will be important to us later on, as well as some simplifying notation. In this paper, we will only ever work with the level subgroup $I$ at $p$. For any compact open subgroup $K^p \sub \GL_2(\A_F^{p\infty})$, we will set $K = K^p I$, and we set
\[
Y_{K}^1 := Y^1_{K^p I \cap \SL_2(\A_F^\infty)} \,\,\, \text{and} \,\,\, Cl_{K}^+ := Cl_{\det(K)}^+.
\]
We also let $H_{K}$ denote the dual group of the finite abelian group $Cl_{K}^+$, i.e. the group of characters of $Cl_{K}^+$ taking values in an algebraically closed field of characteristic $0$. Finally, if $\n \sub \oo_F$ is a non-zero ideal which is prime to $p$, we simplify the notation even further. We set 
\[
Y_{\n} = Y_{K^p(\n)I}, \,\,\,\, Y^1_{\n} = Y^1_{K^p(\n)I}, \,\,\,\, Cl_{\n}^+ = Cl_{K^p(\n)I}^+ \,\,\, \text{and} \,\,\, H_{\n} = H_{K^p(\n)I}.
\]
Now let $E$ be an algebraically closed field of characteristic $0$. Via the inclusion $H_{K} \sub H^0(Y_{K},E)$, $H_{K}$ acts on $H^\ast(Y_{K},N)$ for any right $E[K ]$-module $N$; we review this in more detail in \S \ref{subsec: ocGL2}.

\begin{proposition}\label{limit}
With notation as in the previous paragraph and $K = K^p I$, the natural map
\[
\varinjlim_{K^p} H^\ast(Y_{K},N)_{H_{K}} \to \varinjlim_{K^p} H^\ast(Y_{K}^1,N)
\]
is an isomorphism. Here $K^p$ runs through the compact open subgroups of $\GL_2(\A_F^{p\infty})$; the intersections $K^p \cap \SL_2(\A_F^{p\infty})$ are then cofinal among the compact open subgroups of $\SL_2(\A_F^{p\infty})$.
\end{proposition}

\begin{proof}
We start by describing the map. The map $Y_{K}^1 \to Y_{K}$ induces a map $H^\ast(Y_{K},N) \to H^\ast(Y_{K}^1,N)$. The coinvariants $H^\ast(Y_{K},N)_{H_{K}}$ are canonically isomorphic to the cohomology of $N$ on the component of $Y_{K}$ corresponding to $1 \in Cl_{K}^+$, so the map $H^\ast(Y_{K},N) \to H^\ast(Y_{K}^1,N)$ factors through the $H_{K}$-coinvariants, and one sees easily that these maps define a direct system as $K^p$ varies. With this preparation, the proposition is then a direct consequence of Lemma \ref{shrinking}.
\end{proof}

\begin{remark}
When $F=\Q$, the maps $H^\ast(Y_{K},N)_{H_{K}} \to H^\ast(Y_{K^p}^1,N)$ are isomorphisms even before taking the direct limit, since $Y_{K}^1$ is equal to the component of $Y_{K}$ corresponding to $1 \in Cl_{K}^+$.
\end{remark}

\subsection{Automorphic representations}\label{automorphic reps}

In this subsection we indicate our conventions on automorphic forms and representations. We will also discuss the aspects of (classical) $L$-packets for $\SL_2$ that we will need, as well as functoriality from $\GL_2$ to $\SL_2$.

\medskip

We start with $\GL_2$, where we will mostly try to follow \cite{bh} closely. Since we will exclusively deal with cohomological automorphic representations, we begin by recalling the definition of a cohomological weight from \cite[Definition 2.4.1]{bh}:

\begin{definition}
A cohomological weight (for $\GL_{2/F}$) is a pair $k = (k_1,k_2)$ of characters $F^\times \to \C^\times$ of the form
\[
k_i(x) = \prod_{\sigma \in \Sigma_\infty} \sigma(x)^{e_i(\sigma)}
\]
for $e_i(\sigma)\in \Z$ such that $e_1(\sigma) \geq e_2(\sigma)$ for all $\sigma \in \Sigma_\infty$ and $\omega_k := k_1 k_2$ is trivial on a finite index subgroup of $\oo_F^\times$.
\end{definition}

If $k$ is a cohomological weight, then $e_1(\sigma) + e_2(\sigma)$ is independent of $\sigma$ and will be denoted by $w$. Setting $k_\sigma = e_1(\sigma) - e_2(\sigma)$, we see that a cohomological weight can be described in terms of a $d+1$-tuple $((k_\sigma)_\sigma,w) \in \Z_{\geq 0}^{\Sigma_\infty} \times \Z$ with $k_\sigma \equiv w$ modulo $2$ for all $\sigma$. Using our fixed isomorphism $\C \cong \ol{\Q}_p$ and the corresponding identification $\Sigma_\infty \cong \Sigma_p$ we can view a cohomological weight as a pair of $\ol{\Q}_p$-valued characters; we will do so without further comment. For each $v \mid p$, we define
\[
k_{i,v}(x) = \prod_{\sigma \in \Sigma_v} \sigma(x)^{e_i(\sigma)} \in \ol{\Q}_p,
\] 
where we use the identification $\Sigma_\infty \cong \Sigma_p$ mentioned above.

\medskip

Given integers $k_\sigma \geq 0$ and $w$ with $k_\sigma \equiv w \mod 2$, there is a (unique) discrete series representation $D_{k_\sigma +2, w}$ of $\GL_2(\R)$ of weight $k_\sigma +2$ and central character $x \mapsto x^{-w}$. If $k$ is a cohomological weight then, following \cite[\S 3.1]{bh}, we set $D_k := \bigotimes_{\sigma \in \Sigma_\infty} D_{k_\sigma +2, w}$, viewed as a representation of $\GL_2(F_\infty)$. We then say that a cuspidal automorphic representation $\pi$ is \emph{cohomological of weight $k$} if $\pi_\infty \cong D_k$. Let us also recall how cohomological representations show up in cohomology. If $n\geq 0$ and $w$ are integers with $n \equiv w \mod 2$ and $R$ is a ring, we define $\mathscr{L}_{n,w}(R)$ to be the $R$-module of polynomials with coefficients in $R$, of degree at most $n$, and give it the $\GL_2(R)$-action
\[
\left( \begin{pmatrix} a & b \\ c & d  \end{pmatrix} \cdot P \right) (X) = \det \begin{pmatrix} a & b \\ c & d  \end{pmatrix}^{\frac{w-n}{2}} (a+cX)^n P\left( \frac{b+dX}{a+cX} \right).
\]
If $k$ is a cohomological weight and $E \sub \ol{\Q}_p$ is an extension of $\Qp$ containing the image of all embeddings $F \hookrightarrow \ol{\Q}_p$, we set
\[
\mathscr{L}_k (E) = \bigotimes_{\sigma \in \Sigma_p} \mathscr{L}_{k_\sigma, w}(E).
\] 
It carries a left action of $\prod_{\sigma\in \Sigma_p} \GL_2(E)$. Using the composite homomorphism
\[
\GL_2(\A_F^\infty) \to \GL_2(F_p ) \to \prod_{\sigma\in \Sigma_p} \GL_2(E),
\]
where the first map is the projection and the second is the product of the maps $\sigma: \GL_2(F_p ) \to  \GL_2(E)$, we get a left action of $\GL_2(\A_F^\infty)$ which is trivial on $Z(K)$ for any sufficient small compact open $K \sub \GL_2(\A_F^\infty)$. In particular, $\mathscr{L}_k (E)$ defines a local system on any $Y_K$. We have that 
$H^\ast(Y_{K},\mathscr{L}_k(E)) \cong H^\ast(Y_{K^\prime},\mathscr{L}_k(E))^{K}$
 via the transition map whenever $K^\prime \sub K$ is normal. The direct limit 
 $$H^\ast(Y,\mathscr{L}_k(E)) := \varinjlim_K H^\ast(Y_{K},\mathscr{L}_k(E))$$
  carries commuting actions of $\GL_2(\A_F^\infty)$ and $\pi_0 := \pi_0(\GL_2(F_\infty))$. If $\pi$ is a cuspidal automorphic representation of $\GL_2(\A_F)$, then the $\pi^\infty$-isotypic part $ H^\ast(Y,\mathscr{L}_k(\ol{\Q}_p))[\pi^\infty]$ of $ H^\ast(Y,\mathscr{L}_k(\ol{\Q}_p))$ is non-zero if and only if $\pi$ is cohomological of weight $k$, in which case
\[
 H^d(Y,\mathscr{L}_k(\ol{\Q}_p))[\pi^\infty] \cong \pi^\infty \otimes R(\pi_0)
\]
as a $\GL_2(\A_F^\infty) \times \pi_0$-representation, where $R(\pi_0)$ is the regular representation of $\pi_0$, and 
\[
H^i(Y,\mathscr{L}_k(\ol{\Q}_p))[\pi^\infty] =0, \text{ for } i \neq d,
\] 
(see \cite[p. 61, p. 74]{harder}). Here we have used the isomorphism $\C \cong \ol{\Q}_p$ to view $\pi^\infty$ as a smooth $\GL_2(\A_F^\infty)$-representation on a $\ol{\Q}_p$-vector space; we will continue to do such things without further discussion. Since
\[
H^\ast(Y,\mathscr{L}_k(\ol{\Q}_p)) = H^\ast(Y,\mathscr{L}_k(E)) \otimes_E \ol{\Q}_p,
\]
$H^\ast(Y,\mathscr{L}_k(E))$ is also concentrated in degree $d$, for any $E$ as above.

\medskip

We will need the notion of a ($p$-)refinement of an automorphic representation. Let $\pi$ be a cuspidal cohomological automorphic representation. For any $v \mid p$ we denote by $I_v$ the Iwahori subgroup of $\GL_2(F_v)$ defined as the preimage of the upper triangular matrices under the reduction map $\GL_{2}(\mc{O}_{F_v}) \ra \GL_2(\mb{F}_v)$. Assume from now on that $\pi_v^{I_v} \neq 0$ for all $v \mid p$, as such $\pi$ are the only ones that we will consider in this paper. 

\begin{definition}
Let $\pi$ be as above. A refinement $\alpha = (\alpha_v)_{v\mid p}$ of $\pi$ is a choice of eigenvalue $\alpha_v$ of $U_v = [ I_v \left( \begin{smallmatrix} \varpi_v & 0 \\ 0 & 1 \end{smallmatrix} \right) I_v ]$ acting on $\pi_v^{I_v}$, for each $v\mid p$. 
\end{definition}

\begin{definition}\label{refrep}
We define a refined automorphic representation to be a pair $(\pi,\alpha)$, where $\pi$ is a cuspidal cohomological automorphic representation and $\alpha$ is a refinement of $\pi$.
\end{definition}

We recall that when $\pi_v$ is an unramified principal series representation, the eigenvalues of $U_v$ on $\pi_v^{I_v}$ are precisely the roots of the $v$-Hecke polynomial $X^2 - a_v(\pi) X + \omega_{\pi} (\varpi_v) q_v$. Here $\omega_\pi$ is the central character of $\pi$ and $a_v(\pi)$ is the eigenvalue of $T_v = [ K_v \left( \begin{smallmatrix} \varpi_v & 0 \\ 0 & 1 \end{smallmatrix} \right) K_v ]$ acting on $\pi_v^{K_v}$, where $K_v = \GL_2(\oo_{F_v})$. The notion of refinement is defined in \cite[Definition 3.4.2]{bh}, and it is equivalent to the definition given above by \cite[Proposition 3.4.4]{bh}.

\subsubsection{The Galois representation associated to a cuspidal cohomological automorphic representation}
Let $\pi$ be a cuspidal cohomological automorphic representation of $\GL_2(\A_F)$ of weight $k$ as above. Then $\pi$ has an associated $p$-adic Galois representation
\begin{equation}\label{eq:classrep}
\rho_{\pi}:G_F\rightarrow \GL_2(\overline{\Q}_p),
\end{equation}
depending on our fixed choice of isomorphism $\C \cong \ol{\Q}_p$ (which we suppress from the notation). We recall some of the main properties of $\rho_\pi$. Let $\rho_{\pi,v}$ denote its restriction to a place $v$ of $F$. Then (see e.g. \cite{bh}, Theorem 6.5.1 and Remark 6.5.2) $\rho_\pi$ is irreducible and for any $v\mid p$, $\rho_{\pi,v}$ satisfies the following properties:

\begin{enumerate}
\item If $\sigma \in \Sigma_v$, then the set $HT_{\sigma}(\rho_{\pi,v})$ of $\sigma$-Hodge--Tate weights of $\rho_{\pi,v}$ is equal to $\left\{\frac{w-k_{\sigma}}{2}, \frac{w+k_{\sigma}}{2}+1 \right\}$.

\item If $\pi_v$ is an unramified principal series representation then $\rho_{\pi,v}$ is crystalline and the characteristic polynomial of $\varphi^{f_v}$ acting on $D_{crys}(\rho_{\pi,v})$  is given by the $v$-th Hecke polynomial $X^2-a_v(\pi)X+ \omega_{\pi}(\varpi_v)q_v$.

\item If $\pi_v$ is an unramified twist of a special representation then $\rho_{\pi,v}$ is semistable (but not crystalline).
\end{enumerate}

For all of this, we use the following conventions: We normalize local class field theory by sending uniformizers to geometric Frobenius elements, and we normalize Hodge--Tate weights so that all the Hodge--Tate weights of the cyclotomic character are $-1$.

\subsubsection{L-packets and classical functoriality from $\GL_2$ to $\SL_2$}\label{subsec:packets}

Let $v$ be a finite place of $F$. We recall that given an irreducible smooth representation $\widetilde{\pi}_v$ of $\GL_2(F_v)$, its restriction to $\SL_2(F_v)$ breaks up into finitely many irreducible smooth representations, each occurring with multiplicity one (\cite[Lemmas 2.4 and 2.6]{labesse-langlands}).  A local $L$-packet of $\SL_2(F_v)$ is defined as a finite set $\{\pi_{v,i}: i = 1, \dots, n\}$ of irreducible smooth representations $\pi_{v,i}$ of $\SL_2(F_v)$, such that there exists an irreducible admissible representation $\widetilde{\pi}_v$ of $\GL_2(F_v)$ with 
\[
 \widetilde{\pi}_v|_{\SL_2(F_v)} \cong \bigoplus_{i=1}^n \pi_{v,i} .
\]
Starting from an irreducible smooth representation $\widetilde{\pi}_v$ of $\GL_2(F_v)$ we denote by $\Pi(\widetilde{\pi}_v)$ the $L$-packet defined by $\widetilde{\pi}_v$. Any irreducible smooth representation of $\SL_2(F_v)$ occurs in the restriction of an irreducible smooth representation of $\GL_2(F_v)$ and therefore belongs to an $L$-packet, see \cite[Lemma 2.5]{labesse-langlands}. Furthermore, suppose an irreducible smooth representation $\pi_v$ of $\SL_2(F_v)$ occurs in the restriction of two representations $\widetilde{\pi}_v$ and $\widetilde{\pi}'_v$. Then there exists a character $\chi: F_v^{\times} \rightarrow \C^\times$, trivial on $\det(Z(F_v))=(F_v^{\times})^2$, such that $\widetilde{\pi}'_v \cong \widetilde{\pi}_v\otimes \chi$, and moreover $\Pi(\widetilde{\pi}_v)= \Pi(\widetilde{\pi}'_v)$ (see \cite[Lemma 2.4]{gelbart-knapp}). In particular two $L$-packets either agree or are disjoint, hence forming $L$-packets gives a partition of the set of equivalence classes of irreducible smooth representations into finite sets. $L$-packets of $\SL_2(F_v)$ have size $1,2$ or $4$ \cite[p.15]{labesse-langlands}. Finally, let us recall that if $\widetilde{\pi}_v$ is an unramified representation of $\GL_2(F_v)$, there exists a unique element $\pi^0_v$ in $\Pi(\widetilde{\pi}_v)$ which has a non-zero $\SL_2(\mathcal{O}_{v})$-fixed vector. More generally, one has the following fact, which we record as a lemma for ease of referencing.

\begin{lemma}\label{special member}
Let $\wt{\pi}_v$ be an irreducible admissible representation of $\GL_2(F_v)$. Then there is a member $\pi_v^0 \in \Pi(\wt{\pi}_v)$ with $(\pi_v^0)^{I_{0,v}}\neq 0$ if and only if $(\wt{\pi}_v \otimes (\chi \circ \det))^{I_v} \neq 0$ for some smooth character $\chi$ of $F^\times_v$. 
\end{lemma}

In fact, if $(\wt{\pi}_v \otimes (\chi \circ \det))^{I_v} \neq 0$ for some smooth character $\chi$, then the packet $\Pi(\wt{\pi}_v)$ is either a singleton or has size $2$ (if a packet has size $4$, then any representation $\widetilde{\pi}_v$ defining it is supercuspidal (cf. \cite[Section 2]{labesse-langlands},\cite[Section 3.2]{lansky-raghuram})). Size $2$ only occurs if, up to twist, $\widetilde{\pi}_v \cong \Ind_{B}^{\GL_2(F)}(\psi \otimes \mbf{1})$ is a principal series representation, with $\mbf{1}$  the trivial character and $\psi$ the unramified quadratic character of $F_v^\times$. This follows from the description of principal series $L$-packets in Section 3.2 of \cite{lansky-raghuram} and \cite[Prop. 3.2.8]{lansky-raghuram}, as if $(\wt{\pi}_v \otimes (\chi \circ \det))^{I_v} \neq 0$, then $\Pi(\wt{\pi}_v)$ is necessarily an unramified principal series $L$-packet in the language of loc.\ cit.

\medskip

At the archimedean places one defines $L$-packets in the same way.  In this paper we only consider cuspidal cohomological automorphic representations, so for us it suffices to recall the $L$-packet associated to the discrete series representations $D_{k+2,w}$ of $\GL_2(\R)$ discussed above (in this discussion $k$ is a non-negative integer). Upon restriction from $\GL_2(\R)$ to $\SL_2(\R)$, a discrete series representation $D_{k +2,w}$ decomposes into two irreducible representations $ D^+_{k +2}$ and $D^-_{k +2}$ and these form a so-called discrete series $L$-packet. As the notation suggests, the $L$-packet $\{ D_{k+2}^+, D_{k+2}^- \}$ depends on $k$ but not on $w$.

\medskip

The natural projection defines a map of $L$-groups 
\[
^L\GL_2= \GL_2(\C) \longrightarrow \PGL_2(\C)={^L\SL_2}.
\]
Associated to this map there is a Langlands transfer, which we recall next. Again we only describe the part that is relevant for us.
Let $\widetilde{\pi}=\otimes \widetilde{\pi}_v$ be a cuspidal cohomological automorphic representation of $\GL_2(\A_F)$. The global $L$-packet of $\SL_2(\A_F)$ associated to $\widetilde{\pi}$ is defined as
\[
\Pi(\widetilde{\pi}):= \left\{\bigotimes_v^\prime \pi_v \mid  \pi_v \in \Pi(\widetilde{\pi}_v), \pi_v=\pi^0_v \,\, \text{ for almost all } v \right\}.
\]
We refer to the $L$-packet $\Pi(\widetilde{\pi})$ as the transfer of $\widetilde{\pi}$. We briefly recall the relevant multiplicity results from \cite{labesse-langlands} that complete the description of this transfer. In the following we write $\Pi$ for a global $L$-packet when we do not wish to specify a representation $\widetilde{\pi}$ that gives rise to it. For a finite set $S$ of places of $F$, we let $\Pi_S:= \{\otimes_{v\in S} \pi_v \ |\  \pi_v \in \Pi(\widetilde{\pi}_v)\}$ and $\Pi^S:= \{\otimes_{v\notin S}^\prime \pi_v \ |\  \pi_v \in \Pi(\widetilde{\pi}_v), \pi_v=\pi^0_v \, \text{ for almost all } v \}$. For an irreducible admissible representation $\pi$ of $\SL_2(\A_F)$, we let $m(\pi)$ be the multiplicity of $\pi$ in the cuspidal automorphic spectrum.

\medskip
 
One distinguishes between two kinds of global $L$-packets, namely stable $L$-packets and endoscopic (or unstable) $L$-packets.
Stable $L$-packets have the property that the multiplicity is positive and constant (\cite[Lemma 6.1]{labesse-langlands}). Combined with the multiplicity one result of Ramakrishnan (\cite[Theorem 4.1.1]{ramakrishnan}), we know that for a stable $L$-packet $\Pi$, we have $m(\pi)=1$ for all $\pi \in \Pi$.  

\medskip

If $\Pi$ is an endoscopic $L$-packet then the multiplicity varies. More precisely, $m(\pi)\in \{0,1\}$ for $\pi \in \Pi$. If $\widetilde{\pi}$ is a cuspidal cohomological automorphic representation of $\GL_2(\A_F)$, then $\widetilde{\pi}$ gives rise to an endoscopic $L$-packet if and only if $\widetilde{\pi}$ has complex multiplication by a totally imaginary quadratic extension $\widetilde{F}$ of $F$. In this case $\widetilde{\pi}$ arises from an algebraic Gr\"o{\ss}encharacter $\widetilde{\theta}:\widetilde{F}^{\times}\backslash \A^{\times}_{\widetilde{F}} \rightarrow \C^*$ as described in \cite[Section 12]{jacquet-langlands}, and the $L$-packet $\Pi(\widetilde{\pi})$ only depends on the restriction $\theta$ of $\widetilde{\theta}$ to the subgroup of $\A^{\times}_{\widetilde{F}}$ of elements of norm $1$. We also use the notation $\Pi(\theta)$ for such an $L$-packet. 
Let $v$ be a finite place of $F$ such that $\widetilde{\pi}_v$ is unramified. If $v$ splits in $\widetilde{F}$, then the local $L$-packet $\Pi(\theta)_v$ is a singleton consisting of an $\SL_2(\mc{O}_v)$-unramified representation. If $v$ is inert in $\widetilde{F}$, then $\Pi(\theta)_v=\{\pi_{v,1},\pi_{v,2}\}$ is an unramified $L$-packet of size two, as in Proposition 3.2.12 of \cite{lansky-raghuram}, with a unique $\SL_2(\mc{O}_v)$-unramified element.  

\begin{lemma}\label{swap} Let $\widetilde{\pi}$ be a cuspidal cohomological automorphic representation of $\GL_2(\A_F)$ of cohomological weight $k$, so $\pi_{\infty} \cong D_k = \bigotimes_{\sigma \in \Sigma_\infty} D_{k_\sigma +2, w}$. Assume that $\widetilde{\pi}$ has complex multiplication.
Let $\sigma \in \Sigma_{\infty}$ correspond to an archimedean place $v_0$. Then for any representation $\pi^{v_0}=\bigotimes_{v\neq v_0}^\prime \pi_v \in \Pi(\widetilde{\pi})^{v_0}$
precisely one of the representations 
\[ \pi^+:=\pi^{v_0}\otimes D^+_{k_\sigma +2} \ ,\  \ \pi^-:=\pi^{v_0}\otimes D^{-}_{k_\sigma +2} \]
is automorphic and has multiplicity one, and the other representation has multiplicity zero. 
\end{lemma}
\begin{proof} This follows from \cite{labesse-langlands} and \cite{ramakrishnan} as follows. As recalled above, $\Pi(\widetilde{\pi})=\Pi(\theta)$ for a suitable character $\theta$. Our assumption on $\widetilde{\pi}$ to be cuspidal cohomological forces $\theta$ to be what is called of \emph{type (a)} in the classification of \cite[Section 5, p.764]{labesse-langlands}. 
For any $\pi \in \Pi(\theta)$, the multiplicity $m(\pi)$ is then given by the formula 
\[m(\pi)= \frac{d(\pi)}{2}\left(1 + \langle \epsilon, \pi \rangle\right) \]
from \cite[Proposition 6.7]{labesse-langlands}. Here $d(\pi)$ a priori is a positive integer (which only depends on the $L$-packet of $\pi$; see \cite[p. 773]{labesse-langlands}) and $\langle \epsilon, \pi \rangle = \prod_v \langle \epsilon_v, \pi_v \rangle$ is a product of local signs $\langle \epsilon_v, \pi_v \rangle \in \{-1,1\}$, with $\langle \epsilon_v, \pi_v \rangle =1$ for almost all $v$ (see \cite[p.766]{labesse-langlands}), and for any $v\mid \infty$, we have
\[\{\langle \epsilon_v, D^+_{k_\sigma+2} \rangle, \langle \epsilon_v, D^-_{k_\sigma+2} \rangle\}=\{-1,1\}.\]
This implies that exactly one of $m(\pi^+)$ and $m(\pi^-)$ is zero. From \cite[Theorem 4.1.1]{ramakrishnan} we conclude that $d(\pi)=1$ for any $\pi \in \Pi(\theta)$, hence the lemma follows.  
\end{proof}

\subsection{Pseudorepresentations}\label{subsec: pseudoreps}

In this section we will discuss some results on the pseudorepresentations introduced by V. Lafforgue in \cite{lafforgue}. Let $E$ be a field of characteristic $0$ and let $H$ be a split connected reductive group over $E$. Let $\G$ be a (discrete) group and let $A$ be an $E$-algebra. For any integer $n\geq 1$, let $E[H^n]^H$ denote the global functions on $H^n$ which are invariant under the action of $H$ on $H^n$ by diagonal conjugation, and let $C(\G^n,A)$ denote the set-theoretic functions from $\G^n$ to $A$. A pseudorepresentation of $\G$ valued in $H(A)$ is a collection $(\theta_n)_{n\geq1}$ of $E$-algebra homomorphisms
\[
\theta_n : E[H^n]^H \to C(\G^n,A)
\]
satisfying the following conditions:

\begin{itemize}
\item For any $r,s \in \Z_{\geq 1}$ and for any function $\zeta : \{ 1,\dots, r \} \to \{1,\dots, s\}$, any $f\in E[H^r]^H$ and any $s$ elements $\gamma_1,\dots,\gamma_s \in \G$, we have
\[
\theta_s(f^\zeta)(\gamma_1,\dots,\gamma_s) = \theta_r(f)(\gamma_{\zeta(1)},\dots, \gamma_{\zeta(r)}),
\]
where $f^\zeta \in E[H^s]^H$ is defined by $f^\zeta(h_1,\dots,h_s) = f(h_{\zeta(1)},\dots,h_{\zeta(r)})$.

\smallskip

\item For any $r \in \Z_{\geq 1}$, any $\gamma_1,\dots,\gamma_r, \gamma_{r+1} \in \G$ and any $f\in E[H^r]^H$, we have
\[
\theta_{r+1}(\hat{f})(\gamma_1,\dots,\gamma_r,\gamma_{r+1}) = \theta_r(f)(\gamma_1, \dots, \gamma_r \gamma_{r+1}),
\]
where $\hat{f} \in E[H^{r+1}]^H$ defined by $\hat{f}(h_1,\dots,h_r,h_{r+1}) = f(h_1,\dots,h_r h_{r+1})$. 

\end{itemize}
 

\medskip

We recall that any homomorphism $\rho : \G \to H(A)$ induces a pseudorepresentation $\theta$ by letting $\theta_n : E[H^n]^H \to C(\G^n,A)$ be given by
\[
\theta_n(f)(\gamma_1,\dots,\gamma_n) = f(\rho(\gamma_1),\dots, \rho(\gamma_n)).
\]
If $B$ is an $E$-algebra, $\phi : A \to B$ is homomorphism and $\theta$ is a pseudorepresentation valued in $H(A)$, then we may form a new pseudorepresentation $\phi \circ \theta := (\phi_n \circ \theta_n)_{n \geq 1}$, where $\phi_n$ is the $E$-algebra homomorphism $C(\G^n, A) \to C(\G^n,B)$ induced by $\phi$. Thus we may, for fixed $\G$ and $H$, define a moduli functor sending an $E$-algebra $A$ to the set of pseudorepresentations of $\G$ into $H(A)$. This functor is representable by an $E$-algebra by \cite[Theorem 3.0.6(ii)]{emerson}, which we call $R^{ps}$. Attached to $\G$ and $H$ we also have the representation variety and the character variety; recall that the representation variety is the affine $E$-scheme $\Spec R^{rep}$ representing the functor
\[
A \mapsto \Hom_{\mathrm{Group}}(\G, H(A))
\]
and that the character variety is $\Spec R^{char}$, with $R^{char} := (R^{rep})^H$. Association of a pseudorepresentation to a representation then defines a map
\[
\Spec R^{rep} \to \Spec R^{ps}
\]
which factors through the character variety, inducing a canonical map $R^{ps} \to R^{char}$. The following theorem is due to Emerson \cite[Theorem 6.0.5(iii)]{emerson} (note that it crucially uses the assumption that $E$ has characteristic $0$). Let $\rho^{univ}$ denote the universal homomorphism $\G \to H(R^{rep})$, and let $\theta^{univ}$ denote the universal pseudorepresentation.

\begin{theorem}
The canonical map $R^{ps} \to R^{char}$ is an isomorphism. In particular, the map $\iota : R^{ps} \to R^{rep}$ is injective, and $\theta^{univ}$ may be viewed as the pseudorepresentation attached to $\rho^{univ}$.
\end{theorem}

We now set $H = \SO_{2n+1}$ (we are really interested in $H = \PGL_2$ for the purposes of this paper, but it is convenient to use the exceptional isomorphism $\PGL_2 \cong \SO_3$ via the $\mathrm{ad}^0$-representation). Inside $E[H]^H$ we have a function $\tr$, recording the trace of a matrix in $H$. Let $\theta$ be a pseudorepresentation into $H(A)$. We may consider the function $\theta_1(\tr)$ on $\G$ (if $\theta$ is attached to a representation $\rho$, then $\theta_1(\tr)$ is simply $tr\, \rho$). We aim to prove the following:

\begin{proposition}\label{generation}
Let $\theta^{univ}$ be the universal pseudorepresentation into $H=\SO_{2n+1}$. Then $R^{ps}$ is generated, as an $E$-algebra, by the elements $\theta_1^{univ}(\tr)(\gamma)$, $\gamma \in \G$.
\end{proposition}

An alternative formulation, if we view $\iota : R^{ps} \to R^{rep}$ as an inclusion, is that $R^{ps}$ is the $E$-subalgebra of $R^{rep}$ generated by the elements $\tr(\rho^{univ}(\gamma))\in R^{rep}$, $\gamma \in \G$. We will allow ourselves this point of view in the proof.

\begin{proof}
Let $S\sub R^{ps}$ denote the $E$-subalgebra generated by the traces; we wish to show that $S=R^{ps}$. The key step is to show that, for every $n$, $\theta_n^{univ}$ takes values in $C(\G^n,S)$.

\medskip

Let $n\geq 1$ be arbitrary. For each word $WD(X_1,\dots,X_n)$ in $X_1,\dots,X_n,X_1^{-1},\dots,X_n^{-1}$ (i.e., an element of the free group on $X_1,\dots,X_n$), define a function in $E[H^n]^H$ by
\[
(h_1,\dots,h_n) \mapsto \tr(WD(h_1,\dots,h_n)).
\]
By \cite[Theorem 7.1]{procesi}, $E[H^n]^H$ is generated, as an $E$-algebra, by these functions. From this it follows that $\theta_n^{univ}$ takes values in $C(\G^n,S)$, by using the identity
\[
\theta_n^{univ}(f) (\gamma_1,\dots,\gamma_n) = f(\rho^{univ}(\gamma_1),\dots,\rho^{univ}(\gamma_n)).
\]
The rest of the proof is then standard: We consider the pseudorepresentation $\theta = (\theta_n)_{n \geq 1}$, $\theta_n : E[H^n]^H \to C(\G^n,S)$ which is ``just $\theta^{univ}$'', but viewed as a pseudorepresentation taking values in $H(S)$. Let $i : S \to R^{ps}$ be the inclusion, then $i \circ \theta = \theta^{univ}$ by construction. On the other hand, the universal property of $\theta^{univ}$ implies that there must be a homomorphism $ \pi : R^{ps} \to S$ such that $\theta = \pi \circ \theta^{univ}$. It follows that $\theta^{univ} = i \circ \pi \circ \theta^{univ}$, so the universal property of $\theta^{univ}$ implies that $i \circ \pi$ is the identity on $R^{ps}$. But this forces $i$ to be surjective, which means that $S=R^{ps}$ as desired.
\end{proof}

\subsection{Deformation Theory}\label{subsec: def theory}
Let $\Lambda$ be a complete local Noetherian ring, with residue field $\ka$ (we allow $\Lambda = \ka$). We will make use of the following notation for some standard categories in deformation theory: We let $\Art_{\Lambda}$ and $\CNL_{\Lambda}$ denote the category of artinian and complete local Noetherian $\Lambda$-algebras with residue field $\ka$, respectively. When $A $ is any local ring, we let $\m_A$ denote the maximal ideal of $A$. Given local homomorphisms $A \to B$, $A \to C$ of complete local Noetherian rings, we may form the completed tensor product
\[
B \ctens_A C := \varprojlim_i \left( B/\m_B^i \otimes_{A/\m_A^i} C/\m_C^i \right).
\]
In particular, if $\Lambda \to \Lambda^\prime$ is a local homomorphism of complete local Noetherian rings, then the completed tensor product $A \mapsto A \ctens_\Lambda \Lambda^\prime$ is a functor $\CNL_\Lambda \to \CNL_{\Lambda^\prime}$. We let $\Set$ denote the category of sets. If
\[
X : \CNL_\Lambda \to \Set
\]
is a representable functor (or $X : \Art_\Lambda \to \Set$ is a pro-representable functor), we write $\oo(X)$ for the object in $\CNL_\Lambda$ that (pro-)represents it.

\medskip

Our goal in this section is an analogue of Proposition \ref{generation} for Galois deformation rings, which will be used later to relate the eigenvarieties for $\GL_2$ and $\SL_2$ at points corresponding to a refined automorphic representation. Throughout this subsection, we fix a cuspidal cohomological automorphic representation $\pi$ of $\GL_2(\A_F)$. Let $S$ be any finite set of places of $F$ containing the infinite places and the places above $p$ such that $\rho_\pi$ is unramified outside $S$. We now let $H = \SO_3$ and let $E$ be a finite extension of $\Qp$ inside $\ol{\Q}_p$ which we think of as large, in particular we may assume that $\rho_\pi(G_F) \subseteq \GL_2(\oo)$ (after conjugation if necessary), were $\oo$ is the ring of integers of $E$. We let $k$ be the residue field of $E$.

\medskip

Let 
\[
r = \mathrm{ad}^0\rho_\pi : G_{F,S} \to H(E).
\] 
Note that $r$ is irreducible as an $H$-valued representation since $\pi$ is cuspidal, even though $\mathrm{ad}^0\rho_{\pi}$ might be reducible as a $\mathrm{GL}_3$-valued representation. Here we recall from \cite[Definition 3.5]{bhkt} that a homomorphism $\phi : G_{F,S} \to H(E^\prime)$, with $E^\prime/E$ algebraically closed, is said to be completely reducible if, whenever the Zariski closure of $\phi(G_{F,S})$ is contained in a parabolic subgroup, it is moreover contained in a Levi subgroup of said parabolic. Moreover, $\phi$ is called irreducible if the image is not contained in any proper parabolic of $H$. By assumption $r$ lands in $H(\oo)$ and so has a reduction $\ol{r} : G_{F,S} \to H(k)$. We have a universal framed deformation functor $\fX^\square_r $ parametrising continuous lifts of $r$
\[
\wt{r} : G_{F,S} \to H(A)
\]
to $\CNL_E$. It is representable by an object $R^\square_r \in \CNL_E$. We need the following lemma, whose proof is standard.

\begin{lemma}\label{quotient}
There exists a quotient $Q$ of $G_{F,S}$ by a closed subgroup such that $Q$ is topologically finitely generated and such that any continuous lift $\wt{r} : G_{F,S} \to H(A)$ to an object $A\in \CNL_E$ factors through $Q$.
\end{lemma}

\begin{proof}
Let $G^\prime = \Ker \ol{r} \sub G_{F,S}$ and let $G^{\prime\prime} \sub G^\prime$ be the closed subgroup cutting out the maximal pro-$p$ quotient of $G^\prime$. Then $G^{\prime\prime}$ is characteristic in $G^\prime$ and hence normal in $G_{F,S}$; the quotient $Q = G_{F,S}/G^{\prime\prime}$ is then topologically finitely generated and we claim that every $\wt{r}$ as above factors through $Q$.

\medskip

It suffices to prove this for $A \in \Art_E$. When $A=E$, it follows since $\Ker ( H(\oo) \to H(k) )$ is pro-$p$. For a general $A \in \Art_E$, we filter $H(A)$ by the subgroups $H_i = \Ker (H(A) \to H(A/\mf{m}^i) )$ , where $\mf{m}\sub A$ is the maximal ideal, and the result now follows since $H_i/H_{i+1}$ is a finite-dimensional $E$-vector space for all $i\geq 1$, and all compact subgroups of finite-dimensional $E$-vector spaces are pro-$p$.
\end{proof}

We now fix such a topologically finitely generated quotient $Q$ of $G_{F,S}$ as in Lemma \ref{quotient}, and we can and will consider all lifts of $r$ as representations of $Q$. Let $\G \sub Q$ be a finitely generated dense subgroup. We view $\G$ as a discrete group; any continuous representation of $Q$ is then determined by its restriction to $\G$. We apply the material from \S \ref{subsec: pseudoreps} to $\G$ and the algebraic group $H$ --- we have a representation variety $X^{rep}=\Spec R^{rep}$, the character variety $X^{char} = \Spec (R^{rep})^H$ and the pseudorepresentation variety $X^{ps}=\Spec R^{ps}$, and we may identify $X^{char}$ and $X^{ps}$. Note that these are finite type schemes over $E$, since $\G$ is finitely generated. The representation $r|_{\G}$ defines an $E$-point of $X^{rep}$, and we may form the completed local ring $R^{rep}_r$ of $R^{rep}$ at that point.

\begin{lemma}\label{surjective1}
There is a canonical map $R_r^{rep} \to R_r^\square$, which is surjective.
\end{lemma}

\begin{proof}
The ring $R_r^{rep}$ classifies lifts of $r|_{\G}$ to objects $A \in \Art_E$. So there is canonical map of functors
\[
\fX_r^\square \to \Hom(R^{rep}_r,-)
\]
given by restriction from $Q$ to $\G$, which is injective since $\G$ is dense in $Q$. This gives the natural map $R_r^{rep} \to R_r^\square$, which is surjective since the map on functors of points is injective.
\end{proof}

Now let $\theta$ be the pseudorepresentation of $r|_\G$. This gives a point on $X^{ps}$, and our next goal is to compute the completed local ring $R^{ps}_\theta$ of $R^{ps}$ at $\theta$. We start with a technical lemma.

\begin{lemma}\label{orbitclosed}
The orbit of $r|_{\G}$ in $X^{rep}$ is closed.
\end{lemma}

\begin{proof}
Consider the map $f : X^{rep} \to X^{ps}$. Since $\theta$ defines a closed point of $X^{ps}$, it suffices to prove that the orbit of $r_{\G}$ inside $X^{rep}$ is equal to $f^{-1}(\theta)$. In other words, we need to show that every representation whose pseudorepresentation equals $\theta$ is conjugate to $r|_{\G}$ (over an algebraically closed extension). For completely reducible representations, this is \cite[Theorem 4.5]{bhkt} (the proof being due to V. Lafforgue). Since $r$ is irreducible (in the sense above), any $\G \to H(E^\prime)$ with pseudorepresentation $\theta$ must be irreducible as well, and hence completely reducible. This finishes the proof.
\end{proof}

Before computing $R_\theta^{ps}$, we briefly discuss quotients of functors on $\CNL_\Lambda$ by group functors, refering to \cite[\S 2.4]{khare-wintenberger} for a more thorough discussion. We remark that while the general setup of \cite{khare-wintenberger} assumes that the residue field is finite, the discussion in \cite[\S 2.4]{khare-wintenberger} does not use this assumption and we will freely use results from there without further comment. Let $\Lambda $ be a complete local Noetherian ring and let $X : \CNL_\Lambda \to \Set$ be a functor acted on by a group functor $G : \CNL_\Lambda \to \Grp$. The quotient $X/G$ is defined by
\[
A \mapsto X(A)/G(A).
\]
If $X$ and $G$ are representable, $G$ is smooth and the action is free (meaning that $G(A)$ acts freely on $X(A)$ for all $A\in \CNL_\Lambda$), then $X/G$ is represented by the equalizer of
\[
\oo(X) \rightrightarrows \oo(G) \ctens_\Lambda \oo(X),
\]
where the two maps are dual to the action map $G \times X \to X$ and the projection map $G \times X \to X$. Moreover, the natural map $X \to X/G$ is a (trivial) $G$-torsor, hence smooth; for all of this see \cite[Proposition 2.5]{khare-wintenberger}. We will also write $\oo(X)^G$ for $\oo(X/G)$. 

\medskip

Now we go back to $R_\theta^{ps}$. Let $\hH$ be the (representable) group functor on $\CNL_E$ defined by
\[
\hH (A) = \Ker\left( H(A) \to H(E) \right),
\]
for $A\in \CNL_E$. It acts naturally on $R^{rep}_r$. Moreover, let $H_r$ be the centraliser (in $H$) of the image of $r$; it is either trivial or isomorphic to $\Z/2$ (as a group scheme over $E$). It also acts naturally on $R^{rep}_r$.

\begin{proposition}\label{localring1}
The natural map $R_\theta^{ps} \to R_r^{rep}$ induces an isomorphism $R_\theta^{ps} \cong \left( (R_r^{rep})^{\hH} \right)^{H_r}$.
\end{proposition}

\begin{proof} 
It suffices to prove this after base changing from $E$ to an algebraic closure $\ol{E}$. Indeed the completed tensor product functor $\CNL_E \to \CNL_{\ol{E}}$ is exact, commutes with equalizers and a homomorphism in $\CNL_E$ is an isomorphism if and only if it maps to an isomorphism in $\CNL_{\ol{E}}$. So we may base change to $\ol{E}$, but for simplicity with still use the same notation for the base changed objects throughout this proof.

\medskip

By Luna's \'etale slice theorem \cite[p. 97]{luna} there is an affine, locally closed subvariety $Y = \Spec T \sub X^{rep}$, which contains $r$ and is stable under $H_r$ (being the stabilizer of $r$), such that we have a Cartesian diagram
\[
\xymatrix{Y \times^{H_r} H \ar[r] \ar[d] & X^{rep} \ar[d] \\
Y/H_r \ar[r] & X^{ps}
}
\]
where both horizontal maps are \'etale and the top horizontal map is $H$-equivariant. In particular, we have maps
\[
Y \times H \to (Y \times H)/H_r = Y\times^{H_r} H \to X^{rep},
\]
where the right map is the map appearing in the commutative square above. The left map is \'etale, since $H_r$ acts freely on $Y \times H$, so the composition is \'etale. The point $(r,1) \in Y \times H$ maps to $r \in X^{rep}$, so we get an $\hH$-equivariant isomorphism $T_r \ctens_E \oo(\hH) \cong R_r^{rep}$ of completed local rings, where $T_r$ is the completed local ring of $T$ at $r$. Taking quotients, it follows that $T_r \cong (R_r^{rep})^{\hH}$.

\medskip

To finish off, note that the \'etaleness of $Y/H_r \to X^{ps}$ induces an isomorphism $T_r^{H_r} \cong R_\theta^{ps}$ (here we use that $H_r$ is finite to see that $T_r^{H_r}$ is the completed local ring of $Y/H_r$ at the orbit of $r$). Combining this with the previous isomorphism then gives the proposition.
\end{proof}

Now recall the deformation functor $\fX_r$ of $r$. It is defined as the quotient of $\fX_r^\square$ by the action of $\hH$. This action is free (since $r$ is irreducible), so $\fX_r$ is represented by $R_r = (R_r^\square)^{\hH}$. Similarly, $\hH$ acts freely on $R_r^{rep}$ since $r|_\Gamma$ is irreducible.

\begin{corollary}\label{surjective2}
The surjection $R^{rep}_r \to R_r^\square$ from Lemma \ref{surjective1} induces a surjection $R_\theta^{ps} \to R_r^{H_r}$.
\end{corollary}

\begin{proof}
The maps $\Hom(R_r^\square,A) \to \Hom(R_r^{rep},A)$ are injective for all $A\in \Art_E$ and both the domain and the codomain are freely acted on by $\hH(A)$, so we conclude that
\[
\Hom(R_r,A) = \Hom(R_r^\square,A)/\hH(A) \to \Hom(R_r^{rep},A)/\hH(A) = \Hom((R_r^{rep})^{\hH},A)
\]
is injective for all $A$. It follows that the induced map $(R_r^{rep})^{\hH} \to R_r$ is surjective. The result then follows by taking $H_r$-invariants (which are exact since $H_r$ is finite) and applying Proposition \ref{localring1}.
\end{proof}

We now get to the main result of this section.

\begin{theorem}\label{thethingIwant}
$R_r^{H_r}$ is topologically generated by the traces $\tr(r(\Frob_v))$ for $v \notin S$.
\end{theorem}

\begin{proof} Corollary \ref{surjective2} gives a surjection $R_\theta^{ps} \to R_r^{H_r}$ and, by Proposition \ref{generation}, $R_\theta^{ps}$ is topologically generated by the traces $\tr(r_{\G}^{univ}(\gamma))$, $\gamma \in \G$, where $r_{\G}^{univ}$ is the universal representation $\G \to H(R^{rep})$. Let $r^{univ} : Q \to H(R_\rho)$ be the universal deformation of $r$. By a diagram chase, we see that the image of $\tr(r_{\G}^{univ}(\gamma))$ under the map $R^{ps} \to R_r^{H_r}$ is equal to $\tr(r^{univ}(\gamma))$, for $\gamma \in \G$.

\medskip

To simplify notation, set $A = R_\theta^{ps}$, $B = R_r^{H_r}$ and set $A_n = A/\mf{m}_A^n$ and $B_n = B/\mf{m}_B^n$, for $n\in \Z_{\geq 1}$. We need to show that $B_n$ is generated by $\tr(r^{univ}(\Frob_v))$ for $v \notin S$. By the first paragraph of this proof, $B_n$ is generated by the elements $\tr(r^{univ}(\gamma))$, $\gamma \in \G$. Since $B_n$ is a finite-dimensional $E$-vector space, this means that we can find an $E$-vector space basis of finitely many polynomials in the $\tr(r^{univ}(\gamma))$ for $B_n$. Since $r^{univ}$ modulo $\mf{m}_B^n$ is continuous and $\{ \Frob_v \mid v \notin S \}$ is dense in $Q$, we may approximate this basis arbitrarily well using polynomials in the $\tr(r^{univ}(\Frob_v))$. Since a small perturbation of a basis of a finite dimensional $E$-vector space is still a basis, we see that $B_n$ has an $E$-vector space basis consisting of polynomials in the $\tr(r^{univ}(\Frob_v))$. Thus $B_n$ is generated by the $\tr(r^{univ}(\Frob_v))$ as an $E$-algebra, as desired. This finishes the proof.
\end{proof}

Now set $\rho = \rho_\pi$, where we recall that $\pi$ is assumed to be cuspidal cohomological. We finish this subsection by comparing $R_r$ with the universal deformation ring $R^{\psi}_\rho$ with fixed determinant $\psi = \det \rho$ of $\rho$. Sending a deformation to its projectivization gives a map
\[
R_r \to R^{\psi}_\rho
\]  
which induces an isomorphism on tangent spaces, since both rings have tangent space $H^1(G_{F,S}, \ad^0 \rho)$ and the induced map is the identity. Here and below, we write $\ad^0 \rho$ instead of $r$ when we want to think of it as a three-dimensional representation (and not a projective representation, or an orthogonal representation).

\begin{proposition}\label{smoothness}
If $F\neq \Q$ and $\pi$ has CM by the extension $\wt{F}/F$, assume that $\wt{F} \not\sub F(\zeta_{p^\infty})$. Then $R^{\psi}_\rho$ is formally smooth of dimension $2d$.
\end{proposition}

\begin{proof}
This follows from the vanishing of $H_f^1(G_F, \ad^0\rho)$ by an argument of Kisin. We make a few remarks since our setup does not match the available references (in particular, we don't want to assume that $p > 2$). First, our assumptions do imply that $H_f^1(G_F, \ad^0\rho)=0$, by \cite[Theorem 5.4]{newton-thorne} and \cite[Proposition 2.15(iv)]{bellaiche-crit}. Then, by the argument in \cite[Theorem 8.2]{kisin-geometricdef} (see also the proof of \cite[Theorem 6.1.6]{allen2}), the vanishing of $H_f^1(G_F, \ad^0\rho)$ implies that the tangent space $H^1(G_{F,S}, \ad^0\rho)$ has dimension at most $2d$ (this uses local-global compatibility and genericity at finite places away from $p$. Genericity is satisfied as $\pi$ is cuspidal). By \cite[Lemma 9.7]{kisin-overconvmfs} (the global Euler characteristic formula plus vanishing of $H^3$) and irreducibility of $\rho$, 
\[
\dim H^1(G_{F,S}, \ad^0\rho) - \dim H^2(G_{F,S}, \ad^0\rho) = 2d.
\]
It follows that $\dim H^1(G_{F,S}, \ad^0\rho) = 2d$ and $\dim H^2(G_{F,S}, \ad^0\rho) = 0$, which gives the result.
\end{proof}

\begin{corollary}\label{defringsiso}
If $F\neq \Q$ and $\pi$ has CM by the extension $\wt{F}/F$, assume that $\wt{F} \not\sub F(\zeta_{p^\infty})$. Then the natural map $R_r \to R^{\psi}_\rho$ is an isomorphism.
\end{corollary}

\begin{proof}
Any map $ A \to B$ in $\CNL_E$ which is an isomorphism on tangent spaces and where $B$ is formally smooth is an isomorphism, so the corollary follows from Proposition \ref{smoothness}. 
\end{proof}

\section{Eigenvarieties}\label{sec: eigenvarieties}

In this section we set up the basics of overconvergent cohomology and eigenvarieties for $\GL_2$ and $\SL_2$ over $F$, and recall results on the family of Galois representation over $\GL_2$. We largely follow \cite{bh}, with only some superficial changes. We then discuss the $p$-adic functoriality map between these eigenvarieties and study its fibres at points corresponding to refined automorphic representations, under some mild technical assumptions.

\subsection{Overconvergent cohomology for $\GL_2$} \label{subsec: ocGL2} 
Recall that $\mc{O}_p = \prod_{v\mid p} \mc{O}_v$, and that we have fixed a uniformizer $\varpi_v \in \mc{O}_v$ for every $v\mid p$ and defined $\varpi_p:=\prod_{v \mid p} \varpi_v \in \mc{O}_p$. We let $\mbf{e}:=\prod_{v \mid p} e_v$, where $e_v$ is the ramification index of $F_v/\Q_p$.

\subsubsection{Locally analytic distribution modules on $\mc{O}_p$}
To define overconvergent cohomology, we will work with the locally analytic distribution modules used in \cite{bh}. We recall the definitions from \cite[\S 5.2]{bh}. Choose a $\Z_p$-linear isomorphism $\nu :\Z^d_p \cong \mc{O}_p$ (a global chart on $\oo_p$) and consider the $\Zp$-algebra
\[ 
\mbf{A}^{\circ}(\mc{O}_p, \Qp) := \{f :\mc{O}_p \rightarrow \Q_p : f\circ \nu \in \Z_p\langle z_1,\dots, z_d \rangle\}.
\]
Then $\mbf{A}(\mc{O}_p, \Qp) := \mbf{A}^{\circ}(\mc{O}_p, \Qp)[1/p]$ is isomorphic to $\Q_p\langle z_1,\dots, z_d \rangle$ via $f \mapsto f\circ \nu$ and this is a $\Q_p$-Banach algebra, with norm induced from the supremum norm on  $\Q_p\langle z_1,\dots, z_d \rangle$. 

\medskip

We will use the following spaces of locally analytic functions on $\mc{O}_p$: 
For $\mbf{s}=(s_v)_{v \mid p} \in \Z_{\geq 0}^{v \mid d}$ we define
\[ 
\mbf{A}^{\mbf{s},\circ}(\mc{O}_p, \Q_p) := \{f :\mc{O}_p \rightarrow \Q_p : z\mapsto f(a+ \varpi_p^{\mbf{s}}z) \in \mbf{A}^{\circ}(\mc{O}_p, \Q_p) \text{ for all } a\in \mc{O}_p \} \text{ and }
\]
\[
\mbf{A}^{\mbf{s}}(\mc{O}_p, \Q_p) = \mbf{A}^{\mbf{s},\circ}(\mc{O}_p, \Q_p)[1/p], 
\]
and the latter is again a $\Q_p$-Banach algebra. For any Noetherian $\Q_p$-Banach algebra $R$ we equip 
\[
\mbf{A}^{\mbf{s}}(\mc{O}_p, R):= \mbf{A}^{\mbf{s}}(\mc{O}_p, \Q_p)\widehat{\otimes}_{\Q_p} R
\]
with its induced tensor product topology; this is a potentially orthonormalizable $R$-Banach module. When $\mbf{s}'\geq \mbf{s}$, the natural map $\mbf{A}^{\mbf{s}}(\mc{O}_p, R)\rightarrow \mbf{A}^{\mbf{s}'}(\mc{O}_p, R)$ is injective with dense image and compact whenever $\mbf{s}'\geq \mbf{s} + \mbf{e}$. 

\medskip

The space of $R$-valued locally analytic functions on  $\mc{O}_p$ is defined as the direct limit
\[\mathcal{A}(\mc{O}_p,R)=\varinjlim_{|\mbf{s}|\rightarrow \infty} \mbf{A}^{\mbf{s}}(\mc{O}_p, R), \]
where $|\mbf{s}|=\min \{s_v : v \mid  p\}$. We put the direct limit topology on $\mathcal{A}(\mc{O}_p,R)$.
We denote the $R$-Banach space dual of $\mbf{A}^{\mbf{s}}(\mc{O}_p, R)$ by $\mbf{D}^{\mbf{s}}(\mc{O}_p, R)$ and set
\[\mathcal{D}(\mc{O}_p,R) = \varprojlim_{|\mbf{s}|\rightarrow \infty} \mbf{D}^{\mbf{s}}(\mc{O}_p,R).\]
We equip $\mathcal{D}(\mc{O}_p,R) $ with the projective limit topology. The Banach algebras $R$ that we will use in practice come from $p$-adic weights. To make this precise, let us introduce the weight space that we shall be working with. It differs from that in \cite{bh}, and has been chosen with our application to $\SL_2$ in mind.

\subsubsection{Weight space}\label{sec:weightspace}
We fix a finite extension $E \sub \Qp$ inside $\ol{\Q}_p$, which is assumed to be `large'. All adic spaces we consider will be over $E$. Let $\mbf{G}={\rm Res}_{\Q}^{F}\GL_2$. We set $\mbf{G}_{\Zp}=\prod_{v\mid p}{\rm Res}_{\Zp}^{\oo_{F_{v}}}\GL_{2}$, and 
consider $\mbf{T}=\prod_{v\mid p}{\rm Res}_{\Zp}^{\oo_{F_{v}}}\mbf{T}_{v}$, where $\mbf{T}_{v}$ is the (maximal) diagonal torus in $\GL_2/\oo_{F_{v}}$. Then $\mbf{T}(\Z_p)=\prod_{v\mid p} \mbf{T}_v(\mc{O}_{F_v})\cong (\mc{O}_p^{\times})^2$. This is a compact abelian $p$-adic Lie group and to it we may associate an analytic adic space $\mathcal{W}^{big}:= \Spf(\mathbb{Z}_p[[\mbf{T}(\Z_p)]])^{rig}\times_{\Spa(\Q_p)} \Spa(E)$, which we refer to as the `big' weight space. This is the $2d$-dimensional rigid analytic space over $E$ such that, for any affinoid $E$-algebra $R$, the $R$-points are given by
\[ 
\mathcal{W}^{big}(R)=\mathcal{W}^{big}(R,R^{\circ})= \{\text{continuous characters } \ka: \mbf{T}(\Z_p)\rightarrow R^{\times}\}
\]
(see \cite[Definition 5.1.4]{bh} and the discussion following it). Any point $\ka \in \mathcal{W}^{big}(R)$ has components $\ka_1,\ka_2: \mc{O}_p^{\times}\rightarrow R^{\times}$, given by $\ka_1(z) = \ka\left( \left( \begin{smallmatrix} z & 0 \\ 0 & 1 \end{smallmatrix} \right) \right)$ and $\ka_2(z) = \ka\left( \left( \begin{smallmatrix} 1 & 0 \\ 0 & z \end{smallmatrix} \right) \right)$. Let us fix a classical cohomological weight $k = (k_1,k_2) = ((k_\sigma)_\sigma, w)$. In order to compare eigenvarieties on $\GL_2$ and $\SL_2$ locally around a point of weight $k$, we consider the following subspace of $\mathcal{W}^{big}$.

\begin{definition}
Let $k$ be a cohomological weight. We define the $k$-\emph{weight space} as the Zariski closed subspace $\mc{W}_k \sub \mc{W}^{big}$ such that for any $E$-algebra $R$:
\[ \mc{W}_k(R)=\{ (\ka_1,\ka_2) \in \mathcal{W}^{big}(R): \ka_1\ka_2=k_1 k_2\}.  \]
\end{definition}

It is easy to see that $\mc{W}_k$ is isomorphic to $\Spf(\mathbb{Z}_p[[\mc{O}_p^{\times}]])^{rig}\times_{\Spa(\Q_p)} \Spa(E)$ by sending $\ka$ in the latter to $(\ka, \ka^{-1}k_1 k_2)$ in the former, so $\mc{W}_k$ is $d$-dimensional and smooth. For any $E$-affinoid algebra $R$, we will call an element of $\mc{W}_k(R)$ a \emph{$p$-adic weight}, and we will restrict ourselves to these weights in what follows.

\subsubsection{Monoid-action}
Let $R$ be a Noetherian $E$-Banach algebra. For any $p$-adic weight $\ka :\Spa(R) \rightarrow \mc{W}_k$, the space of locally analytic distributions $\mc{D}(\mc{O}_p,R)$ can be equipped with a `weight $\ka$'-action by a certain monoid. We briefly recall the definitions, referring to \cite[\S 5.3]{bh} for more details. For any continuous character $\chi :\mc{O}_p^{\times} \rightarrow R^{\times}$, there exists (by a theorem of Amice) a tuple $\mbf{s}(\chi) \in \Z_{\geq 0}^{\{v\mid p\}}$ such that the extension $\chi_{!}: \mc{O}_p \rightarrow R$ of $\chi$ by zero is an element of $\mbf{A}^{\mbf{s}(\chi)}(\mc{O}_p, R)$. Note that then also $z \mapsto \chi(cz+d) \in \mbf{A}^{\mbf{s}(\chi)}(\mc{O}_p, R)$ for any $c\in \varpi_p \mc{O}_p, d \in \mc{O}_p^{\times}$, as the map $z\mapsto cz+d$ is a polynomial (cf. \cite[Lemma 5.3.1]{bh}). 

\medskip

Let $\ka = (\ka_1,\ka_2) : \Spa(R) \rightarrow \mc{W}_k$ be a $p$-adic weight and define $\mbf{s}(\ka):= \mbf{s}(\ka_1 \ka_2^{-1})$. For $v \mid p$, recall that $I_v$ denotes the Iwahori subgroup of $\GL_2(\oo_v)$ of upper triangular matrices modulo $\varpi_v$, and that $I = \prod_{v \mid p} I_v \sub \GL_2(F_p)$. Consider the subgroup 
\[
N_1= \begin{pmatrix} 1 & \varpi_p\mc{O}_p \\  & 1 \end{pmatrix}\subset \GL_2(\mc{O}_p).
\]
For  $v\mid p$, let $\Sigma_v:= \left\{ (^{\varpi_v^a} \ _{\varpi_v^b}) : a, b \in \Z \right\}\subset \GL_2(F_v)$ and put $\Sigma:= \prod_{v\mid p} \Sigma_v$. Let 
 $$ \Sigma^{+}= \left\{ t\in \Sigma \mid tN_{1}t^{-1}\sub N_{1} \right\}. $$

We set $\Delta_p:=I\Sigma^{+}I \subset \GL_2(F_p)$ and note that this is a monoid. Consider the submonoid of $\Delta_p$ given by 
\[
\Delta_p^{\circ} = \left\{\begin{pmatrix} a & b \\ c & d \end{pmatrix} \in \GL_2(F_p)\cap M_2(\mc{O}_p)\mid c\in \varpi_p \mc{O}_p \text{ and } d\in \mc{O}_p^{\times}\right\}.
\]
Note that we indeed have $\Delta_p^\circ \sub \Delta_p$: If $g = \left(^a_c \ ^b_d \right) \in \Delta_p^\circ$ with $\det(g) = \varpi_p^{\mbf{a}}x$ where $x\in \oo_p^\times$, then we have
\[
\begin{pmatrix} d & -b \\ 0 & 1 \end{pmatrix}  \begin{pmatrix} a & b \\ c & d \end{pmatrix} =  \begin{pmatrix} ad-bc & 0 \\ c & d \end{pmatrix} = \begin{pmatrix} \varpi_p^{\mbf{a}} & 0 \\ 0 & 1 \end{pmatrix} \begin{pmatrix} x & 0 \\ c & d \end{pmatrix},
\]
which shows that $g\in I \Sigma^+ I$. Moreover, let $\mbf{a},\mbf{b} \in \Z^{\{v\mid p\} }$. As $\mathrm{diag}(\varpi_p^\mbf{a}, \varpi_p^\mbf{b})=\mathrm{diag}(\varpi_p^\mbf{b}, \varpi_p^\mbf{b})\mathrm{diag}(\varpi_p^{\mbf{a}-\mbf{b}},1)$, we can write any element $g \in \Delta_p$ uniquely as $\xi g^\circ$ with $g^\circ \in \Delta_p^{\circ}$ and $\xi=\mathrm{diag}(\varpi_p^{\mbf{b}}, \varpi_p^{\mbf{b}})$ a central element with entries a power of the product of the uniformizers.

\medskip

We equip $\mc{O}_p$ with the continuous left action of $\Delta_p$ given by $g\cdot z:= \frac{az+b}{cz+d}$ for $g = \left(^a_c \ ^b_d \right) \in \Delta_p$ and $z\in \mc{O}_p$. For $\ka$ as above and $\mbf{s}\geq \mbf{s}(\ka)$ we have a continuous $R$-linear right action of $\Delta_p$ on $\mbf{A}^{\mbf{s}}(\mc{O}_p, R)$ defined by
\[
f|_g(z):=\ka_1 \ka_2^{-1}(cz+d) \ka_2(\det(g) \varpi_p^{-v(\det g)}) f(g\cdot z) 
\]
where $g= \xi g^{\circ}$ with $g^{\circ} = \left(^a_c \ ^b_d \right) \in \Delta_p^{\circ}, f\in \mbf{A}^{\mbf{s}}(\mc{O}_p, R)$, and $z \in \mc{O}_p$. We equip $\mbf{D}^{\mbf{s}}(\mc{O}_p, R)$ with the dual left action, i.e., $(g\cdot \mu)(f)=\mu(f|_g)$ for $\mu \in \mbf{D}^{\mbf{s}}(\mc{O}_p, R), f \in \mbf{A}^{\mbf{s}}(\mc{O}_p, R)$ and $g\in \Delta_p$. 
For $\mbf{s}'\geq \mbf{s}$, the action of $\Delta_p $ is compatible with the injective restriction map $\mbf{A}^{\mbf{s}}(\mc{O}_p, R)\rightarrow \mbf{A}^{\mbf{s}'}(\mc{O}_p, R)$, hence we get a continuous action of $\Delta_p$ on $\mc{A}(\mc{O}_p,R)$, and similarly $\Delta_p$ acts on the locally analytic distributions $\mc{D}(\mc{O}_p,R)$.

\begin{definition}\label{weightaction} Let $\ka: \Spa(R) \rightarrow \mc{W}_k$ be a weight. For $\mbf{s}\geq \mbf{s}(\ka)$ we define $\mbf{A}^{\mbf{s}}_{\ka}:=\mbf{A}^{\mbf{s}}(\mc{O}_p, R)$, $\mbf{D}^{\mbf{s}}_{\ka}:=\mbf{D}^{\mbf{s}}(\mc{O}_p, R)$, $\mc{A}_{\ka}:=\mc{A}(\mc{O}_p, R)$ and $\mc{D}_{\ka}:=\mc{D}(\mc{O}_p, R)$ as the respective $R$-modules equipped with the continuous action by $\Delta_p$ defined above.
\end{definition}

\subsubsection{Hecke algebras}\label{subsec:Heckealgebras}
We need some more notation. First, we fix a compact open subgroup $K_{v}\sub \GL_2(F_v)$ for each prime $v \nmid p$, equal to $\GL_2(\mc{O}_v)$ for all but finitely many $v$. We write $K^{p}=\prod_{v \nmid p}K_{v}$ (and refer to it as the \emph{tame level}), and set $K=K^{p}I$. We let $S(K)$ denote the union of the set of places $v \nmid p$ where $K_v$ is not equal to $\GL_2(\mc{O}_v)$ and the places dividing $p\infty$.  For simplicity, we will assume that $K$ is neat, which is the case when $K^{p}$ is sufficiently small. 

\medskip

For our Hecke algebras, we consider the monoid
\[
\Delta = \Delta_p \times \prod_{v\notin S(K)} \GL_2(F_v) \times \prod_{v\in S(K),\, v\nmid p} K_v  \sub \GL_2(\A_F^\infty),
\]
which contains $K$. We will write $\mbf{T}(K) := \mbf{T}(\Delta,K)$ for the corresponding Hecke algebra over $\oo_E$. $\mbf{T}(K)$ may be described as
\[
\mbf{T}(\Delta_p,I) \otimes \bigotimes^{\prime}_{v\notin S(K)} \mbf{T}(\GL_2(F_{v}), K_{v}).
\] 
In particular, $\mbf{T}(K)$ is commutative, and we may describe an explicit set of generators. When $v\notin S(K)$, the spherical Hecke algebra $\mbf{T}(\GL_2(F_v),K_v)$ is generated by $T_v=[K_v (^{\varpi_v}\ _{1}) K_v]$, $S_v=[K_v (^{\varpi_v}\ _{\varpi_v}) K_v]$ and $S_v^{-1}$, where we recall that we have fixed uniformizers $\varpi_v$ of $F_v$.  The Iwahori-Hecke algebra $\mbf{T}(\Delta_p,I)$ is the (commutative) subalgebra of $\mbf{T}(\GL_2(F_p),I)$ generated by the indicator functions $[I\delta I]$, where $\delta \in \Sigma^{+}$. Indeed, $\mbf{T}(\Delta_p,I)$ is generated by the operators $U_v=[I_v (^{\varpi_v}\ _{1}) I_v]$, $S_v=[I_v (^{\varpi_v}\ _{\varpi_v}) I_v]$ and $S_v^{-1}$ for  $v\mid p$. 

\medskip

Any left $\Delta_p$-module $N$ will be viewed as a left $\Delta$-module via the natural projection $\Delta \to \Delta_p$. When $Z(K)$ acts trivially, this gives us a local system on $Y_K$, with an action of $\mbf{T}(K)$.

\subsubsection{Group actions on cohomology}  
Let $K=K^p I$ be as above and consider the compact open subgroup $\det(K) \sub \A_F^\infty$. We will define an action (`twisting') of the character group $H_K$ of the ray class group $Cl_K^+ = Cl_{\det(K)}^{+}$ on the cohomology $H^*(Y_K, N)$, for any $E$-local system $N$. This will play a crucial technical role in our paper. Recall that our finite extension $E\subset \Qbar_p$ of $\Q_p$ is assumed to be sufficiently large, in particular we want all $\Qbar_p^{\times}$-valued characters of $Cl^+_K$ to take values in $E^{\times}$.

\medskip

To define the action of $H_K$, first note that if $\chi \in H_K$ then $\chi \circ \det\colon \GL_2(\A_F) \rightarrow E^{\times}$ defines an element of $H^0(Y_K, E)$. Now consider the cup product
\[
\cup \colon H^0(Y_K, E)\otimes H^n(Y_K, N) \rightarrow H^n(Y_K, N).\]
The following proposition is trivial, since cupping with classes in $H^0(Y_K,E)$ is merely multiplication by locally constant functions.

\begin{proposition}
Let $\chi \in H_K$. For any  $E[\Delta]$-module $N$ and any $n\geq 0$, the cup product 
\[
H_K \times H^n(Y_K, N) \rightarrow H^n(Y_K, N),\,\,\, (\chi, c)\mapsto \chi \cdot c := \chi\cup c 
\]
defines a group action of $H_K$ on $H^n(Y_K, N)$.
\end{proposition}

The following lemma explains how the action of $H_K$ interacts with Hecke operators. 

\begin{lemma}\label{twisting and Hecke}
For any $\delta \in \Delta$, with corresponding element $[K\delta K] \in \mbf{T}(K)$, we have 
\[ [K \delta K] \circ \chi= \chi(\det \delta) \chi\circ  [K \delta K]\]
as elements in $\mathrm{End}(H^*(Y_K,N))$. 
\end{lemma}
\begin{proof}
Recall the space $\mf{Y} = \GL_2(F)^\circ \backslash D_\infty \times \GL_2(\A_F^\infty)$ from \S \ref{subsec: symmetric spaces}. For any $N$ as above, the cohomology $H^*(Y_K, N)$ is the cohomology of the complex
\[
C_{\mathrm{ad}}^{\bullet}(K,N)=\mathrm{Hom}_K(C_{\bullet}(\mf{Y}),N)
\]
by descent, where $C_{\bullet}(\mf{Y})$ is the complex of singular chains of $\mf{Y}$. A singular $n$-simplex of $\mf{Y}$ can be written as $(\sigma,g)$, where $g\in \GL_2(\A_F^\infty)$ and $\sigma$ is a singular $n$-simplex of $D_\infty$. Let $\phi \in  C_{\mathrm{ad}}^{\bullet}(K,N)$ and decompose $K\delta K = \bigcup_i \delta_i K$ into right cosets. 
Then $[K\delta K] \phi = \sum_i \delta_i \phi$, where $\delta_i \phi (\sigma,g)=\delta_i \cdot \phi( \sigma, g \delta_i)$. Note that since $\chi(\det(g))=1$ for all $g\in K$, we have $\chi(\det \delta_i)=\chi(\det \delta)$ for all $i$. Then we have
\begin{align*}
\chi \cdot ( [K\delta K]( \phi))(\sigma, g)&= \chi(\det g) ([K\delta K](\phi))(\sigma,g)  \\
&= \chi(\det g) \sum_i \delta_i \cdot \phi(\sigma,g \delta_i).
\end{align*}

and
\begin{align*}
([K\delta K] (\chi \cdot \phi))(\sigma,g)&= \sum_i \delta_i \cdot (\chi(\phi))(\sigma,g \delta_i) \\
&= \sum_i \chi(\det(g \delta_i))  \delta_i \cdot \phi(\sigma, g \delta_i)  \\
&= \chi(\det \delta) \chi(\det g)  \sum_i \delta_i \cdot  \phi(\sigma, g \delta_i).
\end{align*}
Comparing the two expressions gives the lemma. 
\end{proof}

\begin{remark}
In particular, any Hecke operator $[K\delta K]$ with $\delta \in \Delta \cap \SL_2(\A_{F,f})$ commutes with the action of $H_K$.
\end{remark}

We now want to introduce the eigenvariety, and define an action of $H_K$ on it. For that we need to check that the action of $H_K$ respects slope decompositions. We now briefly recall the relevant aspects of the construction of the eigenvariety, cf. \cite{hansen} and \cite[\S 6.2]{bh}. For $v\mid p$ we set $ u_v:= [K (^{\varpi_v}\ _{\varpi_v^{-1}}) K]$ and let
\[
U_p:= \prod_{v\mid p} U^{e_v}_v, \ u_p = \prod_{v\mid p} u^{e_v}_v, \ S_p=\prod S^{e_v}_v.
\]
Note that $u_v = U_v^2 S_v^{-1}$ and similarly $u_p=U^2_p S_p^{-1}$. Given an open affinoid subset $\Omega \sub \mc{W}_k$, we get a corresponding weight $\ka_\Omega : \T(\Zp) \to \oo(\Omega)^\times$. To simplify notation, we will write $\mbf{A}_{\Omega}^{\mbf{s}}$ instead of $\mbf{A}_{\ka_\Omega}^{\mbf{s}}$ for any $\mbf{s}\geq\mbf{s}(\Omega)$, etc. Both $U_p$ and $u_p$ are `controlling operators', inducing compact operators on each term of a (fixed) choice of Borel--Serre complex $C_\bu(K,\mbf{A}^{\mbf{s}}_{\Omega})$. While it is customary to use $U_p$ as a controlling operator, we will use $u_p$ since it is also an $\SL_2$-Hecke operator, which later will make it easier to compare eigenvarieties for $\GL_2$ and $\SL_2$. As $u_p$ acts compactly on each $C_i(K,\mbf{A}^{\mbf{s}}_{\Omega})$ we can form the Fredholm series 
\[
f_{\Omega,u_p}(t)=\det(1-tu_p \mid  C_\ast(K,\mbf{A}^{\mbf{s}}_{\Omega})).
\]
As $\Omega$ runs over all affinoid open subspaces of $\mc{W}_k$ these glue together to a Fredholm series $f_{u}(t) \in \mc{O}(\mc{W}_k)\{\{t\}\}$ and we let $\mc{Z} \sub \mc{W}_k \times \mb{A}^1$ be the corresponding Fredholm hypersurface. Let $(\Omega, h)$ be a slope-adapted pair for $u_p$ (we refer to \cite[\S 6.2]{bh} or \cite[\S 4.1]{hansen} for this notion). In particular, this means that we have a slope $\leq h$ decomposition 
\[
H^*(Y_K, \mbf{D}_{\Omega}^\mbf{s}))= H^*(Y_K, \mbf{D}_{\Omega}^\mbf{s}))_{\leq h} \oplus H^*(Y_K, \mbf{D}_{\Omega}^\mbf{s}))_{>h} 
\]
where the term $H^*(Y_K, \mbf{D}_{\Omega}^\mbf{s}))_{\leq h}$ is independent of $\mbf{s}$. In light of this independence, we will simply write $H^*(Y_K, \D_{\Omega})_{\leq h}$ for it (cf. \cite[\S 2.3.12]{urban}). We get an induced morphism 
\[
\psi_{\Omega,h} \colon \mbf{T}(K)\rightarrow \mathrm{End}_{\oo(\Omega)}(H^*(Y_K, \D_{\Omega})_{\leq h}).
\]
As the pairs $(\Omega,h)$ range over slope-adapted pairs, the $H^*(Y_K, \D_{\Omega})_{\leq h}$ glue together to a graded $\mc{O}_{\mc{Z}}$-module, which we denote by $\wt{\mc{M}}= \wt{\mc{M}}_{K^p}$, and the $\psi_{\Omega,h}$ glue to a morphism $\psi \colon  \mbf{T}(K)\rightarrow \mathrm{End}_{\oo_\mc{Z}}(\wt{\mc{M}})$. The eigenvariety $\wt{\E} =\wt{\E}_{K^p}$ is defined to be the relative adic spectrum over $\mc{Z}$ of the coherent $\oo_{\mc{Z}}$-subalgebra of $\mathrm{End}_{\oo_\mc{Z}}(\wt{\mc{M}})$ generated by the image of $\psi$, and it is independent of the choice of controlling operator by \cite[Proposition 3.4.4]{JN2}. Phrased differently, $\wt{\E}$ is obtain by gluing the `local pieces' $\wt{\E}_{\Omega,h}=\Spa(\mbf{T}_{\Omega,h})$, with $\mbf{T}_{\Omega,h} := \mathrm{Im}(\psi_{\Omega,h} \otimes \oo(\Omega))$, over $\mc{Z}$, where $(\Omega,h)$ ranges over slope-adapted pairs. The eigenvariety depends on $K^p$ but, as indicated, we will suppress this dependence from the notation unless we need to deal with multiple tame levels.  

\medskip

The action of $H_K$ on cohomology induces an action on the eigenvariety, which we will now describe on the local pieces.
The $\oo(\Omega)$-module $\mathrm{End}_{\oo(\Omega)}(H^*(Y_K, N))$ has an $H_K$-module structure defined by $(h \cdot \phi)(m):= h \phi(h^{-1}m)$ for $\phi \in \mathrm{End}_{\oo(\Omega)}(H^*(Y_K, N))$, $h \in H_K$ and $m \in H^*(Y_K, N)$.

\begin{proposition} 
\begin{enumerate}
\item The action of $H_K$ on $H^*(Y_K, \mbf{D}_{\Omega}^{\mbf{s}})$ commutes with changing $\mbf{s}$, and leaves $H^*(Y_K, \mbf{D}_{\Omega}^{\mbf{s}})_{\leq h}$ invariant.
\item The induced action of $H_K$ on $\mathrm{End}_{\oo(\Omega)}(H^*(Y_K, \mbf{D}_{\Omega}^{\mbf{s}})_{\leq h})$ sends $\mbf{T}_{\Omega,h}$ to itself.
\item If $\phi : \mbf{T}_{\Omega,h} \to \ol{\Q}_p$ is a homomorphism and $\chi \in H_K$, then $\chi \circ \phi$ sends $[K\delta K]$ to $\chi(\det \delta)[K\delta K]$.
\end{enumerate}
\end{proposition}
\begin{proof} Commutation with changing $\mbf{s}$ is clear from the definitions (as, more generally, the action of $H_K$ commutes with morphisms of local systems), and since the action of $H_K$ commutes with $u_p$ it preserves the slope decomposition by \cite[Proposition 2.3.2(a)]{hansen}. This proves (1). Parts (2) and (3) then follow directly from Lemma \ref{twisting and Hecke}.
\end{proof}

In particular, part (3) shows that the action of $H_K$ has the desired twisting effect on systems of Hecke eigenvalues.

\subsection{Galois representations and the geometry of eigenvarieties}\label{subsec: Galois reps}

In this subsection we discuss some further aspects of $\wt{\E}$ and its geometry, using Galois representations as our main tool. Aside from the Hecke algebra $\T(K)=\T(\Delta,K)$, we will also need the subalgebra $\T(K^p) \sub \T(K)$ where we remove all Hecke operators at places dividing $p$. If $x\in \wt{\E}(\ol{\Q}_p)$, then $x$ corresponds to a homomorphism
\[
\phi_x : \T(K) \to \ol{\Q}_p,
\]
which lands inside $\ol{\Z}_p$, by (the proof of) \cite[Lemma 6.5.4]{bh}. The kernel of $\phi_x$ is a height $1$ prime ideal that we will denote by $\p_x$. We may restrict $\phi_x$ to $\T(K^p)$ and then reduce it modulo the maximal ideal of $\ol{\Z}_p$ to obtain a
homomorphism $\T(K^p) \to \ol{\F}_p$. The corresponding maximal ideal will be denoted by $\m_x$. We remark that this notation is in conflict with the notation in \cite{bh}, where our $\p_x$ is denoted by $\m_x$, but we hope that this won't cause any confusion. We will also sometimes write $\phi_x$ for the induced homomorphism $\oo(\wt{\E}) \to \ol{\Q}_p$.

\medskip

By \cite[Theorem 5.4.5]{JN1}, there exists a continuous $2$-dimensional pseudocharacter $T : G_F \to \oo^+(\wt{\E}^{red})$ characterised by
\[
T(Frob_v) = T_v
\]
for all $v\notin S(K)$, where $\wt{\E}^{red}$ denotes the nilreduction of $\wt{\E}$. Its determinant $D(g) := (T(g)^2 - T(g^2))/2$ is the continuous character $G_F \to \oo^+(\wt{\E}^{red})^\times$ characterised by
\[
D(Frob_v)= q_v S_v
\]
for $v\notin S(K)$. For every $x\in \wt{\E}(\ol{\Q}_p)$, $T_x = \phi_x \circ T$ arises from a semisimple continuous Galois representation $\rho_x : G_F \to \GL_2(\ol{\Q}_p)$. When $x$ arises from a refined automorphic representation $(\pi,\alpha)$ (as we recall below), then $\rho_x \cong \rho_{\pi}$ agrees with the Galois representation attached to $\pi$ as recalled in (\ref{eq:classrep}). We write $\ol{\rho}_x$ for the (semisimple) reduction of $\rho_x$. A first consequence of the the existence of $T$ is that $\wt{\E}^{red}$ doesn't change if we remove a finite number of unramified Hecke operators away from $p$, while keeping the same tame level.

\begin{lemma}\label{changing Hecke operators}
Let $K_1^p \sub K^p$ be a compact open subgroup and set $K_1 = K^p_1 I$. Let $(\Omega,h)$ be a slope adapted pair for $\wt{\E}_{K^p}$, and let $\wt{\E}_{\Omega,h} = \Spa \T_{\Omega,h}$ be the corresponding local piece. Then the natural map $\T(K_1) \otimes \oo(\Omega) \to \T_{\Omega,h}^{red}$ is surjective.
\end{lemma} 

\begin{proof}
Let $A \sub \T_{\Omega,h}^{red}$ be the image of $\T(K_1) \otimes \oo(\Omega)$. Since $\T_{\Omega,h}^{red}$ is a finite $\oo(\Omega)$-module, $A$ is a closed subalgebra. Let $v \in S(K_1) \setminus S(K)$. Consider the pseudocharacter $T_\Omega : G_F \to \T_{\Omega,h}^{red}$. By Chebotarev and the fact that $A$ is closed, $T_v$ contained in $A$, which proves the lemma. 
\end{proof}

Next, we recall the following variant of \cite[Lemma 6.5.6]{bh}, which has the same proof. By definition, a good neighbourhood of a point $x\in \wt{\E}(\ol{\Q}_p)$ is a neighbourhood $U$ of $x$ which is a connected component of $\wt{\E}_{\Omega,h}$, for some slope-adapted pair $(\Omega,h)$.

\begin{lemma}\label{good nhood}
Let $x\in \wt{\E}(\olQp)$. Assume that $\rho_x$ is defined over the coefficient field $E$ and is absolutely irreducible. Then there exists a good neighbourbood $U$ of $x$ and a continuous Galois representation $\rho_U : G_F \to \GL_2(\oo(U))$ such that $\rho_U$ restricts to $\rho_y$ for every $y \in U(\olQp)$. Moreover, $\rho_y$ is absolutely irreducible for all $y\in U(\olQp)$. 
\end{lemma}

Let us now briefly recall the middle degree eigenvariety, as constructed\footnote{The eigenvarieties $\wt{\E}$ and $\wt{\E}_{mid}$ in \cite{bh} are constructed using compactly supported cohomology instead of usual cohomology. To construct these eigenvarieties for usual cohomology, replace the use of Borel--Moore homology in \cite{bh} by usual homology.} in \cite[\S 6.4]{bh}. The geometry of the whole $\wt{\E}$ is difficult to control. There is a Zariski open subspace $\wt{\E}_{mid} \sub \wt{\E}$, called the middle degree eigenvariety, which has better geometric properties. In particular, $\wt{\E}_{mid}$ is equidimensional of dimension $d$ and reduced \cite[Proposition 6.4.6, Theorem 6.4.8]{bh}. By \cite[Proposition 6.4.3]{bh}, $\wt{\E}_{mid}$ is characterized by the property that, for $x \in \wt{\E}(\olQp)$ with weight $\ka$, $x\in \wt{\E}_{mid}(\olQp)$ if and only $H^i(Y_K, \mbf{D}_\ka^{\mbf{s}})_{\p_x} = 0$ for all $i \neq d$ (and $\mbf{s}$ sufficiently large). This condition is not easy to check. When $F=\Q$ one has $\wt{\E} = \wt{\E}_{mid}$ (see the beginning of \cite[\S 6.1]{JN1}). In general one now has the following result, using the powerful vanishing result of \cite{caraiani-tamiozzo}.

\medskip

\begin{theorem}\label{consequence of ct}
Let $\ol{\phi} : T(K^p) \to \ol{\F}_p$ be a system of Hecke eigenvalues with kernel $\m$ which has an associated Galois representation $\ol{\rho} : G_F \to \GL_2(\ol{\F}_p)$. Assume that $\ol{\rho}$ is irreducible. Additionally, assume $\ol{\rho}$ is generic in the sense of \cite[Definition 7.1.3(3)]{caraiani-tamiozzo}: There exists a rational prime $\ell \nmid 2p$ such that:
\begin{enumerate}
\item $\ell$ is totally split in $F$;

\item If $v$ is any place above $\ell$ in $F$, then $\ol{\rho}$ is unramified at $v$ and the eigenvalues of $\ol{\rho}(Frob_v)$ have ratios different from $\ell$ and $\ell^{-1}$. 
\end{enumerate}
Then $H^i(Y_K, \mathscr{D}_\ka)_{\leq h, \m} = 0$ unless $i=d$, for any $\ka \in \mc{W}_k(\olqp)$ and any $h$. Hence, if $x\in \wt{\E}(\olqp)$ is a point with $\ol{\rho}_x$ generic and irreducible, then a fortiori $H^i(Y_K, \mathscr{D}_\lambda)_{\p_x} = 0$ unless $i=d$, and $x$ lies in $\wt{\E}_{mid}$.
\end{theorem}

\begin{proof}
The proof of \cite[Theorem B.0.1]{bh} goes through to give this result, using \cite[Theorem 7.1.6]{caraiani-tamiozzo} instead of the weaker (thanks to \cite[Lemma 7.1.8]{caraiani-tamiozzo}) special case used in the proof of \cite[Theorem B.0.2]{bh}.
\end{proof}

As the reduction $\ol{\rho}_x$ is constant along connected components, we see that whole components of $\wt{\E}$ are contained in $\wt{\E}_{mid}$. 

Next, we discuss `classical points', i.e., how to realize refined automorphic representations $(\pi,\alpha)$ as points on $\wt{\E}$. For this (and other things later on), it will be convenient to discuss a different way of normalizing the $\Delta_p$-action on $\mbf{D}_k^{\mbf{s}}$ for cohomological weights $k =(k_1,k_2)$. Let $g = \left(^a_c \ ^b_d \right) \in \Delta_p$ and let $f \in \mbf{A}^{\mbf{s}}(\oo_p, E)$. We define a new action $\star$ of $\Delta_p$ on $\mbf{A}^{\mbf{s}}(\oo_p,E)$ by
\[
(f \star g)(z) = k_1 k_2^{-1}(cz+d)k_2(\det(g))f(gz),
\]
In other words, we remove the normalizations that were forced upon us by the fact that general weights are only defined on $\oo_p^\times$ and not the whole $F_p^\times$.  We write $\mbf{A}_{k,\star}^{\mbf{s}}$ for $\mbf{A}^{\mbf{s}}(\oo_p,E)$ with the right action $\star$, and $\mbf{D}_{k,\star}^{\mbf{s}}$ for $\mbf{D}^{\mbf{s}}(\oo_p,E)$ with the dual left action (also denoted by $\star$). Note that the action of $I$ is unchanged, so we have an equality $H^\ast(Y_K, \mbf{D}_{k,\star}^{\mbf{s}}) = H^\ast(Y_K,\mbf{D}_{k}^{\mbf{s}})$ which respects the action of the Hecke operators away from $p$. The Hecke operators at $p$ are scaled by an easily computable factor. Taking (co)limits, we obtain $\mc{A}_{k,\star}$ and $\D_{k,\star}$. The benefit of this new normalization lies in the comparison with the local system~$\ms{L}_k$. There is a $\Delta_p$-equivariant surjective map
\[
I_k : \D_{k,\star} \to \ms{L}_k,
\]
called the integration map, which is defined in \cite[Definition 5.4.1]{bh} (see also \cite[\S 3.2]{hansen}). The map $I_k$ is the first step $C_0 \to \ms{L}_k$ of a resolution $C_\bu \to \ms{L}_k$, with
\[
C_i = \bigoplus_{\ell(w)=i} \D_{w \cdot k, \star}
\]
of $\Delta_p$-modules. This is the locally analytic BGG resolution; the sum goes over elements of the Weyl group of length $i$. All that is important for us is that all terms in $C_\bu$ are direct sums of things of the form $\D_{k^\prime,\star}$ for some algebraic weights $k^\prime$, so we will not recall the definitions of the maps or of the dot-action $w \cdot k$ of the Weyl group, instead referring to \cite[\S 3.3]{urban} for more details. We then have the following result.

\begin{proposition}\label{integration map surjective}
Let $\ol{\phi} : T(K^p) \to \ol{\F}_p$ be a system of Hecke eigenvalues with kernel $\m$ which has an associated Galois representation $\ol{\rho} : G_F \to \GL_2(\ol{\F}_p)$. Assume that $\ol{\rho}$ is irreducible and generic. Then $I_k$ induces a surjection $H^d(Y_K,\D_{k,\star})_{\leq h, \m} \to H^d(Y_K,\ms{L}_k)_{\leq h, \m}$.
\end{proposition}

\begin{proof}
This follows directly by applying Theorem \ref{consequence of ct} to the resolution $C_\bu \to \ms{L}_k$. 
\end{proof}

Let $(\pi,\alpha)$ be a refined automorphic representation of weight $k$.  Assume that $(\pi^{\infty})^K\neq0$ (we say $\pi$ has level $K$). We get a system of Hecke eigenvalues $\phi_\pi : \T(K) \to \olqp$ as follows: For $v \nmid p$ such that $K_v=\GL_2(\oo_{F_v})$, we let $\phi_\pi(T_v)$ and $\phi_\pi(S_v)$ be the eigenvalues of $T_v$ and $S_v$ acting on $\pi_v^{K_v}$, respectively. For $v \mid p$, we define $\phi_\pi(U_v) = k_{2,v}(\varpi_v^{-1})\alpha_v$ and we let $\phi_\pi(S_v)$ be $k_{1,v}k_{2,v}(\varpi_v^{-1})$ times the scalar that $S_v$ acts on $\pi_v^{I_v}$ by. The normalizing factors account for the difference between the actions of $\Delta_p$ on the modules $\mc{D}_k$ and $\mc{D}_{k,\star}$.

\begin{corollary}\label{defining classical points}
Let $(\pi,\alpha)$ be a refined automorphic representation of $\GL_2(\A_F)$ of weight $k$ and level $K$ and let $\phi_\pi$ be as above. Assume that $\ol{\rho}_\pi$ is irreducible and generic. Then $\phi_\pi$ defines a point on $\wt{\E}_{mid}$ of weight $k$.
\end{corollary}

\begin{proof}
This follows directly from Proposition \ref{integration map surjective} and Theorem \ref{consequence of ct}.
\end{proof}

From now on, we will write $x=(\pi,\alpha)$ for the point constructed from $(\pi,\alpha)$ by Corollary \ref{defining classical points}, and we will refer to these points as `classical'.

\begin{remark}
All $(\pi,\alpha)$ of weight $k$ and level $K$ (in the sense that they contribute to $H^d(Y_K,\ms{L}_k)$) do define points on $\wt{\E}$, but without some result like the corollary above we do not know in general if these points have weight $k$ (in the sense that it is not clear that they contribute to $H^\ast(Y_K,\mc{D}_k)$).
\end{remark}

From now until the end of this subsection, let $x=(\pi,\alpha)$ be a point on $\wt{\E}$ coming from a refined automorphic representation such that $\ol{\rho}_x$ is irreducible and generic\footnote{These assumptions are not strictly speaking needed for all results that follow.}. The Galois representation $\rho_x=\rho_\pi$ satisfies local-global compatibility with $\pi$ at all places, by the work of many people (see \cite[\S 6.5]{bh} for more details). A consequence is the following invariance result for the local geometry of the eigenvariety at classical points, which will be important to us. We thank John Bergdall for discussions related to this result. For the Coleman--Mazur eigencurve, see \cite[Proposition 2.6]{bergdall-upper-bounds}. The result below can be deduced from \cite[Theorem 3.2]{saha}, as pointed out to us by Bergdall, but we give a slightly different phrasing of the proof since we will use the same argument again in Proposition \ref{constancy of classical subspace} below.

\begin{proposition}\label{changing tame level}
Let $x=(\pi,\alpha)$ be as above, and let $\n \sub \oo_F$ be the prime-to-$p$ conductor of $\pi$. Let $K^p \sub K^p_1(\n)$ be any compact open subgroup. Then the natural map
\[
\wt{\E}_{K_1^p(\n)} \to \wt{\E}_{K^p}
\]
is a local isomorphism at $x$.
\end{proposition}

\begin{proof} As recalled above, the middle degree eigenvarieties $\wt{\E}_{K_1^p(\n),mid}$ and $\wt{\E}_{K^p,mid}$ are reduced and equidimensional of dimension $d$. Moreover, the points on $\wt{\E}_{K_1^p(\n)}$ and $\wt{\E}_{K^p}$ defined by $x=(\pi,\alpha)$ lie on $\wt{\E}_{K_1^p(\n),mid}$ and $\wt{\E}_{K^p,mid}$. By Lemma \ref{changing Hecke operators} the natural map $\wt{\E}_{K_1^p(\n)} \to \wt{\E}_{K^p}$ induces a closed immersion $\wt{\E}_{K_1^p(\n)}^{red} \to \wt{\E}_{K^p}^{red}$, and hence a closed immersion $\wt{\E}_{K_1^p(\n),mid} \to \wt{\E}_{K^p}^{red}$ with $x$ inside $\wt{\E}_{K^p,mid}$. It therefore suffices to prove that any component of $\wt{\E}_{K^p,mid}$ which contains $x$ has a Zariski dense set of classical cuspidal regular points of conductor $\leq \n$. Since $\wt{\E}_{K^p,mid}$ is equidimensional of dimension $d$, any component $C$ of $\wt{\E}_{K^p}$ passing through $x$ has a Zariski dense set of points coming from refined automorphic representations $(\pi^\prime, \alpha^\prime)$, and we may take this Zariski dense set to be contained in the intersection of $V = C \cap U$ of $C$ with a good neighbourhood $U$ as in Lemma \ref{good nhood}, where we have a Galois representation $\rho_V $. Fix a prime $v \nmid p \infty$. The Weil--Deligne representation attached to $\rho_V |_{G_{F,v}}$ (by Grothendieck's monodromy theorem in families, see e.g. \cite[\S 7.8.3]{bellaiche-chenevier}) defines a morphism $V \to X_{WD}$ to the moduli space of $2$-dimensional Weil--Deligne representations over $E$, as constructed e.g. in \cite{hellmann2021derived}. At pure points $X_{WD}$ is smooth, so using local-global compatibility and the structure of the components of $X_{WD}$ one sees that, in fact, $\pi^\prime_v |_{\GL_2(\oo_v)}$ is constant on classical points of $V$. In particular, the conductor is constant as desired. 
\end{proof}

Our next goal is to prove that certain twists of classical points are classical. To do this, we first discuss local-global compatibility for the determinant $D$ at arbitrary points of cohomological weight $k=(k_1,k_2)$. Before stating the result, we recall that an algebraic character $\chi : \A_F^\times \to \ol{E}^\times$ of weight $w\in \Z$ is a character that has an open kernel and satisfies
\[
\chi(a) = \prod_\tau \tau(a)^w
\]
for all $a \in F^\times$, where $\tau$ runs through the embeddings of $F$ into $\overline{E}$. Algebraic characters correspond bijectively to geometric characters of $G_F$. The correspondence can be pinned down by the equality of the values of Frobenii at almost all places, and the weight $w$ corresponds to the Hodge--Tate weights (since $F$ is totally real, the Hodge--Tate weights of geometric characters are all equal). Moreover, the geometric character is crystalline at places above $p$ if and only the algebraic character is unramified at all places above $p$, i.e. $\chi|_{\oo_{F,v}^\times}$ is trivial for all $v|p$.

\begin{proposition}\label{actionofZ}
Let $k$ be a cohomological weight and let $f \in H^\ast(Y_K, \mbf{D}_{k,\star}^{\mbf{s}})$ be an eigenvector for $\mbf{T}(K)$. Then the center $Z(\A_F)$ acts on $f$ via an algebraic character of weight $w$ on $\alpha$ which is unramified at all places above $p$ (here $Z(\A_F^\infty)$ acts via $\mbf{T}(K)$, and we define $Z(F_\infty)$ to act trivially).
\end{proposition}

\begin{proof} 
By the definitions the open subgroup $(Z(\A_F^\infty) \cap K)Z(F_\infty)$ acts trivially and $Z(F_p) \cap I = Z(\oo_p)$, so it remains to compute the action of $Z(F)$. The action of $z\in Z(\A_F)$ on $f \in \Hom(C_\bu(\mf{Y}), \mbf{D}_{k,\star}^{\mbf{s}})$ is, by the definitions, given by
\[
(zf)(\sigma) = z_p \star f(\sigma z).
\]
If $z \in Z(F)$, then $z$ acts trivially on $\mf{Y}$, so $\sigma z = \sigma $ for all $\sigma \in C_\bu(\mf{Y})$. Moreover, from the definitions we see that $z_p$ acts by $k_1 k_2^{-1}(z_p)k_2(z_p^2) = k_1 k_2 (z_p)$ on $\mbf{D}_{k,\star}^{\mbf{s}}$, which is equal to $\prod_\tau \tau(z)^w$. The proposition follows.
\end{proof}

\begin{corollary}\label{local-global for determinant}
Let $x\in \wt{\E}(\olqp)$ be a point of cohomological weight $k$. Then the determinant $D_x$ is crystalline of parallel Hodge-Tate weight $w+1$.
\end{corollary}

\begin{proof}
From the identity $D(Frob_v)= q_v S_v$ for $v\notin S(K)$ we see that $D$ is the cyclotomic character times the geometric character corresponding to the algebraic character in Proposition \ref{actionofZ} (for a choice $f$ of eigenvector corresponding to the system of eigenvalues $x$). The corollary follows.
\end{proof}

For the rest of this subsection let us fix a classical point $x=(\pi,\alpha)$ of weight $k$ as above. The Hecke algebra $\T(K)$ has a subalgebra $\s(K)$ consisting of Hecke operators $[K\delta K]$ with $\det(\delta)=1$. This subalgebra is generated by the operators $u_v$ at $v\mid p$, and the operators $t_v = [K \left( \begin{smallmatrix} \varpi_v & 0 \\ 0 & \varpi_v^{-1} \end{smallmatrix} \right) K]$ for $v\notin S(K)$. Let us now consider a point $z \in \wt{\E}(\olqp)$ of weight $k$ and such that
\[
\phi_x(t_v) = \phi_z(t_v)
\]
for $v\notin S(K)$. Equivalently, $\ad^0 \rho_x \cong \ad^0 \rho_z$ as $3$-dimensional representations. Our goal is to show that there exists a finite order Hecke character $\chi$, unramified at all $v\mid p$, such that $z = (\pi \otimes (\chi \circ \det), \alpha^\prime)$ for some refinement $\alpha^\prime$ of $\pi \otimes (\chi \circ \det)$, under an additional assumption on $x$. We start by noting that, since $\ad^0 \rho_x \cong \ad^0 \rho_z$, the projective representations of $\rho_x$ and $\rho_z$ are isomorphic (by Proposition \ref{generation}) and hence there exists a character $\chi : G_F \to \olqp^\times$ such that $\rho_z \cong \rho_x \otimes \chi$. Taking determinants, we see that $\chi^2 = (\det \rho_z)(\det \rho_x)^{-1}$. By Corollary \ref{local-global for determinant}, $\chi^2$ is a crystalline character of Hodge--Tate weight $0$, and hence of finite order. It follows that $\chi$ is of finite order and hence de Rham. To go further, however, we need to discuss the $p$-adic Hodge-theoretic properties of $\rho_x$ and $\rho_z$, which are currently only known over $\wt{\E}_{mid}$. In particular, recall that, since $\pi_p^I\neq 0$, $\rho_x$ is semistable by local-global compatibility. To say more about $\rho_z$, we need the following:

\begin{proposition}\label{twists are on E-mid}
The reduction $\ol{\rho}_z$ is irreducible and generic. In particular, the point $z$ lies on $\wt{\E}_{mid}$.
\end{proposition}

\begin{proof}
Irreducibility and genericity are trivially preserved under twisting, proving the first statement, and the second then follows from Theorem \ref{consequence of ct}.
\end{proof}

By Lemma \ref{good nhood} we may choose a good neighbourhood $U$ of $z$ over which the pseudocharacter $T$ on $\wt{\E}$ comes from a Galois representation $\rho_U$, which additionally is absolutely irreducible at all $z^\prime \in U(\olqp)$. The key $p$-adic Hodge-theoretic property of the representations $\rho_{z^\prime, v}$, for $v \mid p$, is that they are trianguline. This is a consequence of the global triangulation results of \cite{kpx}; we now recall what we need. We will need some notation for rank $1$ $(\varphi,\Gamma)$-modules; see \cite[\S 6.2]{kpx} for more details. Let $v\mid p$ and let $L$ be a (large) finite extension of $\Qp$ containing all embeddings of $F_v$ into $\olqp$. We write $\mc{R}_{F_v,L}$ for the Robba ring of $F_v$ with coefficients in $L$. The rank $1$ $(\varphi, \Gamma_{F_v})$-modules over $\mc{R}_{F_v,L}$ correspond bijectively to continuous characters
\[
\delta : F_v^\times \to L^\times 
\]
and we write $\mc{R}_{F_v,L}(\delta)$ for the corresponding $(\varphi,\Gamma_{F_v})$-module. $\mc{R}_{F_v,L}(\delta)$ is crystalline if and only if $\delta|_{\oo_v^\times}$ is of the form
\[
a^k : a \mapsto a^{k} := \prod_{\sigma \in \Sigma_v} \sigma(a)^{k_\sigma}
\]
for an integer tuple $k=(k_\sigma)_\sigma$; in this case the $\sigma$-Hodge--Tate weight of $\mc{R}_{F_v,L}(\delta)$ is $-k_\sigma$ (recall that with our normalizations all the Hodge--Tate weights of the cyclotomic character are $-1$). If, more generally, $A$ is an affinoid $\Qp$-algebra and $\delta : F_v^\times \to A^\times$ is a continuous character, then one may construct a family $\mc{R}_{\Spa(A)}(\delta)$ of rank $1$ $(\varphi,\Gamma_{F_v})$-modules over $\Spa(A)$, which specializes to the pointwise construction for any $A \to L$. Let us now briefly discuss some aspects of triangulations of semistable representations that we will need. First, let $D$ be an arbitrary rank $2$ $(\varphi,\Gamma_{F_v})$-module over $\mc{R}_{F_v,L}$. A triangulation of $D$ is a short exact sequence
\[
0 \to \mc{R}_{F_v,L}(\delta_1) \to D \to \mc{R}_{F_v,L}(\delta_2) \to 0
\]
of $(\varphi,\Gamma_{F_v})$-modules; $\delta_1$ and $\delta_2$ are called the (ordered) parameters of the triangulation. Now, let $\rho : G_{F_v} \to \GL_2(L)$ be a semistable representation with corresponding $(\varphi,\Gamma_{F_v})$-module $D_{rig}(\rho)$ and filtered $(\varphi,N)$-module $D_{st}(\rho)$. By Berger's equivalence \cite[Th\'eor\`eme A]{berger}, triangulations of $D_{rig}(\rho)$ correspond to short exact sequences
\[
0 \to D_1 \to D_{st}(\rho) \to D_2 \to 0
\]
of filtered $(\varphi,N)$-modules, where $D_1$ and $D_2$ have rank $1$. Assume additionally that the $\varphi^{f_v}$ has no repeated eigenvalue on $D_{crys}(\rho)$ (in particular, $\rho$ is allowed to be semistable non-crystalline) and that the Hodge--Tate weights are $k_1 < k_2$ (i.e. they are distinct) --- we will now describe the possible triangulations of $D_{st}(\rho)$ (all of them are strict in the sense of \cite[Definition 6.3.1]{kpx}). The local Langlands correspondence assigns a smooth admissible $\GL_2(F_v)$-representation to the $(\varphi,N)$-module $D_{st}(\rho)$ (forgetting the filtration) which we denote by $\pi(\rho)$. The $U_v$-eigenvalues on $\pi(\rho)^{I_v}$ are equal to the $\varphi^{f_v}$-eigenvalues on $D_{crys}(\rho)$. We will have one or two triangulations, depending on whether $D_{crys}(\rho)$ has rank one or two. Let $\alpha$ be an eigenvalue of $\varphi^{f_v}$ on $D_{crys}(\rho)$. Following \cite[Example A.2.1]{bh}, one gets an embedding
\[
\mc{R}_{F_v,L}(\delta_1^\prime) \to D_{rig}(\rho),
\]
where $\delta_1^\prime = a^{-k_1}ur(\alpha)$; here $ur(\alpha)$ denotes the character on $F_v^\times$ which is trivial on $\oo_v^\times$ and sends $\varpi_v$ to $\alpha$. The image may not be saturated, so this may not give rise to a triangulation, but its saturation is of the form $\mc{R}_{F_v,L}(\delta_1)$ with $\delta_1 = a^{-s}\delta_1^\prime$, with $s_\sigma = 0$ or $s_\sigma = k_{2,\sigma} - k_{1,\sigma}$. This gives us a triangulation
\[
0 \to \mc{R}_{F_v,L}(\delta_1) \to D_{rig}(\rho) \to \mc{R}_{F_v,L}(\delta_2) \to 0,
\]
where $\delta_2 = \det(\rho)\delta_1^{-1}$. If $s_\sigma = 0$, we say that the triangulation is non-critical at $\sigma$, otherwise we say that it is critical at $\sigma$. If the valuation of $\alpha$ is small compared to the Hodge--Tate weights, then the triangulation will be non-critical at all $\sigma$; cf. \cite[Remark 6.3.9]{bh}.

\begin{remark}\label{rem: non-crit vs crit in split case}
Assume that $\rho = \chi_1 \oplus \chi_2$ is the sum of two characters. Let $\alpha$ and $\beta$ be the $\varphi^{f_v}$-eigenvalues on $D_{rig}(\chi_1)$ and $D_{rig}(\chi_2)$, respectively. and let $k_{\chi_i,\sigma}$ be the $\sigma$-Hodge--Tate weights of $\chi_i$, $i=1,2$. Assume, as above, that $k_{\chi_1,\sigma} \neq k_{\chi_2,\sigma}$ for all $\sigma$. Then $\alpha$ is non-critical at $\sigma$ if and only if $k_{\chi_1,\sigma} < k_{\chi_2,\sigma}$, and similarly for $\beta$. Phrased differently, for a fixed $\sigma$, either $\alpha$ is non-critical and $\beta$ is critical at $\sigma$, or vice versa.
\end{remark}

Now we return to our point $z\in \wt{\E}_{mid}(\olqp)$, of cohomological weight $k$, and its neighbourhood $U$. Assume that $z^\prime = (\pi_{z^\prime}, \alpha_{z^\prime}) \in U(\olqp)$ is a classical point of weight $k_{z^\prime}$, and assume in addition that $z^\prime$ satisfies the very strong form of `numerical non-criticality' given in \cite[Remark 6.3.9]{bh} (in particular, the assumptions imply that $\pi_{z^\prime,v}$ is an unramified principal series for all $v\mid p$). Such points are Zariski dense in $U$ (by the proof of \cite[Proposition 6.4.6(4)]{bh} -- see the remarks in the proof of \cite[Proposition 6.5.8(2)]{bh}). For $v\mid p$, $\rho_{z^\prime,v}$ has a triangulation
\[
0 \to \mc{R}_{F_v,L}(\delta_{2,z^\prime}) \to D_{rig}(\rho_{z^\prime,v}) \to \mc{R}_{F_v,L}(\delta_{1,z^\prime}) \to 0
\] 
by the discussion above and the assumptions on $\alpha_{z^\prime}$, where 
\[
\delta_{2,z^\prime} = a^{-k_{z^\prime,v,2}}ur(\alpha_{z^\prime,v}) \,\,\, \text{and}\,\,\, \delta_{1,z^\prime} = a^{- k_{z^\prime,v,1} -1}ur(\beta_{z^\prime,v}),
\]
where $\beta_{z^\prime,v}$ is the other $U_v$-eigenvalue on $(\pi_{z^\prime,v})^{I_v}$. The parameters $\delta_{i,z^\prime}$ extend to families $\delta_i$ over all of $U$: $\delta_2$ is given by $\delta_2(\varpi_v)=U_v \in \oo(U)^\times$ and $\delta_2|_{\oo_v^\times} = a^{-\ka_{U,v,2}}$, and $\mc{R}_{F_v,L}(\delta_1) = D_{rig}(\det \rho_U) \otimes \mc{R}_{F_v,L}(\delta_2^{-1})$. In particular, $\delta_1|_{\oo_v^\times} = a^{- \ka_{U,v,1} -1}$. With these preparations, we can now prove that $z$ is a classical point.

\begin{proposition}\label{twists are classical}
Assume that $\varphi^{f_v}$ acting $D_{crys}(\rho_{x,v})$ does not have any repeated eigenvalues for all $v\mid p$. Then, with notation and assumptions as above, $\chi$ is unramified and $z=(\pi \otimes (\chi \circ \det), \alpha^\prime)$ for some refinement $\alpha^\prime$ of $\pi \otimes (\chi \circ \det)$.
\end{proposition}

\begin{proof}
Fix $v \mid p$; we need to prove that $\chi$ is unramified at $v$ and that $\alpha_v^\prime := \varpi_v^{k_2}U_v(z)$ is an eigenvalue of $\varphi^{f_v}$ on $D_{crys}(\rho_{z,v})$. If $U$ has multiple irreducible components going through $z$, replace $U$ by one of these components. We restrict all constructions above to this component, and keep all the notation; in particular, the component will still be called $U$. Then, using the classical points described above, $D_{rig}(\rho_{U,v})$ is a densely pointwise strictly trianguline $(\varphi,\Gamma_{F_v})$-module on $U$ of rank $2$ in the sense of \cite[Definition 6.3.2]{kpx}. By \cite[Corollary 6.3.10]{kpx} (and possibly enlarging $E$), there is an affinoid $U^\prime$ and a map $f : U^\prime \to U$ with an $E$-point $w$ such that $f(w)=z$, and such that $D=f^\ast D_{rig}(\rho_U)$ has a (global) triangulation
\begin{equation}\label{globaltriangulation}
0 \to \mc{R}_{U^\prime}(\delta_2) \to D \to Q \to 0,
\end{equation}
where $Q$ embeds into $\mc{R}_{U^\prime}(\delta_1)$ with cokernel $N$ killed by a power $t^n$ of Fontaine's ``$p$-adic $2\pi i$'' $t$. Write $E_w$ for the field $E$ considered as the residue field of $U^\prime$ at $w$. From the $\Tor$-long exact sequence for the short exact sequence
\[
0 \to Q \to \mc{R}_{U^\prime}(\delta_1) \to N \to 0
\]
we see that $\Tor_1^{\oo(U^\prime)}(E_w,Q) = \Tor_2^{\oo(U^\prime)}(E_w,N)$, and hence is killed by $t^n$, and that
\[
0 \to \Tor_1^{\oo(U^\prime)}(E_w,N) \to Q_w \to \mc{R}_{F_v,E}(\delta_{1,z}) \to N_w \to 0
\] 
is exact, so the free quotient of $Q_w$ embeds into $\mc{R}_{F_v,E}(\delta_{1,z})$. That $\Tor_1^{\oo(U^\prime)}(E_w,Q)$ is killed by $t^n$ implies that the specialization
\begin{equation}\label{quasitriangulation}
0 \to \mc{R}_{F_v,E}(\delta_{2,z}) \to D_{rig}(\rho_{z,v}) \to Q_{w} \to 0 
\end{equation}
of (\ref{globaltriangulation}) to $w$ is exact (since any map from $\Tor_1^{\oo(U^\prime)}(E_w,Q)$ to $\mc{R}_{F_v,E}(\delta_{2,z})$ has to be zero). Now $Q_w$ might not be a free $\mc{R}_{F_v,E}$-module, so this may not be a triangulation, but, as noted above, the free part of $Q_w$ embeds into $\mc{R}_{F_v,E}(\delta_{1,z})$ and the torsion part is killed by $t^n$. Arguing as in \cite[Example A.2.1]{bh} we deduce that by replacing $\mc{R}_{F_v,E}(\delta_{2,z})$ in equation (\ref{quasitriangulation}) with its saturation in $D_{rig}(\rho_{z,v})$ we obtain a triangulation
\begin{equation}\label{triangulation}
0 \to \mc{R}_{F_v,E}(\delta_{2,z}^\prime) \to D_{rig}(\rho_{z,v}) \to \mc{R}_{F_v,E}(\delta_{1,z}^\prime) \to 0 
\end{equation}
with $\delta_{2,z}^\prime = a^{s} \delta_{2,z}$ and $\delta_{1,z}^\prime = a^{-s}\delta_{1,z}$, where $s_\sigma=0$ or $s_\sigma =k_{2,\sigma}-k_{1,\sigma} -1$ otherwise. In particular, both parameters are crystalline, and hence $D_{crys}(\rho_{z,v}) \neq 0$. However, if $\chi_v$ was ramified, we would have $D_{crys}(\rho_{z,v}) = 0$ since $\rho_{z,v}=\rho_{x,v}\otimes \chi_v$ and $\rho_{x,v}$ is semistable. We conclude that $\chi_v$ is unramified, as desired, and hence that $\rho_{z,v}$ is semistable and $\pi \otimes (\chi \circ \det)$ is a refinable automorphic representation.

\medskip

To show that $z$ is classical of the desired form, it remains to show that $\alpha_v^\prime := \varpi_v^{k_2}U_v(z)$ is a refinement at $v$, i.e. an eigenvalue of $\varphi^{f_v}$ on $D_{crys}(\rho_{z,v})$. But this follows directly by comparing the triangulation (\ref{triangulation}) with the possible triangulations of semistable representations as outlined above (note that $\varphi^{f_v}$ does not have repeated eigenvalues on $D_{crys}(\rho_{z,v})$ by our assumptions, as this property is stable under twisting). 
\end{proof}

To finish this subsection, we will study the fibre of the coherent sheaf $\wt{\mc{M}}$ on $\wt{\E}$ near a classical point $x=(\pi,\alpha)$. We assume that each $\alpha_v$ has multiplicity one (as an eigenvalue of $U_v$ on $\pi_v^{I_v}$) and that $\ol{\rho}_x$ is irreducible and generic. By Proposition \ref{integration map surjective}, there is a surjection
\[
\wt{\mc{M}}_z \to H^d(Y_K, \ms{L}_{k_z})[\p_z]
\]
for any classical point $z = (\pi_z,\alpha_z)$ of weight $k_z$ satisfying the same assumptions as $x$, which is an isomorphism whenever the slope of $\alpha_z$ is sufficiently small (cf. \cite[Proposition 6.3.8]{bh}). 

\begin{proposition}\label{constancy of classical subspace}
Under the assumptions above, the dimension $\dim H^d(Y_K, \ms{L}_{k_z})[\p_z]$ is constant on classical points in a small enough neighbourhood of $x$. 
\end{proposition}

\begin{proof}
By the description of cuspidal cohomology, we have
\[
H^d(Y_K, \ms{L}_{k_z})[\p_z] \cong \left( \bigotimes_{v\mid p} \pi_{z,v}^{I_v}[U_v = \alpha_{z,v}] \otimes \bigotimes_{v \nmid p\infty} \pi_{z,v}^{K_v} \right)^{\oplus 2^d}.
\]
Since $\dim \pi_v^{I_v}[U_v = \alpha_{v}] =1$ by assumption, we can choose a neighbourhood $V$ of $x$ so that $\dim \pi_{z,v}^{I_v}[U_v = \alpha_{z,v}] =1$ for any classical point $z$ in $V$ (one may choose $V$ so that the slope is constant, and such that any classical point $z \neq x$ in $V$ has sufficiently small slope, i.e., sufficiently large weight). Hence, after possibly shrinking $V$, it suffices to prove that $\dim \pi_{z,v}^{K_v}$ is constant on classical points in $V$ for $v \nmid p$. But by the argument in the proof of Proposition \ref{changing tame level}, $\pi_{z,v}|_{\GL_2(\oo_v)}$ is constant on classical points in a small enough neighbourhood of $x$. This finishes the proof.
\end{proof} 

As noted above, $\wt{\mc{M}}_z \cong H^d(Y_K, \ms{L}_{k_z})[\p_z]$ when $z$ has small slope. By semicontinuity of the fibre dimension of $\wt{\mc{M}}$, we immediately get the following corollary to Proposition \ref{constancy of classical subspace}.

\begin{corollary}\label{3.2.13}
We have $\wt{\mc{M}}_x \cong H^d(Y_K, \ms{L}_{k_x})[\p_x]$ if and only if $\wt{\mc{M}}$ is locally free at $x$.
\end{corollary}

We will later need the following consequence of smoothness:

\begin{proposition}[{\cite[Proposition 8.1.3]{bh}}]\label{locfree}
If $\wt{\mc{E}}_{mid}$ is smooth at $x$, then $\wt{\mc{M}}$ is locally free at $x$. 
\end{proposition}

\subsection{Overconvergent cohomology for $\SL_2$}\label{subsec: ocSL2}

We now discuss overconvergent cohomology and the corresponding eigenvarietes for $\SL_{2}$ over $F$.

\subsubsection{Weight space and the weight action on locally analytic distributions}
Let $T_0\subset \SL_2/\mc{O}_F$ be the diagonal torus. 
Let $\mc{W}=\Spf(\oo_E[[T_0(\mc{O}_p)]])^{rig}$ be the rigid analytic space over $E$ parametrizing continuous characters of $T_0(\mc{O}_p)$. We will often identify $T_0$ with $\mb{G}_m$ by sending $\mathrm{diag}(t,t^{-1})$ to $t$, and use this to identify characters of $T_0$ with characters of $\mb{G}_m$. The weight space $\mc{W}$ is related to $\mc{W}_k$ from \S \ref{sec:weightspace} by the natural inclusion $T_0\subset T$, which induces a map 
\[
\mathrm{res}:\mc{W}_k \rightarrow \mc{W}
\]
by restriction along $T_0(\oo_p) \sub T(\oo_p)$. Explicitly, it is given by  $(\chi_1, \chi_2)\mapsto \chi_1\chi_2^{-1} = k_1 k_2 \chi_2^{-2}$, and in particular it is finite \'etale. 

\medskip

Let  $\kappa :\Spa(R) \rightarrow \mc{W}$ be a $p$-adic weight. We let $\mbf{s}(\kappa) \in \Z_{\geq 0}^{\{v\mid p\}}$ be a tuple such that the extension $\kappa_{!}: \mc{O}_p \rightarrow R$ of $\kappa$ by zero is an element of $\mbf{A}^{\mbf{s}(\kappa)}(\mc{O}_p, R)$. 
Just as for $\GL_2$, we can equip the spaces $\mc{A}(\oo_p,R)$ and $\mc{D}(\mc{O}_p,R)$ with weight actions. Let $I_0=I\cap \SL_2(\mc{O}_p)$ be the product of the Iwahori subgroups $I_{0,v} = I_v \cap \SL_2(\oo_v)$ for $v\mid p$. Still for  $v\mid p$, let $\Sigma_{0,v}:=\Sigma_v \cap \SL_2(F_v)= \left\{ (^{\varpi_v^a} \ _{\varpi_v^{-a}}) : a \in \Z \right\}$ and put $\Sigma_0:= \prod_{v\mid p} \Sigma_{0,v}$. Again consider the subgroup 
\[
N_1= \begin{pmatrix} 1 & \varpi_p\mc{O}_p \\  & 1 \end{pmatrix}
\]
of the unipotent radical and define
\[
 \Sigma_0^{+}= \left\{ t\in \Sigma_0 \mid tN_{1}t^{-1}\sub N_{1} \right\}.
\]
We consider the monoid $\Delta_{0,p}:=I_0\Sigma_0^{+}I_0=\Delta_p\cap \SL_2(F_p)$. Now for $\kappa$ as above and $\mbf{s}\geq \mbf{s}(\kappa)$ we have a continuous $R$-linear right action of $\Delta_{0,p}$ on $\mbf{A}^{\mbf{s}}(\mc{O}_p, R)$ defined by
\[
f|_g(z):=\ka(cz+d) f(g\cdot z) 
\]
where $f\in \mbf{A}^{\mbf{s}}(\mc{O}_p, R)$, $z \in \mc{O}_p$ and $g= \xi g^{\circ}$ is the same decomposition as in the $\GL_2$ case, i.e. $g^{\circ} = \left(^a_c \ ^b_d \right) \in \Delta_p^{\circ} $ satisfies $d\in \mc{O}_p^{\times}$. As in Definition \ref{weightaction}, for any $p$-adic weight $\ka: \Spa(R) \rightarrow \mc{W}$ and $\mbf{s}\geq \mbf{s}(\ka)$ we define $\mbf{A}^{\mbf{s}}_{\kappa}:=\mbf{A}^{\mbf{s}}(\mc{O}_p, R), \mbf{D}^s_{\kappa}:=\mbf{D}^{\mbf{s}}(\mc{O}_p, R), \mc{A}_{\kappa}:=\mc{A}(\mc{O}_p, R)$ and $\mc{D}_{\kappa}:=\mc{D}(\mc{O}_p, R)$ as the respective $R$-modules equipped with the continuous action by $\Delta_{0,p}$ defined above.

\begin{remark}
Note that, by definition, if $\kappa:\Spa(R) \rightarrow \mc{W}_k$ is a $p$-adic $\GL_2$-weight, then $\kappa_0:=\mathrm{res}\circ\kappa$ is a $p$-adic $\SL_2$-weight and the respective $R[\Delta_{0,p}]$-modules $\mbf{A}^{\mbf{s}}_{\kappa}$ and $\mbf{A}^{\mbf{s}}_{\kappa_0}$ agree. The only difference between these objects is that the former has an action of the bigger monoid $\Delta_p$. By dualizing, the same remark applies for distribution modules. 
\end{remark}

\subsubsection{Overconvergent cohomology and eigenvarieties}
Overconvergent cohomology and eigenvarieties for $\SL_2$ are defined according to the same general recipe as outlined for $\GL_2$ in \S \ref{subsec: ocGL2}. We will briefly go through what we need in order to set up notation.

\medskip

Let $K_v \sub \SL_2(F_v)$ be compact open subgroups for $v \nmid p\infty $ with $K_v = \SL_2(\oo_v)$ for all but finitely many $v$. Set $K^p = \prod_{v \nmid p\infty} K_v$ and $K = K^p I_0$ and assume that they are neat. We define the abstract $\SL_2$-Hecke algebra $\mbf{S}(K)$ as follows.  Let $S(K)$ denote the union of set of places $v \nmid p\infty$ where $K_v \neq \SL_2(\oo_v)$ and the places above $p$ and $\infty$. We define 
\[
\mbf{S}(K):= \mbf{S}(\Delta_{0,p},I_0) \otimes \bigotimes^{\prime}_{v\notin S(K)} \mbf{S}(\SL_2(F_{v}), K_{v}) 
\]
where $\mbf{S}(\SL_2(F_v), K_{v})=C_c(K_{v}\backslash \SL_2(F_v)/K_{v}, E) $ is the spherical Hecke algebra over $E$ for $v \notin S(K)$, and $\mbf{S}(\Delta_{0,p},I_0)$ is the Hecke algebra (over $E$) for the Hecke pair $(\Delta_{0,p},I_0)$. The latter is also the commutative subalgebra of the Iwahori--Hecke algebra $C_c(I_0\backslash \SL_2(F_p)/I_0, E)$ generated by the $[I_0\delta I_0]$, where $\delta \in \Sigma^{+}_0$. Note that when $K = \SL_2(\A_F^{\infty}) \cap \wt{K}$ for some compact open $\wt{K} \sub \GL_2(\A_F^\infty)$ with $\wt{K}_p = I$, then $\mbf{S}(K)$ is naturally identified with the subalgebra $\s(\wt{K}) \sub \T(\wt{K})$ defined in \S \ref{subsec: Galois reps}, and in particular we have the $\SL_2$-Hecke operators defined in \S \ref{subsec: ocGL2} in $\s(K)$. Before proceeding, we remark that $t_v = T_v^2 S_v^{-1} - q_v -1$ in the spherical Hecke algebra for $\GL_2(F_v)$ (this is a standard calculation and is used in the proof of Proposition \ref{Galois rep SL2} later on).

\medskip

We define $u_{p}:= [K (^{p}\ _{p^{-1}}) K] \in \mbf{S}(K)$. Let $\Omega = \Spa(R) \sub \mc{W}$ be an open affinoid subdomain and let $\mbf{s}\geq\mbf{s}(\Omega)$. Then $u_{p}$ acts compactly on the chain complex $C_\bullet(Y_K^1,\mbf{A}^{\mbf{s}}_{\Omega})$ (for a choice of Borel--Serre complex) and we can form the Fredholm series 
\[
f_{\Omega,u_{p}}(t)=\det(1-tu_{p}\mid C_\bullet(Y_K^1,\mbf{A}^{\mbf{s}}_{\Omega})).
\]
As $\Omega$ varies these glue together to a Fredholm series $f_{u_{p}}(t) \in \mc{O}(\mc{W}_0)\{\{t\}\}$ and we let $\mc{Z}_{u_p}\subset \mc{W} \times \mb{A}^1_E$ be the corresponding Fredholm hypersurface. For any slope-adapted pair $(\Omega, h)$ for $u_{p}$ we get a slope $\leq h$ decomposition 
\[
H^*(Y_{K}^1, \mathcal{D}_{\Omega})= H^*(Y^1_{K}, \mathcal{D}_{\Omega})_{\leq h} \oplus H^*(Y^1_{K}, \mathcal{D}_{\Omega})_{>h}
\]
(cf. \cite[\S 2.3.12]{urban}) and an induced morphism 
\[
\psi_{\Omega,h} \colon \mbf{S}(K)\rightarrow \mathrm{End}_{\oo(\Omega)}(H^*(Y_{K}^1, \mathcal{D}_{\Omega})_{\leq h}).
\]
The $H^*(Y_{K}, \mathcal{D}_{\Omega})_{\leq h}$ glue together to a graded $\mc{O}_{\mc{Z}_{u_{p}}}$-module, which we denote by $\mc{M}$ and the $\psi_{\Omega,h}$ glue to a morphism $\psi \colon  \mbf{S}(K)\rightarrow \mathrm{End}(\mc{M})$. From this, we may construct the eigenvariety $\E = \E_{K^p}$ as for $\GL_2$. As for $\GL_2$, we may define\footnote{We remark that the construction of the middle degree eigenvariety and the characterization of its points given in \cite{bh} is completely general and uses nothing special to $\GL_{2/F}$. Of course, it may be empty in general.} a middle degree eigenvariety $\E_{mid}$, whose points are characterised by the condition that $x \in \E_{mid}(\ol{\Q}_p)$ if and only if $H^i(Y_K^1,\D_\ka)_{\p_x} = 0$ for $i \neq d$, where $\ka$ is the weight of $x$.

\medskip

Our next goal is to define classical systems of Hecke eigenvalues and their corresponding points on the eigenvariety. It turns out that these are more natural to associate with $L$-packets than individual automorphic representations. We make the following definition.

\begin{definition}
A pair $(\Pi,\gamma)$ is called a refined $L$-packet if 

\begin{enumerate}
\item $\Pi = \Pi(\wt{\pi})$ is the global $L$-packet of $\SL_2(\A_F)$ associated with a cuspidal cohomological automorphic representation $\wt{\pi}$ of $ \GL_2(\A_F)$;

\smallskip

\item $\gamma=(\gamma_v)_{v|p}$, where each $\gamma_v$ is an eigenvalue of $u_v$ acting on $\pi_v^{I_{0,v}}$ for some member $\pi_v$ of the local $L$-packet $\Pi_v$ of $\Pi$ at $v$ ($\gamma$ is the called a refinement of $\Pi$).

\end{enumerate}

\end{definition}

\begin{definition}
Let $(\Pi,\gamma)$ be a refined $L$-packet. Assume that $K=K^p I_0 \sub SL_2(\A_F)$ is a compact open subgroup, with $K^p=\prod_{v\nmid p\infty} K^p_v$. Assume that there exists an element $\pi\in \Pi$ (not necessarily automorphic) such that $(\pi^{\infty})^K\neq 0$. We say that $K^p$ is a tame level for $\Pi$ in this case. Let $S(K)$ denote the set of finite places $v$ for which $K^p_v \neq \SL_2(\oo_v)$. Then $\Pi_v$ is an unramified principal series $L$-packet for all $v\notin S(K)$. Let $\pi_v$ be the unique member of $\Pi_v$ with $\pi_v^{\SL_2(\oo_v)}\neq 0$ for $v\notin S(K)$. We define the system of Hecke eigenvalues $\phi = \phi_{(\Pi,\gamma)} : \s(K) \to \ol{\Q}_p$ attached to $(\Pi,\gamma)$ by setting $\phi(t_v)$ to be the eigenvalue of $t_v$ acting on $\pi_v^{\SL_2(\oo_v)}$ for $v\notin S(K)$, and letting $\phi(u_v) = k_v(\vp_v) \gamma_v$ for $v\mid p$, where $k$ is the weight of $\Pi$.
\end{definition}

\begin{remark}
Let $(\Pi,\gamma)$ be a refined $L$-packet.

\begin{enumerate}
\item Note that the system of Hecke eigenvalues $\phi=\phi_{(\Pi,\gamma)}$ implicitly depends on the choice of tame level $K^p$. The reader should think of this choice as the tame level of an eigenvariety.

\smallskip

\item Consider the system of Hecke eigenvalues $\phi$, with a tame level $K^p$ chosen. For each $v\notin S(K)$ we have the $\SL_2(\oo_v)$-spherical representation $\pi_v$ (as mentioned in the definition of $\phi$) and for each $v\mid p$ we can choose a member $\pi_v \in \Pi_v$ such that $\gamma_v$ is an eigenvalue of $u_v$ on $\pi_v^{I_{0,v}}$. Then, by Lemma \ref{swap}, there exists an automorphic member $\tau \in \Pi$ such that $\tau_v = \pi_v$ for all $v\notin S(K)$ and all $v\mid p$, and $(\pi^\infty)^K\neq 0$. However, $\tau$ (or even $\tau^\infty$) need not be unique. This is the reason that we find it more natural to attach systems of Hecke eigenvalues to $L$-packets rather than automorphic representations. 

\end{enumerate}
\end{remark}

\begin{remark}

Consider the setting of the previous remark. Choose a cuspidal cohomological automorphic representation $\wt{\pi}$ of $\GL_2(\A_F)$ such that $\Pi = \Pi(\wt{\pi})$, and let $K=K^p I_0$ be as above. After twisting $\wt{\pi}$ if necessary, we may assume that $\wt{\pi}$ has level containing $\wt{K} = \wt{K}^p I$ with $S(\wt{K}) = S(K)$, $K \sub \SL_2(\A_F^\infty) \cap \wt{K}$, and that there exists a (not necessarily unique) refinement $\alpha$ of $\wt{\pi}$ such that $\phi_{(\Pi,\gamma)}$ is the restriction of $\phi_{(\wt{\pi},\alpha)}$ to $\s(\wt{K}) = \s(K)$. We will always assume that $\wt{\pi}$ is chosen this way. 

\medskip

With $\wt{\pi}$ as above, let $\alpha = (\alpha_v)_{v|p}$ be a refinement of $\wt{\pi}$ and let $\phi_{(\wt{\pi},\alpha)}$ be the system of Hecke eigenvalues of $(\wt{\pi},\alpha)$ with respect to the tame level $\wt{K}$. For $v\mid p$, define $\gamma^\prime_v$ to be the eigenvalue of $u_v$ acting on $\wt{\pi}_v^{I_v}[U_v=\alpha_v]$. Then $(\Pi,\gamma^\prime)$ is a refined $L$-packet such that $\phi_{(\Pi,\gamma^\prime)}$ is the restriction of $\phi_{(\wt{\pi},\alpha)}$ to $\s(K)$. Any refinement of $\Pi$ arises from a refinement of $\wt{\pi}$ in this way. Moreover, note that two refinements $\alpha$ and $\beta$ of $\wt{\pi}$ give rise to the same refinement $\gamma$ of $\Pi$ if and only $\alpha_v = \pm \beta_v$ for all $v\mid p$ (the signs do not need to be the same for all $v$), since $u_v = U_v^2 S_v^{-1}$.

\end{remark}

Equipped with this, we now define the classical points on $\E$ that we will be interested in. Let $(\Pi,\gamma)$ be a refined $L$-packet and let $K^p$ be a tame level for $\Pi$. Let $\p_{(\Pi,\gamma)}$ be the prime ideal of $\s(K)$ corresponding to $\phi_{(\Pi,\gamma)}$. We will use the following vanishing result.

\begin{lemma}\label{vanishing for SL(2)}
Assume that the reduction  $\ol{\rho}$ modulo $p$ of $\rho_{\wt{\pi}}$ is generic and irreducible. Let $k^\prime$ be a cohomological weight for $\GL_2$; by abuse of notation we denote its restriction to $\SL_2$ by $k^\prime$ as well. Let $\phi_{(\Pi,\gamma)}^\prime$ denote the restriction of $\phi_{(\Pi,\gamma)}$ to $\s(K^\prime)$ for a compact open subgroup $K^\prime = (K^\prime)^p I_0 \sub K$. Let $\p_{(\Pi,\gamma)}^\prime$ be the kernel of $\phi_{(\Pi,\gamma)}^\prime$. Then $H^i(Y_K^1,\D_{k^\prime})_{\p_{(\Pi,\gamma)}^\prime} = 0$ if $i \neq d$. 
\end{lemma}

\begin{proof}
Note that the effect of introducing $K^\prime$ in the formulation is simply to say that we do not need all of $\s(K)$; we may remove a finite number of Hecke operators away from $p$. By Lemma \ref{shrinking}, we may assume that $\wt{K}^\prime \sub \wt{K}$ is such that $H^i(Y_K^1, \D_{k^\prime})$ injects into $H^i(Y_{\wt{K}^\prime},\D_{k^\prime})$ by shrinking it further if necessary. The localization $H^i(Y_K^1,\D_{k^\prime})_{\p_{(\Pi,\gamma)}}$ then injects into
\[
\bigoplus_x H^i(Y_{\wt{K}^\prime},\D_{k^\prime})_{\p_x},
\]
where the sum is taken over those systems of $\GL_2$-Hecke eigenvalues $x$ appearing in $H^i(Y_{\wt{K}^\prime},\D_{k^\prime})$ which lift $\phi_{(\Pi,\gamma)}$. Since the corresponding Galois representations are all twists of $\rho_{\wt{\pi}}$, vanishing follows from Theorem \ref{consequence of ct}.
\end{proof}

As for $\GL_2$, we have the twisted version $\D_{k,\star}$ of $\D_{k}$, the integration map $I_{k} : \D_{k,\star} \to \mathscr{L}_{k}$ and the locally analytic dual BGG resolution of $\mathscr{L}_{k}$, and under the assumptions of Lemma \ref{vanishing for SL(2)} the integration map induces a surjection
\[
H^d(Y_K^1,\D_{k,\star})_{\p_{(\Pi,\gamma)}} \to H^d(Y_K^1, \mathscr{L}_k)_{\p_{(\Pi,\gamma)}},
\]
and $\phi_{(\Pi,\gamma)}$ lies on $\E_{mid}$. By the general arguments proving  \cite[Proposition 6.4.6, Theorem 6.4.8]{bh}, $\E_{mid}$ is equidimensional of dimension $d$ and reduced (and, by definition, Zariski open in $\E$). Finally, we record the existence of a three-dimensional pseudocharacter over the Zariski closure $\ol{\E}_{mid}$ of $\E_{mid}$ (we remark that $\ol{\E}_{mid}$ is exactly the union of irreducible components of $\E$ of dimension $d$).

\begin{proposition}\label{Galois rep SL2}
There exists a three-dimensional pseudocharacter $t : G_{F} \to \oo^+(\ol{\E}_{mid})$ which is unramified outside $S(K)$ and characterized by the property
\[
t(Frob_v) = q_v^{-1}(t_v +1)
\]
for all $v\notin S(K)$. 
\end{proposition}

\begin{proof}
The proof is standard so we content ourselves with a sketch; cf. \cite[Lemma 5.8]{L1} and \cite[Lemma 2.2]{L2}. The classical points $y=(\Pi,\gamma)$ in $\ol{\E}_{mid}$ whose $L$-packet $\Pi$ comes from a cuspidal $\GL_2(\A_F)$-automorphic representation $\pi$ of cohomological weight (and unramified outside $S(K)$) are Zariski dense. For each such $y=(\Pi,\gamma)$ coming from $\pi$, there is an associated Galois representation $\ad^0\rho_\pi$ such that for $v\notin S(K)$, $tr(\ad^0\rho_\pi(Frob_v))q_v -1$ is the $t_v$-eigenvalue on $(\pi^\infty)^K$. The corresponding pseudocharacters then glue together by the technique in \cite[Example 2.32]{chenevier-determinants}.
\end{proof}

\subsection{Comparison and $p$-adic functoriality}\label{subsec: p-adic functoriality} Our final goal in this section is to provide a comparison between the $\GL_2$- and $\SL_2$-eigenvarieties, with the goal of computing the local geometry of $\E$ at the points we are interested in, as well as the fibre of $\mc{M}$ at those points.

\medskip

Let us now describe the setting that we will work in in this subsection. We let $(\pi,\alpha)$ be a refined automorphic representation for $\GL_2(\A_F)$, of cohomological weight $k$. We make the following assumptions on $\pi$:
\begin{enumerate}
\item $\ol{\rho}_\pi$ is irreducible and generic;

\smallskip

\item For any $v\mid p$, $D_{crys}(\rho_{\pi,v})$ has no $\varphi^{f_v}$-eigenvalue of multiplicity $>1$.
\end{enumerate}

We then introduce the following objects:

\begin{itemize}
\item We let $\Pi = \Pi(\pi)$ be the global $L$-packet for $\SL_2(\A_F)$ corresponding to $\pi$;

\smallskip

\item Let $K^p = \prod_{v \nmid p\infty} K_v^p$ be such that $(\pi^{\infty})^{K} \neq 0$, with $K = K^p I$. For simplicity, we assume that $K^p \sub K_1^p(\m)$ with $\m$ the prime-to-$p$ conductor of $\pi$;

\smallskip

\item We write $x$ for the classical point attached to $(\pi,\alpha)$ on $\wt{\E} = \wt{\E}_{K^p}$, and $y$ for the classical point attached to $(\Pi,\gamma)$ on $\E = \E_K := \E_{(K \cap \SL_2(\A_F^\infty))}$, where $\gamma$ is the refinement of $\Pi$ induced by $\alpha$ as in \S \ref{subsec: ocSL2}.
\end{itemize}

It is the local structure of $\E$ around $y$ that we wish to understand. To this end, consider the weight space $\mc{W}_k$ for $\GL_2$, with its restriction map
\[
\mathrm{res} : \mc{W}_k \to \mc{W}.
\]
This map is finite \'etale, so by Proposition \ref{local isom} there is an open affinoid neighbourhood $U$ in $\mc{W}_k$ of $k$ which maps isomorphically onto its image in $\mc{W}$ by $\mathrm{res}$. We use this to identify $U$ and $k$ with their respective images in $\mc{W}$ and to view them as both $\GL_2$-  and $\SL_2$-weights; we hope that this does not cause confusion. Shrinking $U$ if necessary, we may assume that there is an $h \in \Q_{\geq 0}$ such that $(U,h)$ is slope adapted both for $\SL_2$ and $\GL_2$, and that $x$ and $y$ have slope $\leq h$ (for $u_p$ --- recall that all slopes in this paper are with respect to $u_p$ for both $\SL_2$ and $\GL_2$). The natural map $Y_K^1 \to Y_K$ then gives us a map
\[
H^d(Y_K, \D_U)_{\leq h} \to H^d(Y_K^1, \D_U)_{\leq h},
\]
equivariant for the action of $\mbf{S}(K)$ on both sides, where $\mbf{S}(K)$ is viewed as a subalgebra of $\mbf{T}(K)$ for the action on the left hand side. We may factor this map as 
\begin{equation}\label{spaces prelim}
H^\ast(Y_K, \D_U)_{\leq h} \twoheadrightarrow H^\ast(Y_K, \D_U)_{\leq h,H_K} \hookrightarrow H^\ast(Y_K^1, \D_U)_{\leq h},
\end{equation}
with $\mbf{S}(K) := \s(K \cap \SL_2(\A_F^\infty))$ acting on all three terms and the maps are equivariant (we note that $H^\ast(Y_K, \D_U)_{\leq h,H_K}$ can be obtained from $H^\ast(Y_K, \D_U)$ either by first taking the slope $\leq h$-part and then taking $H_K$-coinvariants or vice versa). 

\medskip

We now consider four a priori different Hecke algebras. First, there is $\T_{U,h}$ acting on $H^\ast(Y_{K}, \D_U)_{\leq h}$. To emphasize the level we will call it $\T_{U,h}^K$; note that $\T^K_{U,h}$ is the local piece of the eigenvariety $\wt{\E}_K$ corresponding to $(U,h)$. Second, the subalgebra of $\T_{U,h}^K$ generated by $\oo(U)$ and the image of $\s(K)$ will be denoted by $\s_{U,h}^{\GL_2, K}$. Third, the subalgebra of $\End_{\oo(U)}(H^\ast(Y_K, \D_U)_{\leq h,H_K}))$ generated by $\oo(U)$ and the image of $\s(K)$ will be denoted by $\ol{\s}_{U,h}^{\GL_2, K}$. Finally, the subalgebra of $\End_{\oo(U)}(H^\ast(Y_K^1, \D_U)_{\leq h}))$ generated by $\oo(U)$ and the image of $\s(K)$ will be denoted by $\s_{U,h}^{K}$; note that $\s_{U,h}^{K}$ is the local piece of $\E_K$ corresponding to $(U,h)$. All these Hecke algebras are finite over $\oo(U)$. From the definitions and the diagram (\ref{spaces prelim}) we obtain natural maps between these Hecke algebras, summarized in the diagram
\begin{equation}\label{hecke algebras prelim}
\T^K_{U,h} \hookleftarrow \s_{U,h}^{\GL_2, K} \twoheadrightarrow \ol{\s}_{U,h}^K \twoheadleftarrow \s_{U,h}^K.
\end{equation}
For our analysis, we also want to relate $\T^K_{U,h}$ and $\s_{U,h}^{K}$ more directly, by $p$-adically interpolating classical functoriality. Let $\Spa(\T_{U,h}^{K,mid})$ be the (disjoint) union of connected components $C$ of $\Spa(\T_{U,h}^K)$ which satisfy $C \sub \wt{\E}_{mid}$; note that this union is indeed affinoid. By shrinking $U$ if necessary, we may assume that $U$ is connected and that $x$ and all its twists lie in $\Spa(\T_{U,h}^{K,mid})$. We may then construct the $p$-adic functoriality map.

\begin{proposition}\label{p-adic functoriality}
There is a (necessarily unique) $\oo(U)$-algebra map $\s_{U,h}^K \to \T_{U,h}^{K,mid}$ making the diagram
\[
\begin{tikzcd}
\s(K) \arrow[d] \arrow[r] &   \T(K) \arrow[d]  \\
\s_{U,h}^{K} \arrow[r]   & \T_{U,h}^{K,mid}
\end{tikzcd}
\]
commute.
\end{proposition}

\begin{proof}
This is the $p$-adic interpolation of classical functoriality from $\GL_2$ to $\SL_2$. The proof is standard, so we will only sketch it. Classical weights are Zariski dense in $U$ by connectedness, so classical points $x^\prime = (\pi^\prime, \alpha^\prime)$ of small slope are Zariski dense in $\Spa(\T_{U,h}^{K,mid})$ since it is equidimensional of dimension $d$. These points `map' to points $y^\prime = (\Pi^\prime, \gamma^\prime)$ on $\Spa(\s_{U,h}^K)$ by the same procedure that associated $y$ to $x$ earlier in this subsection. Since $\T_{U,h}^{K,mid}$ is reduced, the map now exists by the interpolation theorem \cite[Theorem 3.2.1]{JN2}.
\end{proof} 

Our next step is to take the completed localization of all terms in diagram (\ref{hecke algebras prelim}) at $\p_y$. We write the diagram thus obtained as:
\begin{equation}\label{hecke algebras localized}
\T^K_{y} \hookleftarrow \s_{y}^{\GL_2, K} \twoheadrightarrow \ol{\s}_{y}^K \twoheadleftarrow \s_{y}^K.
\end{equation}
Note that $\T_y^K$ can also be described as the completed localization of $\T_{U,h}^{K,mid}$ at $\p_y$, and hence the $p$-adic functoriality map from Proposition \ref{p-adic functoriality} induces a map $\s_y^K \to \T_y^K$ whose image is $\s_y^{\GL_2,K}$. Note also that 
\[
\T^K_{y} = \prod_{z \in \pi_K^{-1}(y)} \T^K_{z},
\]
where $\pi_{K} : \Spec \T^K_{U,h}  \to \Spec \s_{U,h}^{\GL_2,K}$ is the map induced by the inclusion, and $\T^K_{z}$ is the completed local ring of $\T^K_{U,h}$ at $z$. To study these Hecke algebras, we will need to consider what happens when we vary $K$. If $K^\prime = (K^\prime)^p I \sub K$, then we have a diagram
\begin{equation}\label{hecke algebras changing level}
\begin{tikzcd}
\T_y^{K^\prime} \arrow[d] &   \s_y^{\GL_2,K^\prime} \arrow[l, hook] \arrow[d] \arrow[r, two heads]  &  \ol{\s}_y^{K^\prime} \arrow[d] \arrow[dr, dashed] & \s^{K^\prime}_y \arrow[l, two heads] \arrow[d]  \\
\T_y^{K}   & \s_y^{\GL_2,K} \arrow[l, hook] \arrow[r,two heads] & \ol{\s}_y^{K}  & \s^{K}_y \arrow[l, two heads]     
\end{tikzcd}
\end{equation}
The vertical arrows are the usual level changing maps, but the dashed arrow comes from Lemma \ref{shrinking} and exists whenever $K^\prime$ is small enough (depending on $K$). To finish this section, we will show that all vertical arrows are surjective. For the map $\T_y^{K^\prime} \to \T_y^{K}$, this follows from Lemma \ref{changing Hecke operators}. Surjectivity of the remaining maps also follow from the same argument (using Chebotarev) once we have established the existence of suitable pseudocharacters.

\medskip

For $\s_y^{\GL_2, K}$, consider the representation $\rho_{\hat{y}} : G_F \to \GL_2(\T_y^K)$ obtained from Lemma \ref{good nhood} and the map $\T_{U,h}^K \to \T_y^K$, and let $T_0 : G_F \to \T_y^K$ be the pseudocharacter corresponding to $\mathrm{ad}^0(\rho_{\hat{y}})$. Since $T_0(Frob_v) = q_v^{-1}(t_v +1) \in \s_y^{\GL_2,K}$ for all $v \notin S(K)$, $T_0$ is valued in $\s_y^{\GL_2,K}$. Chebotarev then gives us surjectivity of $\s_y^{\GL_2,K^\prime} \to \s_y^{\GL_2,K}$, which also implies surjectivity of $\ol{\s}_y^{K^\prime} \to \ol{\s}_y^{K}$. Finally, surjectivity of $\s_y^{K^\prime} \to \s_y^{K}$ follows from Proposition \ref{Galois rep SL2}. 

\medskip

Let us now take stock of the situation. We have a commutative diagram
\begin{equation}\label{main diagram}
\begin{tikzcd}
\T_y^K &   \s_y^{\GL_2,K} \arrow[l, hook] \arrow[r, two heads]  &  \ol{\s}_y^{K} & \s^{K}_y \arrow[l, two heads] \arrow[ll, two heads, bend right]
\end{tikzcd}
\end{equation}
where the map $\s_y^K \to \s_y^{\GL_2,K}$ is the $p$-adic functoriality map. In the rest of this section, we will prove that this $p$-adic functoriality map is an isomorphism, while simultaneously proving that $\s_y^K$ and $\s_y^{\GL_2,K}$ are independent of the tame level $K$. This is the main result of this section; the next section will analyse the situation further using Galois deformation theory. To prove these results, we first have to analyse $\T_y^K$. The set $\pi_K^{-1}(y)$ changes as $K$ changes, so $\T_y^K$ is visibly not independent of $K$, and we have to replace it with something that is. Recall the product decomposition
\[
\T_y^K = \prod_{z \in \pi_K^{-1}(y)} \T_z^K.
\]
We now consider the set $\pi_K^{-1}(y)$. Let $z \in \pi_K^{-1}(y)$. By Proposition \ref{twists are classical}, there exists a finite image Hecke character $\chi$, unramified at all $v\mid p$, such that
\[
z = (\pi \otimes (\chi \circ \det), \alpha_z)
\]
for some refinement $\alpha_z$. We do \emph{not} assert that $\chi \in H_K$, since we have not determined how ramified $\chi$ is away from $p$. Twisting by $\chi^{-1}$, we obtain a point 
\[
(\pi,\alpha^\prime) := (\pi, (\chi(\varpi_v)\alpha_{z,v})_v ).
\]
Here $\alpha^\prime$ is a refinement of $\pi$, possibly different from $\alpha$. Note that $(\pi,\alpha^\prime) \in \pi_K^{-1}(y)$. Let $\Phi$ denote the set of all points $(\pi,\alpha^\prime) \in \pi_K^{-1}(y)$ (i.e. all points in $\pi_K^{-1}(y)$ whose corresponding automorphic representation is $\pi$); this set is independent of $K$. If $\pi$ has CM by $\wt{F}/F$ we let $C_{\pi} = \{1,\psi \}$ where $\psi$ is the quadratic character corresponding to $\wt{F}/F$; otherwise set $C_{\pi} =1$. The group $C_\pi$ acts on $\Phi$ by twisting. Explicitly, if $\chi \in C_\pi$ and $(\pi,\alpha^\prime)\in \Phi$, then
\[
\chi\cdot(\pi,\alpha^\prime) := (\pi \otimes (\chi \circ \det), (\chi(\varpi_v)\alpha^\prime_v)_v ) = (\pi, (\chi(\varpi_v)\alpha^\prime_v)_v ),
\]
since $\pi \otimes (\chi \circ \det) = \pi$ for $\chi \in C_\pi$. Note also that all $\chi \in C_\pi$ are quadratic, so $\chi(\varpi_v)\in \{\pm 1\}$ for all $v\mid p$ (and hence $\chi\cdot(\pi,\alpha^\prime) \in \pi_K^{-1}(y)$).

\begin{proposition}\label{localstructure1}
The natural map 
\[
\s_y^{\GL_2,K} \to \prod_{z\in \Phi} \T_z^{K}
\]
is injective, and these Hecke algebras are naturally independent of $K$. Moreover, $C_\pi$ acts on $\prod_{z\in \Phi} \T_z^K$ and
\[
\s_y^{\GL_2,K} \sub \left( \prod_{z\in \Phi} \T_z^K \right)^{C_\pi}.
\]
\end{proposition}

\begin{proof}
We start by proving injectivity. Let $L \in \s_y^{\GL_2,K}$ be such that it maps to $0$ in $\prod_{z\in \Phi} \T_z^K$. We have to show that $L=0$. For this, it suffices to show that $L$ maps to $0$ in $\T_z^K$ for any $z \in \pi_K^{-1}(y)$. By the discussion above, there exists a finite order Hecke character $\chi$, unramified at all $v \mid p$, and a point $z^\prime \in \Phi$ such that $z^\prime$ is the twist of $z$ by $\chi$. Choose $K^\prime = (K^\prime)^p I  \sub K$ prime to $p$ such that $\chi \in H_{K^\prime}$. Then twisting by $\chi$ defines an automorphism of $\wt{\E}_{K^\prime}$ sending $z$ to $z^\prime$. In particular, $\chi$ defines an isomorphism $\chi : \T_z^{K^\prime} \to \T_{z^\prime}^{K^\prime}$, and the diagram
\[
\begin{tikzcd}
\s_y^{\GL_2,K^\prime} \arrow[d, two heads] \arrow[r] & \T_z^{K^\prime} \arrow[d] \arrow[r, "\chi"] & \T_{z^\prime}^{K^\prime} \arrow[d]  \\
\s_y^{\GL_2,K} \arrow[r] \arrow[rr, bend right] & \T_z^{K} & \T_{z^\prime}^{K}     
\end{tikzcd}
\]
commutes. By Proposition \ref{changing tame level}, the middle and right vertical maps are isomorphisms and by the discussion above the left vertical map is surjective. Since $L$ maps to $0$ in $\T_{z^\prime}^{K}$, a simple diagram chase implies that it must map to $0$ in $\T_z^{K}$ as well, finishing the proof of injectivity.

\medskip

For independence of $K$, Proposition \ref{changing tame level} implies that $\prod_{z\in \Phi} \T_z^{K^\prime} \to \prod_{z\in \Phi} \T_z^{K}$ is an isomorphism. It then follows from the injectivity above that the surjective map $\s_y^{\GL_2,K^\prime} \to \s_y^{\GL_2,K}$ is an isomorphism.

\medskip

Finally, we prove the statements regarding the $C_\pi$-action. If $C_\pi$ is trivial, there is nothing to prove. Otherwise, for small enough $K$, $\psi \in H_K$ which defines the action (by twisting), and it is clear that $\s_y^{\GL_2,K}$ lies inside the $C_\pi$-invariants. The statement for all $K$ then follows from independence of $K$.
\end{proof}

With this, we can now prove that the $p$-adic functoriality map $\s_y^K \to \s_y^{\GL_2,K}$ is an isomorphism.

\begin{theorem}\label{main section 3}
The $p$-adic functoriality map $\s_y^K \to \s_y^{\GL_2,K}$ is an isomorphism, and these Hecke algebras are naturally independent of $K$.
\end{theorem}

\begin{proof}
Consider the commutative diagram
\[
\begin{tikzcd}
\s_y^{\GL_2,K^\prime} \arrow[d, two heads] \arrow[r, two heads]  &  \ol{\s}_y^{K^\prime} \arrow[d, two heads] \arrow[dr, dashed, , two heads] & \s^{K^\prime}_y \arrow[l, two heads] \arrow[d, two heads] \arrow[ll, two heads, bend right] \\
\s_y^{\GL_2,K} \arrow[r,two heads] & \ol{\s}_y^{K}  & \s^{K}_y. \arrow[l, two heads] \arrow[ll, two heads, bend left]  
\end{tikzcd}
\]
obtained from diagram (\ref{hecke algebras changing level}) by removing $\GL_2$-Hecke algebras and instead inserting the $p$-adic functoriality maps. Choose $K^\prime = (K^\prime)^p I \sub K$ small enough so that the dashed arrow exists. By Proposition \ref{localstructure1} the left vertical arrow is an isomorphism. The injectivity of $\s_y^{K} \to \s_y^{\GL_2,K}$ then follows from a simple diagram chase (crucially using the dashed arrow), and the theorem then follows immediately.
\end{proof}

\section{The local geometry of eigenvarieties}\label{sec: local geometry}

In this section we compute the local structure of the $\SL_2$-eigenvariety at classical points under certain assumptions. We build on the description from Theorem \ref{main section 3} by using Galois deformation theory and the results on pseudorepresentations and deformation theory from \S \ref{subsec: pseudoreps} and \S \ref{subsec: def theory}.

\subsection{Setup and the non-CM case}\label{subsec: setup and non-cm} Our setup in this whole section will be as \S \ref{subsec: p-adic functoriality}, with one additional assumption. We now recall this briefly: $(\pi, \alpha)$ is a refined automorphic representation for $\GL_2(\A_F)$ of weight $k$. We make the following assumptions on $\pi$ throughout the rest of this paper:

\begin{enumerate}

\item If $F \neq \Q$, then $\ol{\rho}_\pi$ is irreducible and generic;

\item For every $v\mid p$, there is no $U_v$-eigenvalue on $\pi_v^{I_v}$ of multiplicity bigger than $1$;

\item If $F \neq \Q$ and $\pi$ has complex multiplication by a quadratic CM extension $\wt{F}/F$, then $\wt{F} \not\subseteq F(\zeta_{p^\infty})$.

\end{enumerate}
Conditions (1) and (2) were present in \S \ref{subsec: p-adic functoriality} (and, more generally, heavily used in \S \ref{sec: eigenvarieties}). Condition (3) is new and plays a crucial role because it implies that the Bloch--Kato Selmer group $H_f^1(G_F, \mathrm{ad}^0 \rho_\pi)$ vanishes (more on this in Proposition \ref{decency} below). Our assumptions are closely related to decency in the sense of \cite[Definition 6.6.1]{bh}; they may be described as the most general conditions under which we currently know decency. Unlike the decency hypothesis itself, they are trivially invariant under twisting and do not depend on the choice of refinement\footnote{The problematic condition is that $x=(\pi,\alpha)$ lies on $\wt{\E}_{mid}$. It is not clear to us how to prove Proposition \ref{twists are on E-mid} with only the assumption that $x\in \wt{\E}_{mid}$.}. We record these statements in a proposition.

\begin{proposition}\label{decency}
If $\pi$ satisfies conditions (1)-(3), then $(\pi,\alpha)$ is decent for any choice of refinement $\alpha$. In particular, the adjoint Bloch--Kato Selmer group $H_f^1(G_F,\mathrm{ad}\,\rho_\pi)$ vanishes. Moreover, if $\chi$ is a finite Hecke character of $F$ unramified at all $v \mid p$, then $\pi \otimes (\chi \circ \det)$ satisfies conditions (1)-(3).
\end{proposition}

\begin{proof}
As mentioned above, conditions (1)-(3) are trivially invariant under twisting, so we turn to verifying decency. By definition, $x=(\pi,\alpha)$  decent if it lies on $\wt{\E}_{mid}$, if condition (2) above is satisfied, and if $H_f^1(G_F,\mathrm{ad}\,\rho_\pi)=0$. By Theorem \ref{consequence of ct}, condition (1) implies that $x$ lies on $\wt{\E}_{mid}$, and finally condition (3) implies that $H_f^1(G_F,\mathrm{ad}\,\rho_\pi)=0$  by \cite[Proposition 2.15(iv)]{bellaiche-crit} (when $F=\Q$ and $\pi$ has CM) and \cite[Theorem 5.4]{newton-thorne} (the remaining cases), so $(\pi,\alpha)$ is decent.
\end{proof}

For an extensive discussion of the decency condition we refer to \cite[\S 1.6]{bh}. As in \S \ref{subsec: p-adic functoriality} the pair $(\pi,\alpha)$ gives rise to a global refined $L$-packet $(\Pi,\gamma)$ of $\SL_2(\A_F)$, and we get classical points 
\[
x=(\pi,\alpha) \in \wt{\E} = \wt{\E}_{K^p}
\text{ and }
y = (\Pi,\gamma) \in \E_{K^p}.
\]
In light of Proposition \ref{localstructure1} and  Theorem \ref{main section 3} the choice of $K^p$ does not matter and we will drop it from the notation. We also abbreviate $S:=S(K)$. From those results, we may summarize our situation as follows: Let $\Lambda$ be the completed local ring of weight space at $k$. The completed local ring $\s_y$ of $\E$ at $y$ is the sub-$\Lambda$-algebra of $\prod_{z\in \Phi} \T_z$ generated by the $\SL_2$-Hecke operators, and lies inside the $C_\pi$-invariants. Here $\Phi$ is the set of refined automorphic representations $z = (\pi,\alpha_z)$ such that $\alpha_z$ maps to $\gamma$ and $\T_z$ is completed local ring of $\wt{\E}$ at $z$, and $C_\pi$ is the set of quadratic characters $\eta$ such that $\rho_\pi \otimes \eta \cong \rho_\pi$. The group $C_\pi$ is non-trivial if and only $\pi$ has CM, and in that case $C_\pi$ has order $2$ with the non-trivial element denoted by $\psi$. Since $\pi_v$ has non-zero Iwahori-fixed vectors for all $v\mid p$ (by definition), this means that:
\begin{itemize}
\item $\# \Phi = 2^r$, where $r$ is the number of places $v \mid p$ such that $\pi_v$ is an unramified principal series and $\alpha_v = -\beta_v$, where $\beta_v$ is the other $U_v$-eigenvalue on $\pi_v^{I_v}$.
\end{itemize}
Let us give a few details in the case when $\pi$ has CM by a quadratic CM extension $\wt{F}/F$. When $v$ is inert, the $L$-parameter of $\pi_v$ has the form $\chi \oplus \chi \psi_v$ for some unramified character $\chi$, and $\psi_v$ is the unramified quadratic character. In particular, the Satake parameters are $\chi(\varpi_v)$ and $-\chi(\varpi_v)$, so the ratio of the $U_v$-eigenvalues on $\pi_v^{I_v}$ is $-1$. When $v$ splits into $w \ol{w}$ say, we have $\rho_v = \chi_1 \oplus \chi_2$ for two unramified characters $\chi_1$ and $\chi_2$. The Satake parameters are $\chi_1(\varpi_w)$ and $\chi_2(\varpi_{\ol{w}})$ for uniformizers $\varpi_w$ of $\wt{F}_w$ \and $\varpi_{\ol{w}}$ of  $\wt{F}_{\ol{w}}$. When $F=\Q$ they have different slopes and so their ratio is always different from $-1$. In general however it can happen that $\chi_1(\varpi_w)=-\chi_2(\varpi_{\ol{w}})$. This is discussed in detail in \S \ref{subsec:CMHilbert}; see, in particular, Proposition \ref{behaviour of av at split primes}. 

\medskip

Let us now introduce the Galois deformation rings that will be relevant to us. Let $\rho = \rho_\pi$, and let $R_\rho$ be the (unrestricted) universal deformation ring of $\rho$ with fixed determinant $\det \rho$, as introduced (with slightly different notation) in \S \ref{subsec: def theory}. By Proposition \ref{smoothness}, $R_\rho$ is formally smooth of dimension $2d$. We will need a different description of the tangent space of $R_\rho$ than the one used in the proof of Proposition \ref{smoothness}. If $L/\Qp$ is a finite extension and $\sigma$ is finite dimensional continuous $G_L$-representation over $\olqp$, we write $H^1_f(G_L,\sigma)$ for the local (crystalline) Bloch--Kato Selmer group and $H_{/f}^1(G_L,\sigma)$ for the quotient
\[
H_{/f}^1(G_L,\sigma) := H^1(G_L,\sigma)/H_{f}^1(G_L,\sigma).
\]
There is a natural map
\begin{equation}\label{description of tangent space}
H^1(G_F,\mathrm{ad}^0\rho) \to \prod_{v\mid p} H_{/f}^1(G_{F_v},\mathrm{ad}^0\rho),
\end{equation}
which is injective since the kernel\footnote{Note that, since $\pi$  is cuspidal, $\rho_v$ is generic at all $v \nmid p$ by local-global compatibility. Hence the quotient $H^1(G_{F_v},\mathrm{ad}^0\rho)/H_{ur}^1(G_{F_v},\mathrm{ad}^0\rho) = 0$ for all $v\nmid p$, so $H_f^1(G_F,\mathrm{ad}^0\rho)$ is indeed the kernel to the map in (\ref{description of tangent space}).} $H^1_f(G_F,\mathrm{ad}^0\rho)$ vanishes by Proposition \ref{decency}. In fact, it is an isomorphism:

\begin{proposition}\label{description of tangent space no refinement}
The natural map (\ref{description of tangent space}) is an isomorphism.
\end{proposition}

\begin{proof}
From the proof of Proposition \ref{smoothness}, we know that $H^1(G_F,\mathrm{ad}^0\rho)$ has dimension $2d$. Since the map (\ref{description of tangent space}) is injective it suffices to show that $H_{/f}^1(G_{F_v},\mathrm{ad}^0\rho)$ has dimension $2[F_v:\Qp]$ for all $v \mid p$. By the local Euler characteristic formula we have
\[
\dim H^1(G_{F_v},\mathrm{ad}^0\rho) = \dim H^0(G_{F_v},\mathrm{ad}^0\rho) + \dim H^2(G_{F_v},\mathrm{ad}^0\rho) + 3[F_v:\Qp],
\]
and by \cite[Equation 3-2]{pottharst} we have
\[
\dim H_f^1(G_{F_v},\mathrm{ad}^0\rho) = \dim H^0(G_{F_v},\mathrm{ad}^0\rho) + \dim D_{dR}(\mathrm{ad}^0\rho_v)/D_{dR}^+(\mathrm{ad}^0\rho_v).
\]
It follows that 
\[
\dim H_{/f}^1(G_{F_v},\mathrm{ad}^0\rho) = \dim H^2(G_{F_v},\mathrm{ad}^0\rho) + 3[F_v:\Qp] - \dim D_{dR}(\mathrm{ad}^0\rho_v)/D_{dR}^+(\mathrm{ad}^0\rho_v).
\]
To conclude, we note that $ \dim D_{dR}(\mathrm{ad}^0\rho_v)/D_{dR}^+(\mathrm{ad}^0\rho_v) = [F_v:\Qp]$ and that $H^2(G_{F_v},\mathrm{ad}^0\rho)=0$ (the latter point follows since $D_{rig}(\rho_v)$ has a regular generic triangulation in the sense of \cite[Definition 3.4]{bergdall-smooth}; this is proved in the proof of \cite[Proposition 6.6.5]{bh}).
\end{proof}

Now let $z=(\pi,\alpha_z) \in \Phi$. We will recall the structure of the Hecke algebras $\T_z$ from \cite{bh}, where it is equated with a certain Galois deformation ring. For any finite place $v$ of $F$, let $\rho_v$ denote the restriction of $\rho$ to a decomposition group at $v$, and let $\mf{X}_v$ denote the (not necessarily representable) deformation functor for $\rho_v$ with fixed determinant $\det \rho_v$. Recall that we have a finite extension $E/\Qp$ --- the coefficient field --- which we assume to be large enough for our constructions. Let $\eta \in E$. Following \cite{bergdall-smooth}, we will define a subfunctor $\mf{X}_v^{\mathrm{Ref}} = \mf{X}_v^{\mathrm{Ref},\eta}$. Recall that $x$ has cohomological weight $k=(k_1,k_2)$, and consider the restriction $k_{2,v}$ of $k_2$ to $\mc{O}_v^{\times}$, via the natural embedding $\mc{O}^{\times}_v\hookrightarrow \mc{O}_p^{\times}$, and define $\mathrm{LT}(k_{2,v}):F_v^\times \rightarrow E^{\times}$ to be the extension of $k_{2,v}$ to $F^\times_v$ defined by sending $\varpi_v$ to $1$. As $\mathrm{LT}(k_{2,v})$ is bounded, it extends to a character of $G_{F_v}$ which we also denote by $\mathrm{LT}(k_{2,v})$. The subfunctor $\mf{X}_v^{\mathrm{Ref},\eta}\sub \mf{X}_v$ of weakly refined deformations (with respect to $\eta$) is defined by declaring that a deformation $\wt{\rho}_v$ of $\rho_v$ to $A$ is in $\mf{X}_v^{\mathrm{Ref},\eta}(A)$ if and only if 
\[
D^+_{\rm{crys}}(\widetilde{\rho}_v \otimes \mathrm{LT}(\kappa_{2,v}))^{\varphi^{f_v}=\wt{\eta}}
\]
is free of rank one for some lift $\wt{\eta}$ of $\varpi_v^{k_2}\eta$ to $A$. Here the character $\kappa_{2,v}: \mc{O}_v^{\times}\rightarrow A^{\times}$ is a lift of $k_{2,v}$ (viewed as a character on $\oo_v^\times$) and encodes (and is built out of) the Hodge--Tate--Sen weights of $\wt{\rho}_v$ which lift $-k_{2,v}$ (now viewed as a $\Sigma_v$-tuple of integers); we refer to \cite[\S 2.3, 2.4]{bergdall-smooth} for the details (we remark that $\wt{\rho}$ determines $\kappa_{2,v}$ uniquely). We note that, for any $\sigma \in \Sigma_v$, $k_{2,\sigma}$ is the smallest $\sigma$-Hodge--Tate--Sen weight of $\rho_v$. By construction, $\wt{\rho}_v \otimes \mathrm{LT}(\kappa_{2,v})$ has exactly one $\sigma$-Hodge--Tate--Sen weight equal to $0$ (for all $\sigma \in \Sigma_v$). Let $\alpha_v$ and $\beta_v$ be the $\varphi^{f_v}$-eigenvalues on $D_{\rm{crys}}(\rho_v)$ (with $\beta_v$ only considered if $\rho_v$ is crystalline). Then $\mf{X}_{v}^{\mathrm{Ref},\eta}$ is empty unless $\eta = \alpha_v $ or $\eta = \beta_v$, as $\varpi_v^{-k_{2,v}}\alpha_v $ and $\varpi_v^{-k_{2,v}}\beta_v$ are the eigenvalues for $\varphi^{f_v}$ acting on $D^+_{\rm{crys}}(\rho_v \otimes \mathrm{LT}(k_{2,v}))$. By \cite[Proposition 3.1]{bergdall-smooth}, valid also in the semistable non-crystalline case as pointed out in \cite[Footnote 26]{bh}, $\mf{X}_v^{\mathrm{Ref},\eta} \sub \mf{X}_v$ is relatively representable for $\eta =\alpha_v, \beta_v$. Now let $\eta = (\eta_v)_{v\mid p} \in \prod_{v\mid p} E$. We may then define a global weakly refined deformation functor
\[
\mf{X}_\rho^{\mathrm{Ref},\eta} \sub \mf{X}_\rho,
\]
where $\mf{X}_\rho$ is the unrestricted deformation functor for $\rho$ (represented by $R_\rho$), as the subfunctor of those deformations $\wt{\rho}$ of $\rho$ whose restriction to each $v\mid p$ lies in $\mf{X}_v^{\mathrm{Ref},\eta_v}$. By relative representability of the $\mf{X}_v^{\Reff,\eta_v}$, $\mf{X}_\rho^{\Reff,\eta}$ is representable by a quotient $R_\rho^{\mathrm{Ref},\eta}$ of $R_\rho$. We note that if $\eta$ is a refinement of $\pi$, then $R_\rho^{\Reff,\eta}$ surjects onto the completed local ring $\T_{(\pi,\eta)}$ of $\wt{\E}$, by the argument in the proof of \cite[Theorem 6.6.3]{bh} (the only difference being that we fix the determinant). Let $\mf{t}_\rho^{\Reff,\eta}$ and $\mf{t}_v^{\Reff,\eta_v}$ denote the tangent spaces of $\mf{X}_\rho^{\mathrm{Ref},\eta}$ and $\mf{X}_v^{\Reff,\eta_v}$, respectively.

\begin{proposition}\label{description of tangent spaces I}
Let $v\mid p$.

\begin{enumerate}
\item If $\eta_v = \alpha_v$ or $\beta_v$, then $H^1_f(G_{F_v},\ad^0 \rho_v) \sub \mf{t}_v^{\Reff,\eta_v}$ and the quotient $\mf{t}_{v, /f}^{\Reff,\eta_v}$ has dimension $[F_v:\Qp]$.

\item If $\eta$ is a refinement of $\pi$, then $R_\rho^{\Reff,\eta}$ is formally smooth of dimension $d$ and the surjection $R_\rho^{\Reff,\eta} \to \T_{(\pi,\eta)}$ is an isomorphism. In particular, $\T_{(\pi,\eta)}$ is topologically generated by the $T_v$-operators for $v\notin S$. 

\item If $\pi_v$ is an unramified principal series, then $H_{/f}^1(G_{F_v},\ad^0 \rho_v) = \mf{t}_{v, /f}^{\Reff,\alpha_v} \oplus \mf{t}_{v, /f}^{\Reff,\beta_v}$.
\end{enumerate}
\end{proposition}

\begin{proof}
We prove parts (1) and (2) together. Part (1) is essentially proved by \cite[Proposition 6.6.5, Lemma 6.6.7]{bh}, except that we only get $\dim \mf{t}_{v, /f}^{\Reff,\eta_v} \leq [F_v: \Qp]$. To prove equality, let $\eta = (\eta_v)_{v\mid p}$ to be a refinement of $\pi$, and then note that, from the definition and Proposition \ref{description of tangent space no refinement}, we have 
\[
\mf{t}_\rho^{\Reff,\eta} \cong \bigoplus_{v\mid p} \mf{t}_{v, /f}^{\Reff,\eta_v}.
\]
Thus $R_\rho^{\Reff,\eta}$ has a tangent space of dimension $\leq d$, but its quotient $\T_{(\pi,\eta)}$ has Krull dimension $d$ (for surjectivity of $R_\rho^{\Reff,\eta} \to \T_{(\pi,\eta)}$, see e.g. \cite[Proposition 4.3]{bergdall-smooth}). The equality in (1) follows, and (2) follows as well.

\medskip

For part (3) it suffices to prove that $\mf{t}_{v, /f}^{\Reff,\alpha_v} \cap \mf{t}_{v, /f}^{\Reff,\beta_v} = 0$ by part (1), since $\dim H_{/f}^1(G_{F_v},\ad^0 \rho_v) = 2[F_v:\Qp]$. But any deformation in $\mf{t}_{v}^{\Reff,\alpha_v} \cap \mf{t}_{v}^{\Reff,\beta_v}$ has to be crystalline, and hence lie in $H_f^1(G_{F_v},\ad^0\rho_v)$. This shows that $\mf{t}_{v, /f}^{\Reff,\alpha_v} \cap \mf{t}_{v, /f}^{\Reff,\beta_v} = 0$.
\end{proof}

We can now compute most of the structure of $\s_y$ in the non-CM case.

\begin{theorem}\label{main geometry non-CM}
Assume that $\pi$ does not have CM. Then the map $\s_y \to \T_z$ is surjective for every $z\in \Phi$, and the kernels are distinct. In particular, $\s_y$ has $2^r$ irreducible components, all of which are smooth.
\end{theorem}

\begin{proof}
Choose $z=(\pi,\alpha_z) \in \Phi$. We consider the surjective map $R_\rho \to R_\rho^{\Reff,\alpha_z} \cong \T_z$. By Theorem \ref{thethingIwant} (note $H_r$ is trivial as $\pi$ does not have CM) and Corollary \ref{defringsiso}, $R_\rho$ is topologically generated by the traces of Frobenius at $v\notin S$ of $\ad^0$ of the universal deformation, which map to $q_v^{-1}(t_v +1) \in \T_z$, $v\notin S$. In particular, these elements topologically generate $\T_z$ and are in the image of $\s_y \to \T_z$, proving the surjectivity. Since $\s_y$ is equidimensional of dimension $d$ and reduced, and the map
\[
\s_y \to \prod_{z\in \Phi} \T_z
\] 
is injective, we conclude that the kernels of the maps $\s_y \to \T_z$ account for all minimal primes in $\s_y$. It remains to show that they are distinct. For this, it suffices to show that the kernels of the maps $R_\rho \to R_\rho^{\Reff,\alpha_z}$ are distinct. But this follows immediately from the description of tangent spaces in Proposition \ref{description of tangent spaces I} and its proof.
\end{proof}

\subsection{Tangent spaces in the CM case}\label{subsec: CM case I}

After Theorem \ref{main geometry non-CM}, we focus our attention on the CM case, which is rather more interesting. So, for the rest of section \ref{sec: local geometry}, we assume that $\pi$ has CM by a (necessarily unique) quadratic CM extension $\wt{F}$ of $F$. As before, we write $\psi$ for the quadratic character corresponding to the extension $\wt{F}/F$. The extension $\wt{F}/F$ is unramified at all $v\mid p$ in $F$, and $\psi(Frob_v) = 1$ if $v$ splits in $\wt{F}$ and $\psi(Frob_v) = -1$ if $v$ is inert in $\wt{F}$. We will assume that our tame level $K^p$ is small enough so that $\psi \in H_K$. In particular, $\psi$ defines a twisting automorphism on the eigenvariety $\wt{\E} = \wt{\E}_K$ and if $\eta$ is a refinement of $\pi$, then the twist
\[
(\pi \otimes (\psi \circ \det), (\psi(\varpi_v)\eta_v)_v ) = (\pi, (\psi(\varpi_v)\eta_v)_v )
\]
of $(\pi,\eta)$ by $\psi$ is equal to $(\pi, \eta)$ if and only if all $v \mid p$ in $F$ split in $\wt{F}$. Consider $\rho = \rho_\pi$. Since $\rho \cong \rho \otimes \psi$, there is a matrix $X \in \GL_2(E)$ such that
\[
\rho(g) = X \rho(g)\psi(g) X^{-1}
\]
for all $g\in G_F$. In particular, if $A$ is an Artinian local $E$-algebra, then the rule
\begin{equation}\label{eq: twisting action of def spaces}
\wt{\rho} \mapsto \left( g \mapsto X \wt{\rho}(g)\psi(g) X^{-1} \right)
\end{equation}
gives a well defined involution on $\mf{X}_{\rho}(A)$ compatible with changing $A$ which is independent of the choice of $X$. We will refer to this action on $\mf{X}_\rho$ (and $R_\rho$) as `twisting by $\psi$'. Under the isomorphism
\[
\mf{X}_\rho \to \mf{X}_{\ad^0 \rho}
\]
from Corollary \ref{defringsiso}, the twisting involution corresponds to the conjugation action of the non-trivial element centralizing $\ad^0 \rho$ (here viewed as $\PGL_2$-valued representation), and we will use this without further comment. Note that the formula (\ref{eq: twisting action of def spaces}) also defines a `twisting action' on $\mf{X}_v$ for any $v\mid p$, such that the natural map $\mf{X}_\rho \to \mf{X}_v$ is equivariant. We stress here that we are using the matrix $X$ from (\ref{eq: twisting action of def spaces}) from the global situation (which is unique up to a scalar) to define the action on $\mf{X}_v$; we are not choosing an arbitrary matrix $X$ satisfying equation (\ref{eq: twisting action of def spaces}) only for $g$ in $G_{F_v}$.

\medskip

The goal of this subsection is to compute the twisting action of $\psi$ on the tangent space of $\mf{X}_\rho$; we will then use the results in \S \ref{subsec: CM case II} to understand $\s_y$. Recall from Proposition \ref{description of tangent space no refinement} that we have isomorphism
\[
H^1(G_F,\ad^0 \rho) \cong \bigoplus_{v\mid p} H_{/f}^1(G_{F_v},\ad^0 \rho).
\]
Under this isomorphism, the twisting action of $\psi$ corresponds to the local twisting actions on the summands defined above (note that these actions preserve $H_f^1(G_{F_v},\ad^0 \rho)$ and hence descend to $H_{/f}^1(G_{F_v},\ad^0 \rho)$). In particular, our goal is reduced to understanding the twisting action on $H_{/f}^1(G_{F_v},\ad^0 \rho)$. This divides naturally into two cases: $v$ inert in $\wt{F}$ or $v$ split in $\wt{F}$. 

\medskip

We begin with the case when $v$ is inert. In this case $\psi_v$ is the quadratic character corresponding to the unique unramified quadratic extension of $F_v$, and $\beta_v = -\alpha_v$. From Proposition \ref{description of tangent spaces I} we have the description
\[
H_{/f}^1(G_{F_v},\ad^0 \rho) = \mf{t}_{v,/f}^{\Reff, \alpha_v} \oplus \mf{t}_{v,/f}^{\Reff, -\alpha_v}.
\]
Since $D_{\rm{crys}}(\psi_v)$ is the filtered $\varphi$-module with $\varphi^{f_v}$ acting by $-1$ and Hodge--Tate weights all being $0$, we see that the twisting action of $\psi_v$ interchanges the two subspaces $\mf{t}_{v,/f}^{\Reff, \alpha_v}$ and $\mf{t}_{v,/f}^{\Reff, -\alpha_v}$. This more or less determines the action in this case. We record this observation in the form of a proposition.

\begin{proposition}\label{inert tangent space}
Let $v \mid p$, and assume that $v$ is inert in $\wt{F}$. Then the twisting action of $\psi_v$ on $H_{/f}^1(G_{F_v},\ad^0 \rho)$ interchanges the two subspaces $\mf{t}_{v,/f}^{\Reff, \alpha_v}$ and $\mf{t}_{v,/f}^{\Reff, -\alpha_v}$.
\end{proposition}

We now move on to the case when $v$ is split in $\wt{F}$, which is more interesting. In this case $\psi_v$ is trivial, but note however that the twisting action may not be trivial, since we are still conjugating by a non-scalar matrix in $\GL_2(E)$. We note that in this case the twisting action preserves the subspaces $\mf{t}_{v,/f}^{\Reff, \alpha_v}$  and $\mf{t}_{v,/f}^{\Reff, \beta_v}$, since the condition defining these subspaces only depends upon the isomorphism class of the deformation (which \emph{is} unchanged by the twisting action, unlike the deformation class in general). Since $\rho$ is induced from $G_{\wt{F}}$, we can conjugate it if necessary so that the \emph{homomorphism} 
\[
\rho : G_F \to \GL_2(E)
\]
lands inside the normalizer of the diagonal matrices. In particular, we may choose the matrix $X$ above to be
\[
X = \begin{pmatrix} -1 & 0 \\ 0 & 1 \end{pmatrix}.
\]
The restriction of $\rho$ to $G_{F_v}$ may then be written as  
\[
\rho_v = \begin{pmatrix} \chi_1 & 0 \\ 0 & \chi_2 \end{pmatrix} 
\]
for two characters $\chi_1, \chi_2 : G_{F_v} \to E^\times$. From this, we see that in the decomposition
\begin{equation}\label{eq: decomposition of H1ad}
H^1(G_{F_v},\ad\, \rho) = \bigoplus_{i,j=1,2} H^1(G_{F_v}, \chi_i \chi_j^{-1}),
\end{equation}
twisting acts trivially on the terms with $i=j$ and by $-1$ on the terms with $i \neq j$, and it preserves $H_f^1(G_{F_v},\ad\, \rho)$. To understand the relationship between this decomposition and the subspaces $\mf{t}_{v,/f}^{\Reff, \alpha_v}$ and $\mf{t}_{v,/f}^{\Reff, \beta_v}$, we now consider the map
\begin{equation}\label{twisting map for refinements}
H^1(G_{F_v}, \ad\, \rho) \cong \Ext^1(\rho_v,\rho_v) \to \Ext^1(\rho_v \otimes \mathrm{LT}(k_{2,v}), \rho_v \otimes \mathrm{LT}(k_{2,v}) ) \cong H^1(G_{F_v}, \ad\, \rho)
\end{equation}
which is obtained by sending a lift $\wt{\rho}_v$ of $\rho_v$ to $E[\epsilon]$ to $\wt{\rho}_v \otimes \mathrm{LT}(\kappa_{2,v})$, as in the definition of the weakly refined deformation functor. We note that this endomorphism of $H^1(G_{F_v}, \ad\, \rho)$ is equivariant for the twisting action (i.e. conjugation by $X$), but it is not injective. However, one has the following lemma:

\begin{lemma}\label{injectivity of twisting}
Let $L: H^1(G_{F_v}, \ad^0 \rho) \to H^1(G_{F_v}, \ad\, \rho)$ be the restriction of the map in equation (\ref{twisting map for refinements}) to the subspace $H^1(G_{F_v}, \ad^0 \rho) \sub H^1(G_{F_v}, \ad\, \rho)$. Then $L^{-1}(H_f^1(G_{F_v}, \ad\, \rho)) = H_f^1(G_{F_v}, \ad^0 \rho)$. In particular, $L$ induces an injection $H_{/f}^1(G_{F_v}, \ad^0 \rho) \to H_{/f}^1(G_{F_v}, \ad\, \rho)$ which we will also denote by $L$.
\end{lemma}

\begin{proof}
Let $\wt{\rho}_v$ be a lift of $\rho_v$ to $E[\epsilon]$ lying in $H^1(G_{F_v}, \ad^0 \rho)$, i.e. with determinant $\det \rho_v$. If $\wt{\rho}_v \otimes \mathrm{LT}(\kappa_{2,v})$ is crystalline, then its determinant $\det \rho_v \cdot \mathrm{LT}(\kappa_{2,v})^2$ is crystalline as well, which implies that $\mathrm{LT}(\kappa_{2,v})$ is crystalline and in fact equal to $\mathrm{LT}(k_{2,v})$, so $\wt{\rho}_v$ must be crystalline as well. Conversely, if $\wt{\rho}_v$ is crystalline, then $\kappa_{2,v} = k_{2,v}$ and the twist $\wt{\rho}_v \otimes \mathrm{LT}(\kappa_{2,v})$ is crystalline as well.
\end{proof}

We can also characterize the image of $L$. Momentarily, let $\lambda : G_{F_v} \to \GL_n(E)$ be a crystalline representation of dimension $n$, with \emph{distinct} Hodge--Tate weights $k_{\lambda,\sigma,i}$, $i=1,\dots,n$, for $\sigma \in \Sigma_v$. Every deformation $\wt{\lambda}$ of $\lambda$ to an Artinian $E$-algebra $A$ has Hodge--Tate weights $k_{\wt{\lambda},\sigma,i} \in A$ lifting $k_{\lambda,\sigma,i}$, for $\sigma \in \Sigma_v$ (cf. \cite[\S 2.3]{bergdall-smooth}). In particular, if $A=E[\epsilon]$ is the dual numbers, we get map 
\[
d_{HT}(\lambda) : H^1(G_{F_v},\ad\, \lambda) \to \bigoplus_{\sigma\in \Sigma_v} E^n
\]
sending a deformation $\wt{\lambda}$ over $E[\epsilon]$ to $(\epsilon^{-1}(k_{\wt{\lambda},\sigma,i} - k_{\lambda,\sigma,i}))_{\sigma\in \Sigma_v, i=1,\dots,n}$. The kernel of $d_{HT}(\lambda)$ contains $H_f^1(G_{F_v},\ad\, \lambda)$, so we get an induced map
\[
H_{/f}^1(G_{F_v},\ad\,\lambda) \to \bigoplus_{\sigma\in \Sigma_v} E^n
\]
which we will also call $d_{HT}(\lambda)$. When $\lambda = \chi$ is a character, i.e. $n=1$, then the kernel of $d_{HT}(\chi)$ is exactly $H_f^1(G_{F_v},\chi \chi^{-1})$ (see, for example, \cite[Theorem 3.12]{bergdall-smooth}) and $d_{HT}(\chi)$ induces an isomorphism
\[
H_{/f}^1(G_{F_v},\chi \chi^{-1}) \to \bigoplus_{\sigma\in \Sigma_v} E,
\]
e.g. by comparing dimensions, using \cite[Equation 3-2]{pottharst}. Thus, in this case, both spaces above have dimension $[F_v:\Qp]$, with one dimension corresponding to deforming each Hodge--Tate weight. In particular, given a subset $\Delta \sub \Sigma_v$, we can define $H^1(G_{F_v},\chi \chi^{-1})^\Delta$ to be the subspace of deformations whose $\sigma$-Hodge--Tate weight is constant for $\sigma \in \Delta$. By construction $H^1(G_{F_v},\chi \chi^{-1})^\Delta$ is the preimage of a subspace $H_{/f}^1(G_{F_v},\chi \chi^{-1})^\Delta \sub H_{/f}^1(G_{F_v},\chi \chi^{-1})$ of dimension $[F_v:\Qp] - \# \Delta$. For $i=1,2$, let $\Delta_i \sub \Sigma_v$ be the set of $\sigma$ such that the $\sigma$-Hodge--Tate weight $k_{\chi_i,\sigma}$ of $\chi_i$ is equal to $k_{2,\sigma}$. Note that $\Delta_1$ and $\Delta_2$ are disjoint and that their union is $\Sigma_v$.

\begin{lemma}\label{Image of L}
Consider the injective map $L : H_{/f}^1(G_{F_v},\ad^0 \rho) \to H_{/f}^1(G_{F_v},\ad\, \rho)$. Under the decomposition in equation (\ref{eq: decomposition of H1ad}), the image of $L$ is
\[
H_{/f}^1(G_{F_v}, \chi_1 \chi_1^{-1})^{\Delta_1} \oplus H_{/f}^1(G_{F_v}, \chi_1 \chi_2^{-1}) \oplus H_{/f}^1(G_{F_v}, \chi_2 \chi_1^{-1}) \oplus H_{/f}^1(G_{F_v}, \chi_2 \chi_2^{-1})^{\Delta_2}.
\]
In particular, this space has dimension $2[F_v : \Qp]$.
\end{lemma}

\begin{proof}
We start by proving the last statement. From the remarks above, the sum of the dimensions of $H_{/f}^1(G_{F_v}, \chi_1 \chi_1^{-1})^{\Delta_1}$ and $H_{/f}^1(G_{F_v}, \chi_2 \chi_2^{-1})^{\Delta_2}$ is $[F_v:\Qp]$. Then, a computation using the local Euler characteristic formula and \cite[Equation 3-2]{pottharst} shows that $H_{/f}^1(G_{F_v}, \chi_1 \chi_2^{-1})$ has dimension $\# \Delta_1$ and $H_{/f}^1(G_{F_v}, \chi_2 \chi_1^{-1})$ has dimension $\#\Delta_2$, since $k_{\chi_1,\sigma} \neq k_{\chi_2,\sigma}$ for all $\sigma$. This proves the dimension formula.

\medskip

To prove the rest of the lemma, it therefore suffices that prove that the image of $L$ is contained in the given subspace, since both spaces have the same dimension. For this, let $\wt{\rho}_v \in H^1(G_{F_v},\ad^0\rho)$ be a deformation to $E[\epsilon]$ and consider $L(\wt{\rho}_v) = \wt{\rho}_v \otimes \mathrm{LT}(\ka_{2,v})$. By construction, the Hodge--Tate weights of $\wt{\rho}_v \otimes \mathrm{LT}(\ka_{2,v})$ that reduce to $0$ modulo $\epsilon$ are in fact $0$. After applying $d_{HT}$, this precisely translates into the condition that the components of $L(\wt{\rho}_v)$ in $H_{/f}^1(G_{F_v}, \chi_1 \chi_1^{-1})$ and $H_{/f}^1(G_{F_v}, \chi_2 \chi_2^{-1})$ lie in $H_{/f}^1(G_{F_v}, \chi_1 \chi_1^{-1})^{\Delta_1}$ and $H_{/f}^1(G_{F_v}, \chi_2 \chi_2^{-1})^{\Delta_2}$, respectively, proving the lemma.
\end{proof}

We can now prove the main result on the twisting action in the split case. For the purpose of this theorem, we label the eigenvalues of $\varphi^{f_v}$ on $D_{\rm{crys}}(\rho_v)$ so that $\alpha_v$ is the eigenvalue on $D_{\rm{crys}}(\chi_1)$ and $\beta_v$ is the eigenvalue on $D_{\rm{crys}}(\chi_2)$. 

\begin{theorem}\label{split tangent space}
We use the notation as above. Then, in terms of the decomposition (\ref{eq: decomposition of H1ad}) and the notation of Lemma \ref{Image of L}, we have
\[
L(\mf{t}_{v,/f}^{\Reff, \alpha_v}) = H_{/f}^1(G_{F_v}, \chi_1 \chi_2^{-1}) \oplus H_{/f}^1(G_{F_v}, \chi_2 \chi_2^{-1})^{\Delta_2}
\]
and 
\[
L(\mf{t}_{v,/f}^{\Reff, \beta_v}) = H_{/f}^1(G_{F_v}, \chi_1 \chi_1^{-1})^{\Delta_1} \oplus H_{/f}^1(G_{F_v}, \chi_2 \chi_1^{-1}).
\]
In particular, the dimension of the $-1$-eigenspace of $\mf{t}_{v,/f}^{\Reff, \alpha_v}$ is equal to the number of $\sigma \in \Sigma_v$ for which the refinement $\alpha_v$ of $\rho_v$ is critical, and similarly for $\beta_v$.
\end{theorem}

\begin{proof}
(Cf. the proof of \cite[Theorem 2.16]{bellaiche-crit}.) We prove the statement for $\alpha_v$; the proof for $\beta_v$ is the same.  Consider $\wt{\rho}_v \in H^1(G_{F_v},\ad^0 \rho_v)$ and its image $L(\wt{\rho}_v)$ under $L$, viewed as an extension
\[
0 \to \rho_v \otimes \mathrm{LT}(k_{2,v}) \to L(\wt{\rho}_v) \to \rho_v \otimes \mathrm{LT}(k_{2,v}) \to 0.
\]
Since $\rho_v = \chi_1 \oplus \chi_2$ we get two extensions
\[
0 \to \rho_v \otimes \mathrm{LT}(k_{2,v}) \to \tau_1 \to \chi_2 \otimes \mathrm{LT}(k_{2,v}) \to 0
\]
and 
\[
0 \to \rho_v \otimes \mathrm{LT}(k_{2,v}) \to \tau_2 \to \chi_1 \otimes \mathrm{LT}(k_{2,v}) \to 0
\]
which, by Lemma \ref{Image of L} are elements of $H^1(G_{F_v},\chi_1\chi_2^{-1}) \oplus H^1(G_{F_v},\chi_2\chi_2^{-1})^{\Delta_2}$ and $H^1(G_{F_v},\chi_1\chi_1^{-1})^{\Delta_1} \oplus H^1(G_{F_v},\chi_2\chi_1^{-1})$, respectively, and $L(\wt{\rho}_v) = \tau_1 + \tau_2$. Now if $\wt{\rho}_v \in \mf{t}_{v}^{\Reff, \alpha_v}$, then $\tau_2$ is crystalline by \cite[Lemma 7.2]{bergdall-paraboline}, which shows that 
\[
L(\mf{t}_{v,/f}^{\Reff, \alpha_v}) \sub H_{/f}^1(G_{F_v}, \chi_1 \chi_2^{-1}) \oplus H_{/f}^1(G_{F_v}, \chi_2 \chi_2^{-1})^{\Delta_2}.
\]
Equality then follows since both spaces have the same dimension. Finally, we see that the $-1$-eigenspace of $\mf{t}_{v,/f}^{\Reff, \alpha_v}$ is (isomorphic to) $H_{/f}^1(G_{F_v}, \chi_1 \chi_2^{-1})$, whose dimension is equal to the number of $\sigma \in \Sigma_v$ for which $k_{\chi_1,\sigma} > k_{\chi_2,\sigma}$, i.e. the number of $\sigma$ for which $\alpha_v$ is critical.
\end{proof}

\subsection{The local structure in the CM case}\label{subsec: CM case II}

We now apply the computations of the previous subsection to the problem of understanding $\s_y$. We return to the setup from \S \ref{subsec: setup and non-cm}, and keep the assumption from \S \ref{subsec: CM case I} that $\pi$ has CM by $\wt{F}$. We will use the following notation throughout the rest of this paper: If $W$ is any $E$-vector space with an involution, we will write $W^+$ for the $+1$-eigenspace of the involution and $W^-$ for the $-1$-eigenspace of the involution, and we will write $W_+$ and $W_-$ for the largest quotients of $W$ on which the involution acts by $+1$ and $-1$, respectively. Note that the natural maps $W^+ \to W_+$ and $W^- \to W_-$ are isomorphisms --- this is sometimes useful and we will use it without further comment, though we will also want to distinguish between them. In what follows, the involution will always mean the twisting action coming from the non-trivial element $\psi \in C_\pi$.

\medskip

Next, we have the set $\Phi$ of points $z = (\pi,\alpha_z)$ mapping to $y$, and the inclusion
\[
\s_y \sub \prod_{z\in \Phi} \T_z.
\]
As in \S \ref{subsec: setup and non-cm}, we focus on the map
\[
R_\rho \to \prod_{z\in \Phi} R_\rho^{\Reff,\alpha_z} \cong \prod_{z\in \Phi} \T_z.
\]
We will use the identifications $\T_z \cong R_\rho^{\Reff,\alpha_z}$ throughout this subsection without further comment. Let $S_p$ denote the set of places of $F$ above $p$. We will use the following piece of notation: If $V$ is a finite dimensional $E$-vector space, then we write $E \llbracket V \rrbracket$ for the completion of the symmetric algebra $\Sym^{\bu}V$ along the ideal $\bigoplus_{n \geq 1} \Sym^n V$. $E \llbracket V \rrbracket$ is then isomorphic to a power series ring in $\dim V$ variables by choosing a basis for $V$. If $v\in S_p$, we set 
\[
V_v = H_{/f}^1(G_{F_v},\ad^0 \rho)^\vee.
\]
If $V_p := \bigoplus_{v\mid p} V_v$, then the isomorphism (\ref{description of tangent space}) gives us an isomorphism $R_\rho \cong E\llbracket V_p \rrbracket$. If $\eta = (\eta_v)_{v\mid p}$ is a refinement of $\pi$, we also define 
\[
V_{v,\eta_v} = (\mf{t}_{v, /f}^{\Reff,\eta_v})^\vee,
\]
which we view as a \emph{direct summand} of $V_v$ via Proposition \ref{description of tangent spaces I}(3). If we set
\[
V_{p,\alpha_z} = \bigoplus_{v\mid p} V_{v,\alpha_{z,v}},
\]
then we have $R_\rho^{\Reff,\alpha_z} \cong E \llbracket V_{p,\alpha_z} \rrbracket$.

\medskip

We divide $S_p$ into a disjoint union
\[
S_p = S_{1} \cup S_{2},
\] 
where $S_{1}$ is the set of $v\mid p$ such that $\alpha_v=-\beta_v$ and $S_{2}$ is the set $v\mid p$, where $\alpha_v\neq-\beta_v$. By definition, $\# S_1 = r$. Note that if $v$ is inert in $\wt{F}$, then $v \in S_1$. We denote by $S_{in}$ the set of inert places $v\mid p$ and by $S_{sp}$ the set of split places $v\mid p$. Our first statement about $\s_y$ concerns the number of its irreducible components.

\begin{theorem}
Let $f_z$ denote the natural map $R_\rho^+ \to R_\rho^{\Reff,\alpha_z}$. If $r=0$, then $\s_y$ is irreducible. If $r\geq 1$, then $\Ker f_z = \Ker f_{z^\prime}$ if and only if $z$ and $z^\prime$ lie in the same orbit of $\Phi$ under the action of $\psi$. As a consequence, the number of irreducible components of $\s_y$ equals the number of orbits of $\Phi$ under the action of $\psi$.
\end{theorem}

\begin{proof}
If $r=0$, then $\Phi = \{ x\}$ and $\s_y \sub \T_x$, proving the result since $\T_x$ is smooth. Assume that $r\geq 1$; we prove the ``if and only if'' statement. Since the map
$R_\rho \to \prod_{z\in \Phi} \T_z$ is equivariant with respect to twisting, and twisting identifies $\T_z$ and $\T_{z^\prime}$ when $z$ and $z^\prime$ are twists of one another, the ``if'' part follows. Now assume that $z$ and $z^\prime$ are not in the same orbit. We have two cases:
\begin{enumerate}
\item $\alpha_{z,v} = \alpha_{z^\prime,v}$ for all $v \in S_{sp}$;
\smallskip
\item $\alpha_{z,v} = - \alpha_{z^\prime,v}$ for some $v \in S_{sp}$.
\end{enumerate}
We start with the first case (note that we must have $r\geq 2$ in this case). Then we can find $v,v^\prime \in S_{in}$ such that $\alpha_{z,v} = \alpha_{z^\prime,v}$ but $\alpha_{z,v^\prime} = - \alpha_{z^\prime,v^\prime}$. Then choose $x_1 \in V_{v,\alpha_{z,v}}$ and $x_2 \in V_{v^\prime,\alpha_{z,v^\prime}}$, and let $y_1 \in V_{v,-\alpha_{z,v}}$ and $y_2 \in V_{v^\prime,-\alpha_{z,v^\prime}}$ be the twists of $x_1$ and $x_2$, respectively. By construction, $x_1 x_2 + y_1 y_2 \in R_\rho^+$. On the other hand, we have $x_1,x_2 \in V_{p,\alpha_z}$, $y_1,y_2 \notin V_{p,\alpha_z}$, $x_1,y_2 \in V_{p,\alpha_{z^\prime}}$ and $y_1,x_2 \notin V_{p,\alpha_{z^\prime}}$, also by construction. In particular, $x_1 x_2 + y_1 y_2$ maps to $x_1 x_2$ in $\T_z$ but maps to $0$ in $\T_{z^\prime}$, showing that the kernels of $f_z$ and $f_{z^\prime}$ are distinct, as desired.

\medskip

Now consider the second case. We have a split place $v\mid p$ such that $\alpha_{z,v} = -\alpha_{z^{\prime},v}$. To show that $\Ker f_z \neq \Ker f_{z^\prime}$, we consider the inclusions
\[
V_{v,\alpha_{z,v}} \oplus V_{v,\alpha_{z^{\prime},v}}\sub V_p \sub R_{\rho}
\]
from above. Without loss of generality, $V_{v,\alpha_{z,v}}^+ \neq 0$, so we may pick any element $x \neq 0$ in $V_{v,\alpha_{z,v}}^+$. Then $x \in R_{\rho}^{+}$, $f_z(x)\neq 0$ (as $f_z$ projects onto $V_{v,\alpha_{z,v}}^+$), but $f_{z^{\prime}}(x)=0$ and hence the kernels are indeed different.

\medskip

Finally, to show the statement about the number of irreducible components of $\s_y$, consider
\[
R_\rho^+ \to \s_y \sub A:= \left( \prod_{z\in \Phi}\T_z \right)^+.
\]
Note that $R_\rho^+$ does map into $\s_y$ by Theorem \ref{thethingIwant} and Corollary \ref{defringsiso}. If $S_{in}\neq \emptyset$, then $A = \prod_{z\in \Phi^\prime}\T_z$, where $\Phi^\prime \sub \Phi$ is a set of orbit representatives for the twisting action; if $S_{in} = \emptyset$, then $A = \prod_{z\in \Phi}\T^+_z$ (and all orbits are singletons). The inclusion $\s_y \sub A$ is finite, and both $\s_y$ and $A$ are equidimensional of dimension $d$. Thus, by lying over and going up \cite[Proposition 4.15, Corollary 4.18]{eisenbud}, minimal primes of $\s_y$ are exactly the intersections with the minimal primes in $A$, and hence exactly the kernels of the maps $\s_y \to \T_z$ for $z$ going through a set of orbit representatives of $\Phi$ (since the $\T_z$ are domains).
\end{proof}

It remains to determine the structure of the components, for which we have to determine the image of $\s_y \to \T_z$. The methods developed here will determine the image of $R_\rho^+ \to \T_z$. In general this image might be smaller, but in the cases below (Theorems \ref{main geometry CM noncritical} and \ref{main geometry CM critical}) we can prove equality and hence compute the structure of the components.

\begin{theorem}\label{main geometry CM noncritical}
Assume that $\alpha_z$ is non-critical at all $\sigma$ corresponding to places in $S_{sp}$, for all $z\in \Phi$. Then the components of $\s_y$ are smooth.
\end{theorem} 

\begin{proof}
The theorem follows from showing that $f_z : R_\rho^+ \to \T_z$ is surjective (for any $z\in \Phi$), since $\T_z$ is smooth. Note that the assumptions imply that $S_{sp}=S_2$ (cf. Remark \ref{rem: non-crit vs crit in split case}). Recall that $R_\rho \cong E\llbracket V_p \rrbracket$ and that $\T_z \cong E \llbracket V_{p,\alpha_z} \rrbracket$, and that when $\alpha_z$ is non-critical at all $\sigma$ over split places, 
\[
V_{p,\alpha_z} = \bigoplus_{v \in S_{in}} V_{v,\alpha_{z,v}} \oplus \bigoplus_{v\in S_{sp}}V_v^+
\]
(to see that $V_v^+ = V_{v,\alpha_{z,v}}$ for $v\in S_{sp}$, use Theorem \ref{split tangent space}). We need to show that the image of $R_\rho^+ \to \T_z$ contains a spanning set of $V_{p,\alpha_z}$. If $v\in S_{sp}$ and $x \in V_v^+$, then $x$ itself lies in $R_\rho^+$ and hence in the image of $f_z$. If $v\in S_{in}$ and $x \in V_{v,\alpha_{z,v}}$, let $y \in V_{v,-\alpha_{z,v}}$ denote its twist. Then $x+y \in R_\rho^+$ and $f_z(x+y)=x$, proving that $x$ lies in the image. By the description of $V_{p,\alpha_z}$ above, this finishes the proof.
\end{proof}

Our second main theorem, which we consider to be the most interesting one, concerns the case $S_{in}=\emptyset$.

\begin{theorem}\label{main geometry CM critical}
Assume that $S_{in}=\emptyset$. Then the maps $\s_y \to \T_z^+$ are surjective for all $z\in \Phi$. In particular, when $r=0$, $\s_y = \T_x^+$.
\end{theorem}

\begin{proof}
The maps $R_\rho \to \T_z$ are surjective and equivariant for the twisting action (which preserve the $\T_z$, since $S_{in} = \emptyset$). It follows that $R_\rho^+ \to \T_z^+$ is surjective, and hence so is $\s_y \to \T_z^+$, as desired. In particular, $\s_y = \T_x^+$ when $r=0$.
\end{proof}

Let us now look more closely at the structure of the components in Theorem \ref{main geometry CM critical}. Since all the relevant phenomena occur when $r=0$, we focus on this case to avoid heavier notation. We have $\T_x \cong E \llbracket V_{p,\alpha} \rrbracket$, so we see that $\T_x^+$ is the image of the natural map
\[
E \llbracket V_{p,\alpha}^+ \oplus \Sym^2 V_{p,\alpha}^- \rrbracket \to E \llbracket V_{p,\alpha} \rrbracket.
\]
If we set $m = \dim V_{p,\alpha}^+$, then $m$ is the number of $\sigma \in \Sigma_p$ for which $\alpha$ is non-critical at $\sigma$, and $n = \dim V_{p,\alpha}^-$ is the number of $\sigma \in \Sigma_p$ for which $\alpha$ is critical at $\sigma$. Choose bases $X_1,\dots, X_m$ of $V_{p,\alpha}^+$ and $Y_1,\dots,Y_n$ of $V_{p,\alpha}^-$. Then $\s_y$ is isomorphic to the subring
\[
E\llbracket X_i, Y_j Y_k \mid i=1,\dots,m,\, j,k = 1,\dots,n \rrbracket \sub E \llbracket X_1,\dots,X_m,Y_1,\dots,Y_n \rrbracket.
\]
This is a simple but yet interesting type of quotient singularity, whose geometric properties we recall in the appendix. In particular, we deduce the following varied behaviour of the local rings $\s_y$ from Proposition \ref{quotientsingularities}.

\begin{corollary}\label{main geometry CM critical cor}
Assume that $r=0$ and let $n$ be the number of $\sigma \in \Sigma_p$ such that $\alpha$ is critical at $\sigma$. Then the local ring $\s_y$ is normal and Cohen-Macaulay. Moreover:
\begin{enumerate}
\item If $n=0$ or $n=1$, then $\s_y$ is smooth.

\item If $n=2$, then $\s_y$ is a complete intersection but not regular.

\item If $n$ is odd and $\geq 3$, then $\s_y$ is $2$-Gorenstein (but not Gorenstein).

\item If $n$ is even and $\geq 4$, then $\s_y$ is Gorenstein but not a complete intersection.
\end{enumerate}

\end{corollary}

Let us finish this section by summarizing our results in the case $F=\Q$ (no longer assuming $S_{in}\neq \emptyset$). Recall that in this case we have $\alpha \neq -\beta$ if $p$ is split in $\wt{F}$, by Proposition \ref{behaviour of av at split primes}(1).

\begin{corollary}
Let $x=(\pi,\alpha)$ be a refined cuspidal cohomological automorphic representation for $\GL_2(\A_{\Q})$ and assume that $\pi_p^I$ does not have a repeated $U_p$-eigenvalue. Then the corresponding point $y$ on the $\SL_2$-eigencurve is smooth, unless $\pi$ does not have CM and $-\alpha$ is also a refinement of $\pi$, in which case there are two components going through $y$ (locally), both of which are smooth.
\end{corollary}

\section{$p$-adic endoscopic forms}\label{sec: endoscopic forms}

In this section we prove our main theorems on classicality of $p$-adic eigenforms with classical eigenvalues. The notation, setup and assumptions are the same as in \S \ref{sec: local geometry}. Additionally, we use the notation for $\pm 1$-eigenspaces of involutions etc., introduced in \S \ref{subsec: CM case II}.

\medskip

In particular we let $(\pi, \alpha)$ be a refined automorphic representation satisfying the assumptions stated at the beginning of Section \ref{subsec: setup and non-cm}. This gives rise to a refined $L$-packet $(\Pi,\gamma)$ of $\SL_2(\A_F)$ and a classical point $y= (\Pi,\gamma)$ on $ \mc{E}= \mc{E}_{K^p}$.
We begin in \S \ref{subsec: formal nhood} by deriving a formula for the family of smooth $\SL_2(\A_F^{p,\infty})$-representations on $\E$ in a (formal) neighbourhood of $y$. The main ingredients here are Lemma \ref{shrinking} and the computation of the fibres $\pi_{K}^{-1}(y)$. In \S \ref{subsec: computation of fibre} we combine this formula with the results on the local geometry at $y$ from \S \ref{sec: local geometry} to study the fibres.

\subsection{Computation of the sheaf in the formal neighbourhood}\label{subsec: formal nhood}

Recall from \S \ref{sec: eigenvarieties} that we have the eigenvarieties $\wt{\E}_{K^p}$ and $\E_{K^p}$ together with their graded coherent sheaves $\wt{\mc{M}}_{K^p}$ and $\mc{M}_{K^p}$. In section we describe the fibre of $\mc{M}_{K^p}$ at $y$, proving Theorem \ref{thm E}. To do this, we first establish a formula for the formal completion of $\mc{M}_{K^p}$ at $y$ in terms of $\wt{\mc{M}}_{K^p}$ in this subsection. Since the map
\[
Y_K^1 \to Y_K
\]
rarely identifies $Y_K^1$ with the identity component of $Y_K$, it turns out to be better to work with the direct limit $\varinjlim_{K^p} \mc{M}_{K^p}$ instead (locally around $y$). In many ways this is natural --- the relationship between $\SL_2$ and $\GL_2$ is cleaner at the level of automorphic representation than at the level of their fixed vectors.

\medskip

We start by recalling some material from \S \ref{sec: eigenvarieties}. First, recall that we have the common weight space $U$ from \S \ref{subsec: p-adic functoriality}, which is a small neighbourhood of the cohomological weight $k$ of $\pi$ that we may shrink as we see fit. Recall that $h$ is the slope of $x$ and $y$. By construction, we have
\[
\mc{M}_{K^p}(\s^K_{U,h}) = H^\ast(Y_K^1,\mc{D}_U)_{\leq h}
\]
and
\[
\wt{\mc{M}}_{K^p}(\T^K_{U,h}) = H^\ast(Y_K,\mc{D}_U)_{\leq h}.
\]
From (\ref{hecke algebras prelim}) we have the sequence of Hecke algebras
\[
\T^K_{U,h} \hookleftarrow \s_{U,h}^{\GL_2, K} \twoheadrightarrow \ol{\s}_{U,h}^K \twoheadleftarrow \s_{U,h}^K
\]
and their localizations
\[
\T^K_y \hookleftarrow \s_y^{\GL_2, K} \twoheadrightarrow \ol{\s}_y^K \twoheadleftarrow \s_y^K,
\]
and we recall that we have proved that the second and third maps are isomorphisms and that these $\SL_2$-Hecke algebras are independent of $K^p$; this common algebra is what we have called $\s_y$. The goal of this subsection is to compute
\[
\mc{M}^{\wedge}_y := \varinjlim_{K^p} \mc{M}^{\wedge}_{K^p,y}
\]
in terms of the corresponding objects for $\GL_2$, where $\mc{M}^{\wedge}_{K^p,y}$ denotes the formal completion of $\mc{M}_{K^p}$ at $y$ (we will use similar notation for formal completions of other coherent sheaves as well). More precisely, for $z\in \Phi$, set
\[
\wt{\mc{M}}_{z}^\wedge := \varinjlim_{K^p} \wt{\mc{M}}_{K^p,z}^\wedge.
\]
To actually define these direct limits, we need transition maps.

\begin{lemma}\label{transition maps}
Let $K^p_0 \sub K^p$ be compact open subgroups of $\GL_2(\A_F^{p\infty})$, and set $K = K^p I$ and $K_0 = K_0^p I$. Then the natural maps $H^\ast(Y_K^1,\mc{D}_U)_{\leq h} \to H^\ast(Y_{K_0}^1,\mc{D}_U)_{\leq h}$ and $H^\ast(Y_K,\mc{D}_U)_{\leq h} \to H^\ast(Y_{K_0},\mc{D}_U)_{\leq h}$ induce maps
\[
\mc{M}_{K^p,y}^\wedge \to \mc{M}_{K_0^p, y}^{\wedge} \,\,\,\,\, \text{and} \,\,\,\,\, \wt{\mc{M}}_{K^p,z}^\wedge \to \wt{\mc{M}}_{K_0^p, z}^\wedge
\]
of $\s_y$- and $\T_z$-modules, respectively (the latter for any $z\in \Phi$).
\end{lemma}

\begin{proof}
We construct the map $\mc{M}_{K^p, y}^\wedge \to \mc{M}_{K_0^p, y}^\wedge$; the other map is constructed in exactly the same way. We need to show that $H^\ast(Y_K^1,\mc{D}_U)_{\leq h} \to H^\ast(Y_{K_0}^1,\mc{D}_U)_{\leq h}$ induces a map
\[
H^\ast(Y_K^1,\mc{D}_U)_{\leq h} \otimes_{\s_{U,h}^K} \s_y \to H^\ast(Y_{K_0}^1,\mc{D}_U)_{\leq h} \otimes_{\s_{U,h}^{K_0}} \s_y.
\]
Using the map $\s_{U,h}^{K_0} \to \s_{U,h}^K$, we get a map
\[
H^\ast(Y_K^1,\mc{D}_U)_{\leq h} \otimes_{\s_{U,h}^{K_0}} \s_y \to H^\ast(Y_{K_0}^1,\mc{D}_U)_{\leq h} \otimes_{\s_{U,h}^{K_0}} \s_y
\]
and it remains to show that $H^\ast(Y_K^1,\mc{D}_U)_{\leq h} \otimes_{\s_{U,h}^{K_0}} \s_y = H^\ast(Y_K^1,\mc{D}_U)_{\leq h} \otimes_{\s_{U,h}^K} \s_y$. Recall that $\Lambda$ is the completed local ring of $U$ at $k$. Set $\s_{\Lambda}^K = \s_{U,h}^K \otimes_{\oo(U)} \Lambda$ and similarly for $K_0$; then we are reduced to showing that
\[
(H^\ast(Y_K^1,\mc{D}_U)_{\leq h} \otimes_{\oo(U)} \Lambda) \otimes_{\s_\Lambda^{K_0}} \s_y \to (H^\ast(Y_K^1,\mc{D}_U)_{\leq h} \otimes_{\oo(U)} \Lambda) \otimes_{\s_\Lambda^K} \s_y
\]
is an isomorphism. This then follows if we know that $\s_\Lambda^K \otimes_{\s_\Lambda^{K_0}} \s_y = \s_y$, which is equivalent to knowing that $y$ is the only point on $\E_{K^p}$ mapping to $y$ on $\E_{K_0^p}$ (here we are already using that $\s_y = \s_y^{K_0} = \s_y^K$). Let $y^\prime$ be a point on $\E_{K^p}$ mapping to $y$ on $\E_{K_0^p}$. Then $y^\prime$ lies on $\E_{K^p,mid}^{\SL_2}$ by Lemma \ref{vanishing for SL(2)}, and hence has an associated Galois representation by Proposition \ref{Galois rep SL2}, so $y^\prime = y$ as desired by Chebotarev.
\end{proof}

Then we have the following:

\begin{proposition}\label{computation of sheaf}
We have
\[
\mc{M}_{y}^\wedge = \left( \prod_{z\in \Phi}    \wt{\mc{M}}_{z}^\wedge \right)_+
\]
and it is concentrated in degree $d$.
\end{proposition}

\begin{proof}
We start by considering the diagram 
\[
\begin{tikzcd}
H^\ast(Y_{K_0},\mc{D}_U)_{\leq h,H_{K_0}} \ar{r} & H^\ast(Y^1_{K_0},\mc{D}_U)_{\leq h}  \\
H^\ast(Y_K,\mc{D}_U)_{\leq h,H_K} \ar{r} \ar{u} & H^\ast(Y^1_K,\mc{D}_U)_{\leq h} \ar{u} \arrow[lu,dashed]
\end{tikzcd}
\]
where the dashed arrow exists for small enough $K_0 = K_0^p I$ by Lemma \ref{shrinking} and all arrows are injective. Tensoring with $\s_y$ and using the vanishing results (Theorem \ref{consequence of ct} and Lemma \ref{vanishing for SL(2)}), we obtain the diagram
\[
\begin{tikzcd}
H^d(Y_{K_0},\mc{D}_U)_{\leq h,H_{K_0}} \otimes_{\s_{U,h}^{K_0}} \s_y \ar{r} & H^d(Y^1_{K_0},\mc{D}_U)_{\leq h} \otimes_{\s_{U,h}^{K_0}} \s_y  \\
H^d(Y_K,\mc{D}_U)_{\leq h,H_K} \otimes_{\s_{U,h}^K} \s_y \ar{r} \ar{u} & H^d(Y^1_K,\mc{D}_U)_{\leq h} \otimes_{\s_{U,h}^K} \s_y \ar{u}.  \arrow[lu,dashed]
\end{tikzcd}
\]
Here the horizontal arrows are injective by flatness of $\s_y$, the dashed arrow and the right vertical map exist and are injective by (the proof of) Lemma \ref{transition maps} and flatness, and the left vertical map exists since the composition of the lower horizontal map and the right vertical map lands inside $H^d(Y_{K_0},\mc{D}_U)_{\leq h,H_{K_0}} \otimes_{\s_{U,h}^{K_0}} \s_y$. In particular, the solid arrows exist for all $K_0^p \sub K^p$, not just those for which the dashed arrow exists. Taking direct limits, we get that 
\[
\mc{M}_y^\wedge = \varinjlim_{K^p} (H^d(Y_K,\mc{D}_U)_{\leq h,H_K} \otimes_{\s_{U,h}^K} \s_y).
\]
Next, note that
\[
H^d(Y_K,\mc{D}_U)_{\leq h,H_K} \otimes_{\s_{U,h}^K} \s_y = H^d(Y_K,\mc{D}_U)_{\leq h,H_K} \otimes_{\ol{\s}_{U,h}^K} \s_y = H^d(Y_K,\mc{D}_U)_{\leq h,H_K} \otimes_{\s_{U,h}^{\GL_2,K}} \s_y.
\]
Moreover,
\[
H^d(Y_K,\mc{D}_U)_{\leq h,H_K} \otimes_{\s_{U,h}^{\GL_2,K}} \s_y = \left( H^d(Y_K,\mc{D}_U)_{\leq h} \otimes_{\s_{U,h}^{\GL_2,K}} \s_y \right)_{H_K}
\]
and we have
\[
H^d(Y_K,\mc{D}_U)_{\leq h} \otimes_{\s_{U,h}^{\GL_2,K}} \s_y =  H^d(Y_K,\mc{D}_U)_{\leq h} \otimes_{\T_{U,h}^{K}} \T_y^K  = \prod_{z\in \pi_K^{-1}(y)} H^d(Y_K,\mc{D}_U)_{\leq h} \otimes_{\T_{U,h}^K} \T_z.
\]
Now note that  
\[
\varinjlim_{K^p} \left( \prod_{z\in \pi_K^{-1}(y)} H^d(Y_K,\mc{D}_U)_{\leq h} \otimes_{\T_{U,h}^K} \T_z \right)_{H_K} = \varinjlim_{K^p} \left( \prod_{z\in \Phi} H^d(Y_K,\mc{D}_U)_{\leq h} \otimes_{\T_{U,h}^K} \T_z \right)_{+}
\]
by the same type of argument as in Proposition \ref{localstructure1}. Summing up, we see that 
\[
\mc{M}_y^\wedge = \varinjlim_{K^p} \left( \prod_{z\in \Phi} H^d(Y_K,\mc{D}_U)_{\leq h} \otimes_{\T_{U,h}^K} \T_z \right)_{+} = \left( \prod_{z\in \Phi} \wt{\mc{M}}_z^\wedge \right)_+
\]
as desired (which is concentrated in degree $d$).
\end{proof}

\subsection{Computation of the fibre}\label{subsec: computation of fibre}

We start by comparing the spaces of classical forms of $\GL_2$ and $\SL_2$. Write $H^\ast(Y_K^1,\mathscr{L}_k)_y$ for the quotient of  $H^\ast(Y_K^1,\mathscr{L}_k)$ by $\p_y$, and similarly for $\GL_2$. Since the Hecke action is semisimple, $H^\ast(Y_K^1,\mathscr{L}_k)_y$ is also (canonically isomorphic to) the localization of $H^\ast(Y_K^1,\mathscr{L}_k)$ at $\p_y$, or the $\p_y$-torsion submodule $H^\ast(Y_K^1,\mathscr{L}_k)[\p_y]$.

\begin{proposition}\label{classical spaces}
We have
\[
\varinjlim_{K^p} H^\ast(Y_K^1,\mathscr{L}_k)_y = \left( \prod_{z\in \Phi} \varinjlim_{K^p} H^\ast(Y_K, \mathscr{L}_k)_z \right)_+
\]
and both sides are concentrated in degree $d$.
\end{proposition}

\begin{proof}
Both $H^\ast(Y_K^1,\mathscr{L}_k)_y$ and $H^\ast(Y_K, \mathscr{L}_k)_z$ are concentrated in degree $d$ by cuspidality of $\pi$. By Proposition \ref{limit} we have
\[
\varinjlim_{K^p} H^\ast(Y_K^1,\mathscr{L}_k)_y = \varinjlim_{K^p} H^\ast(Y_K,\mathscr{L}_k)_{y,H_K}.
\]
A similar (but simpler) calculation to that in the proof of Proposition \ref{computation of sheaf} then gives us
\[
\varinjlim_{K^p} H^\ast(Y_K,\mathscr{L}_k)_{y,H_K} = \varinjlim_{K^p} \left( \prod_{z\in \Phi} H^\ast(Y_K,\mathscr{L}_k)_z \right)_+
\]
and commuting the direct limit with the product and the coinvariants gives us the proposition.
\end{proof}

Let $\m_y \sub \s_y$ and $\m_z \sub \T_z$ for $z \in \Phi$ be the maximal ideals of these local rings. If $\m_k\sub \oo(U)$ is the maximal ideal corresponding to $k$, then natural map 
\[
H^d(Y_K,\mc{D}_U)_{\leq h} \otimes_{\oo(U)} \oo(U) / \m_k \to H^d(Y_K,\mc{D}_k)_{\leq h}
\]
is an isomorphism (by the Tor-spectral sequence of \cite[Theorem 3.3.1]{hansen} and vanishing outside degree $d$). By Proposition \ref{locfree}, $\wt{\mc{M}}_{K^p,z}^\wedge$ is locally free over $\T_z$. Thus, by Corollary \ref{3.2.13}, the integration map
\[
 \wt{\mc{M}}_{K^p,z} := \wt{\mc{M}}_{K^p,z}^\wedge \otimes_{\T_z}\T_z /\m_z =  H^d(Y_K,\mc{D}_k)_{\leq h} \otimes_{\T(K)} \T(K) / \p_z \to H^\ast(Y_K,\mathscr{L}_k)_z
\]
is an isomorphism. We want to compute the fibre
\[
\mc{M}_{K^p,y} := \mc{M}_{K^p,y}^\wedge \otimes_{\s_y} \s_y /\m_y,
\]
which is the maximal semisimple quotient of $H^d(Y_K^1,\mc{D}_k)_{\leq h}$ on which $\s(K)$ acts through $\phi_y$ (the `co-eigenspace'). To do this, we will divide into cases and combine Proposition \ref{computation of sheaf} with the results on the structure of $\s_y$ from \S \ref{sec: local geometry}.

\medskip

To start with, let us assume that $\pi$ does not have CM. Then Proposition \ref{computation of sheaf} reads
\[
\mc{M}_y^\wedge = \prod_{z\in \Phi} \wt{\mc{M}}_{z}^\wedge,
\]
corresponding to the embedding of rings $\s_y \sub \prod_{z\in \Phi} \T_z$, and each $\s_y \to \T_z$ is surjective by Theorem \ref{main geometry non-CM}. Set $\mc{M}_y = \mc{M}_y^\wedge \otimes_{\s_y} \s_y / \m_y$; note that $\mc{M}_y^{K^p} = \mc{M}_{K^p,y}$. This gives us the following result.

\begin{theorem}\label{classicality non-CM}
Assume that $\pi$ does not have CM. Then the integration map $\mc{M}_y \to \varinjlim_{K^p} H^\ast(Y_K^1, \mathscr{L}_k)_y$ is an isomorphism. 
\end{theorem}

\begin{proof}
From the discussion above, we see that
\[
\mc{M}_y = \mc{M}_y^\wedge \otimes_{\s_y} \s_y/\m_y = \prod_{z\in \Phi} \wt{\mc{M}}_{z}^\wedge \otimes_{\s_y} \s_y/ \m_y = \prod_{z\in \Phi} \wt{\mc{M}}_{z}^\wedge \otimes_{\T_z} \T_z/ \m_z = \prod_{z\in \Phi} \wt{\mc{M}}_z.
\]
Each $\wt{\mc{M}}_z$ maps isomorphically onto $\varinjlim_{K^p} H^\ast(Y_K, \mathscr{L}_k)_z$ via the integration map. Combining this with Proposition \ref{classical spaces} gives the result.
\end{proof}

Now assume that $\pi$ has CM by a quadratic extension $\wt{F}/F$. Recall that $S_{in}$ and $S_{sp}$ are the sets of places $v\mid p$ which are inert and split in $\wt{F}$, respectively.

\begin{theorem}\label{classicality CM non critical}
Assume that $S_{in}\neq \emptyset$ and that $\alpha_z$ is non-critical at all $\sigma$ corresponding to places in $S_{sp}$, for all $z\in \Phi$. Then the integration map $\mc{M}_y \to \varinjlim_{K^p} H^\ast(Y_K^1, \mathscr{L}_k)_y$ is an isomorphism. 
\end{theorem}

\begin{proof}
The proof is essentially identical to that of Theorem \ref{classicality non-CM}. Let $\Phi^\prime \sub \Phi$ be a set of orbit representatives for the twisting action. Then, since $S_{in}\neq \emptyset$,
\[
\mc{M}_y^\wedge = \prod_{z\in \Phi^\prime} \wt{\mc{M}}_{z}^\wedge,
\]
corresponding to the inclusion $\s_y \sub \prod_{z\in \Phi^\prime} \T_z$, and each $\s_y \to \T_z$ is surjective by Theorem \ref{main geometry CM noncritical}. Moreover, 
\[
\varinjlim_{K^p} H^\ast(Y_K^1,\mathscr{L}_k)_y = \prod_{z\in \Phi^\prime} \varinjlim_{K^p} H^\ast(Y_K, \mathscr{L}_k)_z
\]
by Proposition \ref{classical spaces}. One may then argue exactly as in the proof of Theorem \ref{classicality non-CM} (changing $\Phi$ to $\Phi^\prime$).  
\end{proof}

Finally, we come to the most interesting case. Assume that $S_{in} = \emptyset$. By Theorem \ref{main geometry CM critical}, $\s_y$ surjects onto $\T_z^+$ for all $z\in \Phi$. By Proposition \ref{computation of sheaf}, we then have
\[
\mc{M}_y^\wedge = \prod_{z\in \Phi}(\wt{\mc{M}}_z^\wedge)_+.
\]
Set $\n_z = \m_y\T_z$ and note that $\m^2_z \sub \n_z \sub \m_z$. Tensoring the exact sequence
\[
0 \to \m_z/\n_z \to \T_z / \n_z \to \T_z / \m_z \to 0 
\]
with the flat $\T_z$-module $\wt{\mc{M}}_z^\wedge$ we obtain the exact sequence
\[
0 \to \m_z \wt{\mc{M}}_z^\wedge /\n_z \wt{\mc{M}}_z^\wedge \to \wt{\mc{M}}_z^\wedge / \n_z \wt{\mc{M}}_z^\wedge \to \wt{\mc{M}}_z \to 0.
\]
Taking coinvariants, we get
\[
0 \to (\m_z \wt{\mc{M}}_z^\wedge /\n_z \wt{\mc{M}}_z^\wedge)_+ \to (\wt{\mc{M}}_z^\wedge / \n_z \wt{\mc{M}}_z^\wedge)_+ \to (\wt{\mc{M}}_z)_+ \to 0,
\]
and taking the product over all $z\in \Phi$ we get the short exact sequence
\[
0 \to \prod_{z\in \Phi}(\m_z \wt{\mc{M}}_z^\wedge /\n_z \wt{\mc{M}}_z^\wedge)_+ \to \prod_{z\in \Phi}(\wt{\mc{M}}_z^\wedge / \n_z \wt{\mc{M}}_z^\wedge)_+ \to \prod_{z\in \Phi}(\wt{\mc{M}}_z)_+ \to 0.
\]
We compute the three terms of this short exact sequence:

\begin{lemma}\label{terms in the SES}
We have $\SL_2(\A_F^{p\infty})$- and $u_v$-equivariant (for all $v\mid p$) isomorphisms:
\begin{enumerate}
\item $\prod_{z\in \Phi} (\wt{\mc{M}}_z)_+ = \left( \prod_{z\in \Phi} \varinjlim_{K^p} H^\ast(Y_K,\mathscr{L}_k)_z \right)_+ = \varinjlim_{K^p} H^\ast(Y_K^1, \mathscr{L}_k)_y$;

\item $\prod_{z\in \Phi}(\wt{\mc{M}}_z^\wedge / \n_z \wt{\mc{M}}_z^\wedge)_+ = \mc{M}_y$;

\item $\prod_{z\in \Phi}(\m_z \wt{\mc{M}}_z^\wedge /\n_z \wt{\mc{M}}_z^\wedge)_+ \cong  \prod_{z\in \Phi} \left( \varinjlim_{K^p} H^\ast(Y_K,\mathscr{L}_k)_z \right)_-^{\oplus n_z}$, where $n_z = \dim_E \m_z/\n_z$.

\item In fact, $\left( \varinjlim_{K^p} H^\ast(Y_K,\mathscr{L}_k)_z \right)_- \cong \left( \varinjlim_{K^p} H^\ast(Y_K,\mathscr{L}_k)_x \right)_-$ for all $z\in \Phi$, so
\[
\prod_{z\in \Phi}(\m_z \wt{\mc{M}}_z^\wedge /\n_z \wt{\mc{M}}_z^\wedge)_+ \cong \left( \varinjlim_{K^p} H^\ast(Y_K,\mathscr{L}_k)_x \right)_-^{\oplus n},
\]
where $n=\sum_{z\in \Phi}n_z$.
\end{enumerate}
\end{lemma}

\begin{proof}
Part (1) follows by Proposition \ref{classical spaces} since $\wt{\mc{M}}_z = \varinjlim_{K^p} H^\ast(Y_K,\mathscr{L}_k)_z $ for all $z\in \Phi$ via the integration map.

\medskip

For part (2), we have
\[
\prod_{z\in \Phi}(\wt{\mc{M}}_z^\wedge / \n_z \wt{\mc{M}}_z^\wedge)_+ = \prod_{z\in \Phi}(\wt{\mc{M}}_z^\wedge / \m_y \wt{\mc{M}}_z^\wedge)_+ = \prod_{z\in \Phi} (\wt{\mc{M}}_z^\wedge)_+ / \m_y (\wt{\mc{M}}_z^\wedge)_+ = \mc{M}_y.
\]

We now prove part (3). To start with, we analyse $\m_z / \n_z$, and in particular the dimension $n_z := \dim \m_z / \n_z$. First, in this proof only, write $\p_z$ for $\m_y \T_z^+$; this is the maximal ideal of $\T_z^+$ and we have $\n_z = \p_z \T_z$. Next, choose coordinates such that
\[
\T_z \cong E \llbracket X_1,\dots,X_m,Y_1,\dots,Y_n \rrbracket
\]
with the $X_i$ $+1$-eigenvectors for the twisting action and the $Y_j$ $-1$-eigenvectors. Then $\T_z^+$ corresponds to the subring  
\[
E \llbracket X_i,Y_j Y_k \mid i=1,\dots,m, \,\, j,k=1,\dots,n \rrbracket,
\]
giving generators for $\p_z$. In particular, the quantity labelled $n$ is $n_z$, and it is the dimension of $-1$-eigenspace of $\m_z/\m_z^2$ or, equivalent, its dual, the tangent space. By Theorem \ref{split tangent space}, $n_z$ is the number of $\sigma \in \Sigma_p$ for which $\alpha_{z}$ is critical.

\medskip

Next, we write $\m_z \wt{\mc{M}}_z^\wedge /\n_z \wt{\mc{M}}_z^\wedge = \m_z/\n_z \otimes_{\T_z} \wt{\mc{M}}_z^\wedge$ by flatness of $\wt{\mc{M}}_z^\wedge$ over $\T_z$. Since $\m_z^2 \sub \n_z$, the $\T_z$-action on $\m_z/\n_z$ factors through $\T_z/\m_z =E$ and we get
\[
\m_z/\n_z \otimes_{\T_z} \wt{\mc{M}}_z^\wedge = \m_z/\n_z \otimes_{E} \wt{\mc{M}}_z.
\]
Since $\m_z/\n_z$ is dual to the $-1$-eigenspace of the tangent space $(\m_z/\m_z^2)^\vee$, twisting acts by $-1$ on $\m_z/\n_z$. As $n_z = \dim_E \m_z/\n_z$, we obtain
\[
(\m_z \wt{\mc{M}}_z^\wedge /\n_z \wt{\mc{M}}_z^\wedge)_+ = (\m_z/\n_z \otimes_{E} \wt{\mc{M}}_z)_+ \cong \left( \varinjlim_{K^p} H^\ast(Y_K,\mathscr{L}_k)_z \right)_-^{\oplus n_z}.
\]
Taking the product over all $z\in \Phi$ then finishes the proof.

\medskip

It remains to prove part (4). Fix $z\in \Phi$ and consider the left $\GL_2(\A^\infty)$-module
\[
V = \varinjlim_{K} H^\ast(Y_K,\mathscr{L}_k),
\]
where now $K$ runs through all compact open subgroups of $\GL_2(\A^\infty)$. The $\pi^\infty$-part $V_\pi$ of $V$ is a direct summand, and we have
\[
\varinjlim_{K^p} H^\ast(Y_K,\mathscr{L}_k)_z = V_\pi^I [U_v = \alpha_{z,v} \mid v|p ],
\]
and similarly for $x$. Consider the twisting action of $\psi$ on $V$, which preserves $V_\pi$. Since $\psi$ is trivial on $F_p^\times$, a calculation analogous to that in the proof of Lemma \ref{twisting and Hecke} shows that the $\GL_2(F_p)$-action on $V$ intertwines the twisting action. In particular, the action of the full Iwahori-Hecke algebra (at $p$) on $V_{\pi}^I$ is $\psi$-equivariant. Since $\pi_p$ is (absolutely) irreducible, we can choose an element $Q$ of the Iwahori-Hecke algebra which is an isomorphism of $\pi_p^I$ and exchanges $\pi_p^I [U_v = \alpha_{z,v} \mid v|p ]$ and $\pi_p^I [U_v = \alpha_{x,v} \mid v|p ]$. The action of $Q$ on $V_\pi^I$ is then invertible and $\psi$- and $\GL_2(\A^{p,\infty}_F)$-equivariant and sends $V_\pi^I [U_v = \alpha_{z,v} \mid v|p ]$ to $V_\pi^I [U_v = \alpha_{x,v} \mid v|p ]$. Applying $(-)_-$ then completes the proof of the first part of (4), and the second part is a direct consequence.
\end{proof}

We summarize what we have proved, and elaborate further on the statement in Remark \ref{interpretation}.

\begin{theorem}\label{non-classicality CM critical}
For each $z\in \Phi$, let $n_z$ be the number of $\sigma \in \Sigma_p$ for which $\alpha_z$ is critical at $\sigma$. Set $n= \sum_{z\in \Phi} n_z$. The fibre $\mc{M}_y$ sits in a short exact sequence
\[
0 \to \left( \varinjlim_{K^p} H^\ast(Y_K,\mathscr{L}_k)_x \right)_-^{\oplus n} \to \mc{M}_y \to \varinjlim_{K^p} H^\ast(Y_K^1, \mathscr{L}_k)_y \to 0.
\]
In particular, the non-classical part of the $L$-packet of $\pi$ contributes with $p$-adic, non-classical, forms to the fibre $\mc{M}_y$ when $n$ is positive.
\end{theorem}

\begin{remark}\label{interpretation}
We make some remarks on Theorem \ref{non-classicality CM critical}, mainly concerning the interpretation of the exact sequence. See also Remark \ref{interpretation quat}.

\begin{enumerate}
\item Because of our setup, $\mc{M}_{K^p,y}$ is the maximal quotient of $H^d(Y_K^1,\D_k)$ on which $\s(K)$ acts via $\phi_y$, rather than the $\phi_y$-eigenspace. It is more common, and easier in terms of established vocabulary, to think of classical spaces of automorphic forms as \emph{subspaces} of $p$-adic automorphic forms, rather than quotients as in the distribution-valued cohomology that we use here. This slight niggle can be solved by applying (smooth) duality. Applying this to the exact sequence in Theorem \ref{non-classicality CM critical}, we get  
\[
0 \to \varinjlim_{K^p} (H^\ast(Y_K^1, \mathscr{L}_k)_y)^\vee  \to (\mc{M}_y)^\vee \to  \varinjlim_{K^p} \left(  H^\ast(Y_K,\mathscr{L}_k)_x \right)_-^{\oplus n, \vee}\to 0.
\]
With some work, including duality between $H^\ast(Y_K,\D_k)$ and $H_\ast(Y_K,\mc{A}_k)$ (cf. the proof of \cite[Proposition 3.1.5]{hansen}), Poincar\'e duality (in the form \cite[Theorem III.3.11]{bellaiche-eigenbook}) as well as the interaction between the integration map and duality, one obtains an exact sequence 
\[
0 \to \varinjlim_{K^p} H^\ast(Y_K^1, \mathscr{L}_k)[\phi_y]  \to \varinjlim_{K^p} H^\ast(Y_K^1, \mc{A}_k)[\phi_y] \to  \varinjlim_{K^p} \left(  H^\ast(Y_K,\mathscr{L}_k)[\phi_x] \right)_-^{\oplus n}\to 0.
\]
This puts us back in a setting where we are discussing eigenspaces rather than quotients. This passage between $\mc{A}$'s and $\D$'s using duality is entirely analogous to the passage between completed homology and completed cohomology in Emerton's construction of eigenvarieties.

\medskip

\item Let us now discuss the term $\left( \varinjlim_{K^p} H^\ast(Y_K,\mathscr{L}_k)_x \right)_-^{\oplus n}$ from Theorem \ref{non-classicality CM critical}. We have
\[
\varinjlim_{K^p} H^\ast(Y_K,\mathscr{L}_k)_x \cong \pi^{p\infty} \otimes \bigotimes_{v\mid p} \pi_v^{I_v}[U_v=\alpha_v] \otimes \bigotimes_{v\mid \infty} H^\ast(\mf{g}_v,K_v,\pi_\infty \otimes \mathscr{L}_k)
\]
as representations of $\GL_2(\A_F^{p\infty})$, the $\GL_2$-Iwahori-Hecke algebra at $p$, and $\prod_{v\mid \infty} \pi_0(\GL_2(\R))$, where, for $v\mid \infty$, $\mf{g}_v$ is the Lie algebra of $\GL_2(F_v)$ and $K_v$ is the maximal compact connected subgroup of $\GL_2(F_v)$. The twisting action on $\varinjlim_{K^p} H^\ast(Y_K,\mathscr{L}_k)_x$ breaks up into a product of local twisting actions on the right hand side, with $\varinjlim_{K^p} (H^\ast(Y_K^1, \mathscr{L}_k)_y) = \varinjlim_{K^p} (H^\ast(Y_K,\mathscr{L}_k)_x)_+$ being the automorphic part for $\SL_2$ and $\varinjlim_{K^p} (H^\ast(Y_K,\mathscr{L}_k)_x)_-$ being the non-automorphic part. The local $L$-packets at places $v\mid p$ are singletons, so the action at those places are trivial, whereas the local action at a place $v\mid \infty$ breaks up the two-dimensional space $H^\ast(\mf{g}_v,K_v,\pi_\infty \otimes \mathscr{L}_k)$ into one-dimensional $+1$- and $-1$-eigenspaces. In particular, one sees that 
\[
\varinjlim_{K^p} (H^\ast(Y_K,\mathscr{L}_k)_x)_+ \cong \varinjlim_{K^p} (H^\ast(Y_K,\mathscr{L}_k)_x)_-
\]
abstractly as representations for $\SL_2(\A_F^{p\infty})$ and the $\SL_2$-Iwahori-Hecke algebra at $p$. One might thus wonder about our interpretation of Theorem \ref{non-classicality CM critical}, where we say that the non-automorphic part of the $L$-packet contributes to the non-classical part of the fibre, but we hope the proof together with the remarks above justify our interpretation. To further strengthen this point we present what happens in the case of definite quaternion algebras in the next section, where the $+1$- and $-1$-eigenspaces \emph{are non-isomorphic} (see Remark \ref{interpretation quat}). We also note that one could try to reinstate the difference between $\varinjlim_{K^p} (H^\ast(Y_K,\mathscr{L}_k)_x)_+$ and
$\varinjlim_{K^p} (H^\ast(Y_K,\mathscr{L}_k)_x)_-$ abstractly by taking into account some kind of action at the infinite places, which remembers whether the forms are holomorphic or anti-holomorphic discrete series at infinity. It is not clear to the authors what kind of action at infinity that one should expect a ``$p$-adic automorphic representation'' to have, so we do not pursue this here. Nevertheless, it seems likely to us that there should be such an action, and that it would allow us to distinguish between $\varinjlim_{K^p} (H^\ast(Y_K,\mathscr{L}_k)_x)_+$ and $\varinjlim_{K^p} (H^\ast(Y_K,\mathscr{L}_k)_x)_-$.

\medskip

\item In the case when $S_{in} \neq \emptyset$ and $\alpha_v$ is critical at some $\sigma\in \Sigma_v$ for $v$ split, we currently lack the necessary control over the local geometry to determine whether there are non-classical forms in the fibre or not. We suspect is that there are non-classical forms and that the components are singular.

\medskip

\item Finally, we remark that \cite{L1} uses Emerton's construction of eigenvarieties via completed (co)homology, whereas we have used overconvergent cohomology. One might reasonably wonder whether this makes a difference for the end result. The answer to this should be no, in a very strong sense: The $d$-dimensional parts of both eigenvarieties are canonically isomorphic, as are the fibres of the coherent sheaves on the $d$-dimensional parts. The isomorphism of $d$-dimensional parts of eigenvarieties follows from a standard argument using the interpolation theorem, whereas the isomorphism of fibres of coherent sheaves is unpublished work of one of us (C.J) and David Hansen. Moreover, that work strongly suggests that if one replaces Emerton's eigenvariety by a suitable ``derived'' version (perhaps that of \cite{fu}), the whole eigenvarieties and their coherent sheaves should be isomorphic. For groups which are compact modulo centre at infinity, this comparison is essentially contained in \cite{loeffler}. In light of these results and heuristics, we regard the choice between overconvergent cohomology and completed cohomology as a matter of convenience.
\end{enumerate}
\end{remark}

\section{Definite quaternion algebras}\label{sec: quat}

As is well known, it can be advantageous to work with non-split inner forms of $\GL_2$ (or $\SL_2$) instead of the split form. In particular, if $F$ has even degree over $\Q$, there is a unique quaternion algebra $B$ with centre $F$ which is split at all finite places and ramified at all infinite places. In this case, the Jacquet--Langlands correspondence gives us a canonical bijection between infinite dimensional automorphic representations of $(B\otimes_F \A_F)^\times$ and cuspidal automorphic representation of $\GL_2(\A_F)$ that are discrete series at all infinite places. In this section, we will very briefly outline the results that our methods give if one works with $B$ instead of $\GL_2$. The main differences are the following:
\begin{enumerate}
\item One can drop the condition that $\ol{\rho}_\pi$ is irreducible and generic;

\item Automorphic and non-automorphic members of the same $L$-packet for the corresponding inner form of $\SL_{2/F}$ can no longer be isomorphic at all finite places; cf. Remark \ref{interpretation}.
\end{enumerate}

In particular, point (2) makes it clear that, even abstractly, the extra contribution comes from the non-automorphic part of the $L$-packet.

\subsection{Setup and assumptions}

In this section, we assume that $[F:\Q]$ is even. As above, we let $B$ denote the quaternion algebra with centre $F$ which is split at all finite places and ramified at all infinite places. Let $\mathrm{Nm} : B \to F$ be the reduced norm map. We define algebraic groups $G_B$ and $G_B^1$ over $F$ by
\[
G_B(R) = (B \otimes_F R)^\times
\]
and
\[
G_B^1(R) = \Ker (\mathrm{Nm} : G_B(R) \to R^\times).
\]
Then $G_B$ is an inner form of $\GL_{2/F}$ and $G_B^1$ is an inner form of $\SL_{2/F}$. By assumption $B \otimes_F F_v \cong M_2(F_v)$ ($2 \times 2$-matrices) for all finite places $v$, and we will fix such isomorphisms and use them to equate $G_B(F_v)$ and $G_B^1(F_v)$ with $\GL_2(F_v)$ and $\SL_2(F_v)$, respectively (and similarly for $\A_F^S$-points, where $S$ is a finite set of places containing the infinite places). Let us now very briefly recall overconvergent cohomology and eigenvarieties for $G_B$ and $G_B^1$. We will use the same weight spaces as we did for $\GL_2$ and $\SL_2$, and the distribution modules are exactly the same as $G_B(F_v) = \GL_2(F_v)$ and $G_B^1(F_v) = \SL_2(F_v)$ for all $v\mid p$. If $K \sub \GL_2(\A_F^\infty)$ is a compact open subgroup, the locally symmetric space for $G_B$ with level $K$ is
\[
X_K := G_B(F)^\circ \backslash \GL_2(\A_F^\infty) / K,
\]
where $G_B(F)^\circ = G_B(F) \cap \prod_{v\mid \infty} G_B(F_v)^\circ$. This is a finite discrete topological space, and it carries a surjection
\[
X_K \to Cl^+_K
\]
induced by the reduced norm. If $\alpha \in Cl_K^+$, we will denote the preimage of $\alpha$ in $X_K$ by $X_K^{\mathrm{Nm}=\alpha}$. The centre of $G_B$ can be canonically identified with the centre $Z$ of $\GL_{2/F}$, and as for $\GL_{2/F}$ we set $Z(K) = Z(F) \cap K$. Just as in \S \ref{subsec: symmetric spaces}, any right $K/Z(K)$-module $N$ induces a local system on $X_K$ which we will also denote by $N$. In particular, all the local systems we used on $Y_K$ have counterparts on $X_K$, corresponding to the same $K/Z(K)$-module. Similarly, for any compact open $K \sub \SL_2(\A_F^\infty)$, we have the locally symmetric space
\[
X_K^1 := G_B^1(F) \backslash \SL_2(\A_F^\infty) / K
\]
for $G_B^1$, which is also finite and discrete. As before, if $K \sub \GL_2(\A_F^\infty)$ is compact open, we will write $X_K^1 := X_{K\cap \SL_2(\A_F^\infty)}^1$. If $K\sub \GL_2(\A_F^\infty)$ and $K^\prime \sub \SL_2(\A_F^\infty)$ are compact opens with $K^\prime \sub K$, then we have an induced map
\[
X_{K^\prime}^1 \to X_K 
\]
whose image lands inside $X_K^{\mathrm{Nm}=1}$. We will use the same notation for compact open subgroups as in \S \ref{subsec: symmetric spaces} and the rest of this paper. In particular, if $\m\sub \oo_F$ is an ideal coprime to $p$, we write $X_\m$ for $X_{K^p(\m)I}$ (and define $X_\m^1$ similarly). We record the analogue of Lemma \ref{shrinking}.

\begin{lemma}\label{shrinking quaternion}
Let $K^p \sub \GL_2(\A_F^{p\infty})$ be a compact open subgroup and set $K = K^p I$. For sufficiently small compact opens $K_0^p \sub \GL_2(\A_F^{p\infty})$, the dashed arrow in the diagram
\[
\begin{tikzcd}
X_{K_0}^1 \ar{r} \ar{d} & X_{K_0}^{\mathrm{Nm}=1} \ar{d} \arrow[ld,dashed] \\
X_{K}^1 \ar{r} & X_{K}^{\mathrm{Nm}=1}
\end{tikzcd}
\]
exists, where $K_0 = K_0^p I$.
\end{lemma}

\begin{proof}
One may reduce to the case $K^p = K^p(\m)$ and $K_0^p = K^p(\n)$ for some ideals $\n \sub \m \sub \oo_F$ coprime to $p$. Let
\[
C(\m) = \{ x \in \wh{\oo}_F \mid x \equiv 1 \,\, (\m) \}.
\]
Then
\[
X_\m^{\mathrm{Nm}=1} = \{ [g] \in G_B(F)^\circ \backslash \GL_2(\A_F^\infty) / K^p(\m)I \mid \mathrm{Nm}(g) \in (F^\times)^\circ \cdot C(\m) \}.
\]
Recall the notation $U(\m)_+ = C(\m) \cap (F^\times)^\circ$ for the totally positive units congruent to $1$ modulo $\m$. Now assume that $\n$ is sufficiently small that $U(\n)_+ \sub U(\m)_+^2$; this is possible by \cite[Th\'eor\`eme 1]{chevalley}.Under this assumption, we may define a map $X_\n^{\mathrm{Nm}=1} \to X_\m^1$ as follows: If $[g] \in X_\n^{\mathrm{Nm}=1}$, one may choose $\gamma \in G_B(F)^\circ$ and $k \in K^p(\n)I$ such that $\mathrm{Nm}(\gamma g k) = 1$, and we send $[g]$ to $[\gamma g k]$. One then easily checks that this is well defined and makes the diagram above commute.
\end{proof}

We can now define overconvergent and classical cohomology using the same local systems (and weight spaces) that we used for $\GL_2$ and $\SL_2$. We get eigenvarieties $\wt{\E}_B$ for $G_B$ and $\E_B$ for $G_B^1$, which are equidimensional and reduced of dimension $d$, since $X_K$ and $X_K^1$ only have cohomology in degree $0$. It is standard to interpolate the Jacquet--Langlands correspondence (cf.\ \cite{chenevier-jl}, \cite[5.1]{hansen} and \cite{JN2}) to closed immersions
\[
\wt{\E}_B \to \wt{\E}, \,\,\,\,\, \E_B \to \E,
\]
which identifies $\wt{\E}_B$ with the union of the $d$-dimensional irreducible components of $\wt{\E}$ (by contrast, $\E_B$ is a union of $d$-dimensional irreducible components of $\E$, but it is subtle to determine if it is all of them). 

\medskip

Finally, let us briefly recall the functorial transfer from $G_B$ to $G_B^1$, which is exactly parallel to that from $\GL_2$ to $\SL_2$. Given an infinite dimensional automorphic representation $\pi$ of $G_B(\A_F)$, we may form the corresponding $L$-packet $\Pi$ of $G_B^1(\A_F)$-representations by taking all irreducible subquotients of the restriction of $\pi$ to $G_B^1(\A_F)$. The main difference from the case of $\GL_2$ and $\SL_2$ is that, if we write
\[
\pi = \pi^\infty \otimes \pi_\infty
\]
as the tensor product of a representation $\pi^\infty$ of $\GL_2(\A_F^\infty)$ and a representation $\pi_\infty$ of $G_B(F_\infty)$, then $\pi_\infty$ remains irreducible when restricted to $G_B^1(F_\infty)$ (since $G_B(F_\infty)$ is the product of its centre and $G_B^1(F_\infty)$). In particular, decomposing $\pi|_{G_B^1(\A_F)}$ is equivalent to decomposing $\pi^\infty |_{\SL_2(\A_F^\infty)}$.

\subsection{Results}

With this setup, all of the arguments that we used for $\GL_2$ and $\SL_2$ hold essentially verbatim, with the important simplification that we do not have to worry about having overconvergent cohomology in multiple degrees. Let $\pi$ be an infinite dimensional automorphic representation of $G_B(\A_F)$, corresponding to a cuspidal automorphic representation on $\GL_2(\A_F)$ of cohomological weight $k$ via the Jacquet--Langlands correspondence. We make the following assumptions on $\pi$:

\begin{enumerate}

\item For every $v\mid p$, the $U_v$-eigenvalues on $\pi_v^{I_v}$ have multiplicity $1$;

\item If $\pi$ has complex multiplication by a quadratic CM extension $\wt{F}/F$, then $\wt{F} \not\subseteq F(\zeta_{p^\infty})$.

\end{enumerate}

Compared to working with $\GL_2$, we no longer need the assumption that $\ol{\rho}_\pi$ is irreducible and generic, since we only have cohomology in degree $0$. Let $\alpha$ be a refinement of $\pi$ and let $x=(\pi,\alpha)$ be the corresponding classical point on $\wt{\E}_B$. Let $\Pi$ be the $L$-packet of $G_B^1(\A_F)$ corresponding to $\pi$ and let $\gamma$ be the refinement induced by $\alpha$; we write $y=(\Pi,\gamma)$ for the corresponding classical point on $\E_B$ (and we assume that we have chosen tame levels so that these points appear as classical points). As before, $\s_y$ denotes the completed local ring of $\E_B$ at $y$, $\T_x$ denotes the completed local ring of $\wt{\E}_B$ at $x$, $\Phi$ denotes the set of $z=(\pi,\alpha^{\prime}) \in \wt{\E}_B$ that are mapped to $y$ and $\# \Phi = 2^r$ is the number of refinements of $\pi$ whose corresponding refinement of $\Pi$ is $\gamma$ (as before). We then obtain exactly the same results on the geometry of $\s_y$ as we did for $\SL_2$; we summarize them as follows:

\begin{theorem}\label{main geometry non-CM quat}
Assume that $\pi$ does not have CM. Then $\s_y$ has $2^r$ irreducible components, all of which are smooth.
\end{theorem}

\begin{theorem}
Assume that $\pi$ has CM by $\wt{F}/F$. Then the following holds
\begin{enumerate}
\item If $r=0$, then $\s_y$ is irreducible. If $r\geq 1$, the number of irreducible components of $\s_y$ equals the number of orbits of $\Phi$ under the twisting action.

\smallskip

\item Assume that $\alpha_z$ is non-critical at all $\sigma$ corresponding to places in $S_{sp}$, for all $z\in \Phi$, then all irreducible components of $\s_y$ are smooth.

\smallskip

\item Assume that all places $v \mid p$ of $F$ are split in $\wt{F}$. Then the maps $\s_y \to \T_z^+$ are surjective for all $z\in \Phi$ and in particular, when $r=0$, $\s_y = \T_x^+$.

\end{enumerate} 
\end{theorem}

In this, we are using the fact that, for $z\in \Phi$, the rings $\T_z$ are smooth. This is not strictly speaking contained in the smoothness results of \cite{bh} since the transfer of $z$ might not lie on $\wt{\E}_{mid}$, but the proof goes through since $\wt{\E}_B$ is equidimensional of dimension $d$. 

\medskip

Next, we discuss the analogues of the results of \S \ref{sec: endoscopic forms}. We begin with the computation of the classical eigenspace.

\begin{proposition}\label{classical spaces quat}
We have
\[
\varinjlim_{K^p} H^0(X_K^1,\mathscr{L}_k)_y = \left( \prod_{z\in \Phi} \varinjlim_{K^p} H^0(X_K, \mathscr{L}_k)_z \right)_+.
\]
\end{proposition}

We let $\mc{M}_{B, K^p}$ denote the coherent sheaf on $\E_B$, characterized by the property
\[
\mc{M}_{B,K^p}(\Spa \s_{U,h}^{K}) = H^0(X_{K}^1, \D_U)_{\leq h}
\] 
for any slope adapted pair $(U,h)$. Similar to before, $\mc{M}_{B, K^p, y}$ denotes the fibre of $\mc{M}_{B, K^p}$ at $y$, and set $\mc{M}_{B,y} = \varinjlim_{K^p} \mc{M}_{B,K^p, y}$. Then we have the following results for the structure of $\mc{M}_{B,y}$.

\begin{theorem}\label{classicality non-CM quat}
Assume that $\pi$ does not have CM. Then the integration map $\mc{M}_{B,y} \to \varinjlim_{K^p} H^0(X_K^1, \mathscr{L}_k)_y$ is an isomorphism. 
\end{theorem}

\begin{theorem}\label{classicality CM non critical quat}
Assume that $\pi$ has CM and that $\alpha_z$ is non-critical at all $\sigma$ corresponding to places in $S_{sp}$, for all $z\in \Phi$. Then the integration map $\mc{M}_{B,y} \to \varinjlim_{K^p} H^0(X_K^1, \mathscr{L}_k)_y$ is an isomorphism. 
\end{theorem}

\begin{theorem}\label{non-classicality CM critical quat}
Assume that $\pi$ has CM by $\wt{F}/F$ and that all places $v \mid p$ in $F$ split in $\wt{F}$. For each $z\in \Phi$, let $n_z$ be the number of $\sigma \in \Sigma_p$ such that $\alpha_z$ is critical at $\sigma$, and set $n=\sum_{z\in \Phi} n_z$. Then $\mc{M}_{B,y}$ sits in a short exact sequence
\[
0 \to \left( \varinjlim_{K^p} H^0(X_K,\mathscr{L}_k)_x \right)_-^{\oplus n} \to \mc{M}_{B,y} \to \varinjlim_{K^p} H^0(X_K^1, \mathscr{L}_k)_y \to 0.
\]
\end{theorem}

To finish, we follow up on Remark \ref{interpretation}(2). We continue to assume that we are in the situation of Theorem \ref{non-classicality CM critical quat}. Let $\Pi^{aut} \sub \Pi$ denote the subset consisting of automorphic members of the $L$-packet, and let $\Pi^{non-aut} = \Pi \setminus \Pi^{aut}$. Then, from Proposition \ref{classical spaces quat}, the definitions and multiplicity one, we have
\[
\varinjlim_{K^p} H^0(X_K^1, \mathscr{L}_k)_y = \left( \prod_{z\in \Phi}\varinjlim_{K^p} H^0(X_K, \mathscr{L}_k)_z \right)_+ \cong \bigoplus_{\sigma \in \Pi^{aut}} \left( \sigma^{p,\infty} \otimes \bigotimes_{v\mid p} \sigma_v^{I_{0,v}}[u_v = \gamma_v] \right)
\]
and \[
\left( \prod_{z\in \Phi} \varinjlim_{K^p} H^0(X_K, \mathscr{L}_k)_z \right)_- \cong \bigoplus_{\sigma \in \Pi^{non-aut}} \left( \sigma^{p,\infty} \otimes \bigotimes_{v\mid p} \sigma_v^{I_{0,v}}[u_v = \gamma_v] \right)
\]
as $\SL_2(\A_F^\infty)$-representations. The following proposition then shows that, as $\SL_2(\A_F^{p,\infty})$-representations, these two spaces have no common subrepresentations.

\begin{proposition}
Assume that $\sigma \in \Pi^{aut}$ and $\tau \in \Pi^{non-aut}$. Then $\sigma^{p,\infty} \not\cong \tau^{p,\infty}$.
\end{proposition}

\begin{proof}
Let $A = \varinjlim_{K} H^0(X_K,\ms{L}_k)$ denote the space of smooth automorphic forms of $G_B(\A_F^\infty)$ of weight $k$. By multiplicity one, there is a unique subrepresentation of $A$ isomorphic to $\pi^\infty$; we will conflate $\pi^\infty$ with this subspace of $A$. Since the $L$-packets for $\SL_2(F_v)$ are multiplicity free for all finite $v$, $\pi^\infty$ is multiplicity free when considered as a representation of $\SL_2(\A^{p,\infty})\times \GL_2(F_p)$. Now the twisting action on $A$, which is defined by
\[
(\psi \otimes f)(g) = f(g)\psi(\det(g)),
\]
preserves $\pi^\infty$ since $\pi \cong \pi \otimes (\psi \circ \det)$. Since $\wt{F}/F$ is split at all $v\mid p$, $\psi$ is trivial on $F_p^\times$ and hence $\psi$ intertwines the $\SL_2(\A^{p,\infty})\times \GL_2(F_p)$-action on $A$. It follows that $(\pi^\infty)_+$ and $(\pi^\infty)_-$ are $\SL_2(\A^{p,\infty})\times \GL_2(F_p)$-representations, and hence have no isomorphic $\SL_2(\A^{p,\infty})\times \GL_2(F_p)$-subrepresentations. Since
\[
(\pi^\infty)_+ \cong \bigoplus_{\sigma \in \Pi^{aut}} \sigma^{p,\infty} \otimes \pi_p
\]
and
\[
(\pi^\infty)_- \cong \bigoplus_{\tau \in \Pi^{non-aut}} \tau^{p,\infty} \otimes \pi_p
\]
as $\SL_2(\A^{p,\infty})\times \GL_2(F_p)$-representations, the proposition follows.
\end{proof}

\begin{remark}\label{interpretation quat}
This shows that the sub- and quotient representations in the short exact sequence
\[
0 \to  \left( \varinjlim_{K^p} H^0(X_K,\mathscr{L}_k)_x \right)_-^{\oplus n} \to \mc{M}_{B,y} \to \varinjlim_{K^p} H^0(X_K^1, \mathscr{L}_k)_y \to 0
\]
of $\SL_2(\A_F^{p,\infty})$-representations have no common subquotients, and that the non-classical part is a contribution from the non-automorphic part of the $L$-packet. Compare with Remark \ref{interpretation}(2).
\end{remark}

\begin{remark}
This provides an explanation for the phenomenon observed in \cite{L2}, to which we refer for details on the terminology used in this remark (the setting there is slightly different, but the key fact is that we work with a quaternion algebra that is ramified at all infinite places). Consider an $L$-packet $\Pi$ for which there is an idempotent $e$ of the type considered in \cite{L2} that picks out a particular member $\sigma \in \Pi$. The corresponding eigenvariety of idempotent type $e$ is a Zariski closed subvariety of the eigenvariety which contains $y$, and its corresponding coherent sheaf will have fibre equal to $e\mc{M}_{B,y}$, sitting in the exact sequence
\[
0 \to e \left( \varinjlim_{K^p} H^\ast(X_K,\mathscr{L}_k)_x \right)_-^{\oplus n} \to e \mc{M}_{B,y} \to e \left( \varinjlim_{K^p} H^\ast(X_K^1, \mathscr{L}_k)_y \right) \to 0.
\]
Thus $e\mc{M}_{B,y} \cong \sigma^{p,\infty} \otimes \bigotimes_{v\mid p} \sigma_v^{I_{0,v}}[u_v = \gamma_v]$ if $\sigma$ is automorphic, and
\[
e\mc{M}_{B,y} \cong \left( \sigma^{p,\infty} \otimes \bigotimes_{v\mid p} \sigma_v^{I_{0,v}}[u_v = \gamma_v] \right)^{\oplus n}
\]
if $\sigma$ is non-automorphic. This gives a general form of (the analogue of) the main theorem of \cite{L2}, for the quaternion algebras considered here. 
\end{remark}

\appendix

\section{Miscellaneous results}\label{sec:appendix}

\subsection{CM Hilbert modular forms}\label{subsec:CMHilbert}

When $f=\sum_n a_n q^n$ is a (classical) CM modular eigenform of weight $k\geq 2$ with respect to an imaginary quadratic field $\wt{F}$, then $a_\ell = 0$ for primes $\ell$ which are inert in $\wt{F}$, and $a_\ell \neq 0$ for primes $\ell$ that split in $\wt{F}$. In this subsection we discuss the corresponding behaviour for CM Hilbert modular forms, which is slightly more subtle.

\medskip

As in the main text, let $F$ be a totally real field, and let $\wt{F}/F$ be a totally imaginary quadratic extension. Let $\psi$ be an algebraic character of $\wt{F}$ of weights $(k_\sigma)_\sigma$, i.e., a group homomorphism 
\[
\psi : \A_{\wt{F}}^\times \to \Qbar^\times
\] 
with open kernel satisfying $\psi(a) = \prod_{\sigma} \sigma(a)^{k_\sigma}$ for all $a\in \wt{F}^\times$, where $\sigma$ runs through all the embeddings of $\wt{F}$ into $\Qbar$. `Complex conjugation on $\wt{F}$' will always refer to the non-trivial element of $\Gal(\wt{F}/F)$ and will be denoted by $z \mapsto \bar{z}$. If $\sigma : \wt{F} \to \Qbar$ is an embedding, then $\bar{\sigma}$ denotes the conjugate embedding, obtained by precomposing with complex conjugation. The quantity $w := k_\sigma + k_{\bar{\sigma}}$ is independent of $\sigma$ and is called the weight of $\psi$. 

\medskip

Fix an embedding $\Qbar \to \Qbar_p$ and use it to identify the set $\Sigma_p$ of embeddings $\wt{F} \to \Qbar_p$ with the set of embeddings $\sigma : \wt{F} \to \Qbar$. We can decompose 
\[
\Sigma_p = \bigsqcup_{w|p} \Sigma_w
\]
where $\Sigma_w$ is the set of embeddings $\wt{F}_w \to \Qbar_p$. Let $v_p$ denote the $p$-adic valuation on $\Qbar_p$, normalized so that $v_p(p)=1$. Then we have the following key lemma.

\begin{lemma}\label{key lemma appendix}
Let $w\mid p$ be a place of $\wt{F}$, and let $\varpi_w$ be a uniformizer of $\wt{F}_w$ which we also view as the element of $\A_{\wt{F}}^\times$ with $\varpi_w$ in the $w$-component and $1$'s elsewhere. Then
\[
v_p(\psi(\varpi_w)) = e_w^{-1} \sum_{\sigma\in \Sigma_w} k_{\sigma},
\]
where $e_w$ is the ramification index of the extension $\wt{F}/\Qp$.
\end{lemma} 

\begin{proof}
Let $\p_w \sub \oo_{\wt{F}}$ denote the prime ideal corresponding to $w$ and choose $n\in \Z_{>0}$ so that $\p_w^n$ is principal. Let $a \in \oo_{\wt{F}}$ generate $\p^n_w$. We can write the id\`ele $\varpi^n_w$ as
\[
\varpi^n_w = aua_\infty^{-1}
\]
where $a\in \wt{F}^\times$ (embedded diagonally), $a_\infty^{-1}$ has $1$ at all finite components and $a^{-1}$ at all infinite components, and $u$ is simply $\varpi_w^na^{-1}a_\infty$. Since $a$ generates $\p_w^n$, one sees that $u \in \wh{\oo}_{\wt{F}}^\times$. Since $\psi$ has open kernel (and $\wt{F}$ is totally complex), we have $\psi(a_\infty) = 1$ and $\psi(u)$ is a root of unity. In particular, $v_p(\psi(a_\infty)) = v_p(\psi(u))=0$ and we see that
\[
v_p(\psi(\varpi_w^n)) = v_p(\psi(a)) = \sum_{\sigma \in \Sigma_p} k_\sigma v_p(\sigma(a)).
\]
By definition (and since $a$ generates $\p_w^n$), $v_p(\sigma(a))=0$ if $\sigma \notin \Sigma_w$ and $v_p(\sigma(a)) = e_w^{-1}n$ if $\sigma \in \Sigma_w$. Thus $v_p(\psi(\varpi_w^n)) = e_w^{-1}n \sum_{\sigma \in \Sigma_w} k_\sigma$, and dividing by $n$ gives the result.
\end{proof}

\begin{corollary}\label{condition on char for av=0}
Let $v\mid p$ be a place of $F$, which we assume splits as $v=w\bar{w}$ in $\wt{F}$. Let $\varpi_w$ and $\varpi_{\bar{w}}$ be uniformizers of $\wt{F}_w$ and $\wt{F}_{\bar{w}}$, respectively. Assume that $\psi(\varpi_w) = - \psi(\varpi_{\bar{w}})$. Then we must have
\[
\sum_{\sigma\in \Sigma_w} k_{\sigma} = \sum_{\sigma\in \Sigma_{\bar{w}}} k_{\sigma}.
\]
\end{corollary}

\begin{proof}
If $\psi(\varpi_w) = - \psi(\varpi_{\bar{w}})$, then we must have $v_p(\psi(\varpi_w)) = v_p( \psi(\varpi_{\bar{w}}))$. Since $e_w = e_{\bar{w}}$, the result follows directly from Lemma \ref{key lemma appendix}.
\end{proof}

Before discussing CM Hilbert modular forms, we construct examples of $\psi$ satisfying $\psi(\varpi_w) = - \psi(\varpi_{\bar{w}})$ for some $w$ and such that $k_\sigma \neq k_{\bar{\sigma}}$ for all $\sigma$ (the relevance of this condition, if not clear to the reader now, will be explained later). Let $L$ be the cyclotomic field $\Q(\mu_8)$, which we view as a subfield of $\Qbar$. For the facts about $L$ that we list below, we refer to \cite[https://www.lmfdb.org/NumberField/4.0.256.1]{lmfdb}. $L$ has class number one, is CM, Galois and biquadratic,  and its quadratic subfields are $\Q(\sqrt{2})$, $\Q(i)$ and $\Q(i\sqrt{2})$. The ring of integers of $L$ is $\oo_L=\Z[\mu_8]$ and the unit group $\oo_L^\times$ has rank one, with a fundamental unit given by $1+\sqrt{2}$ and torsion equal to $\mu_8$.

\medskip

Let $\sigma_1 \in \Gal(L/\Q(\sqrt{2}))$ and $\sigma_2 \in \Gal(L/\Q(i))$ be the non-trivial elements. Set $\sigma_0=id$ and $\sigma_3 = \sigma_1 \sigma_2 \in \Gal(L/\Q)$. Note that $\sigma_1$ is the complex conjugation of the CM field $L$. Since $L$ has class number one, we can write
\[
\A_L^\times = L^\times \wh{\oo}_L^\times L_\infty^\times.
\] 
Concretely, every $\alpha \in \A_L^\times$ can be written as $\alpha = \gamma u x$ with $\gamma \in L^\times$, $u \in \wh{\oo}_L^\times$ and $x \in L_\infty^\times$. Note that $5$ prime factorizes as $5=(2+i)(2-i)$ in $L$ and that $2+i$ and $2-i$ are complex conjugates. To try to avoid confusion, let us write $\varpi_+$ for the id\`ele that is $2+i$ at $L_{(2+i)}$ and $1$ elsewhere. Similarly, we write $\varpi_-$ for the id\`ele that is $2-i$ at $L_{(2-i)}$ and $1$ elsewhere.

\begin{lemma}\label{infinity type Hecke character}
Let $a,b \in \Z$ with $a+b = w \in 2\Z$. Set $k_0 = k_3 = a$ and $k_1 = k_2 = b$. Write $\alpha \in \A_L^\times$ as $\alpha = \gamma u x$ as above. Then
\[
\psi_{a,b}(\alpha) = \prod_{i=0}^3 \sigma_i(\gamma)^{k_i}
\]
is an algebraic character of $L$ of weight $(k_i)_{i=0}^3$. Moreover, $\psi_{a,b}(\varpi_+) = \psi_{a,b}(\varpi_-)$.
\end{lemma}

\begin{proof}
The first statement is clear by definition as long as $\psi_{a,b}$ is well defined (note that the kernel contains $\wh{\oo}_L^\times L_\infty^\times$). If we write $\alpha$ in a different way as $\alpha = \delta v y$ with $\delta \in L^\times$, $v \in \wh{\oo}_L^\times$ and $y \in L_\infty^\times$, then $\gamma = \delta \omega$ for some $\omega \in \oo_L^\times$, so we need to check that $\prod_{i=0}^3 \sigma_i(\omega)^{k_i} = 1$ for all $\omega \in \oo_L^\times$. This can be checked on generators, so it is enough to check it for $\omega = (1+i)/\sqrt{2}$ and $\omega = 1+\sqrt{2}$, which is a short straightforward calculation (using that $w$ is even). Similarly, a straightforward computation gives that $\psi_{a,b}(\vp_+) = \psi_{a,b}(\vp_-)=5^w$.
\end{proof}

We then have the following corollary.

\begin{corollary}\label{hecke character with av=0}
Let $a,b \in \Z$ with $a+b = w \in 2\Z$ and set $k_0 = k_3 = a$ and $k_1 = k_2 = b$. Then there exists an algebraic character $\psi$ of $L$ of weight $(k_i)_{i=0}^3$ which is unramified at $2+i$ and $2-i$ and such that  $\psi(\vp_+) = - \psi(\vp_-)$.
\end{corollary}

\begin{proof}
Let $L^\prime/L$ be a quadratic extension such that $2+i$ splits and $2-i$ is inert in $L^\prime$, which exists by e.g. \cite[\S 10.2, Theorem 5]{artin-tate}\footnote{Note that there is a misprint in the original version of \cite{artin-tate}, so that there are two `Theorem 5' in \S 10.2. We are citing the first one, which concerns the existence of characters with prescribed local behaviour at a finite set of places.}. If $\chi$ is the quadratic character corresponding to $L^\prime/L$, then we can view it as an algebraic character of $L$ which is trivial on $L^\times L_\infty^\times$ and satisfies $\chi(2+i)=1$ and $\chi(2-i)=-1$. Then, if $\psi_{a,b}$ is as in Lemma \ref{infinity type Hecke character}, $\psi = \chi \cdot \psi_{a,b}$ satisfies the requirements of the corollary.  
\end{proof}

We now return to Hilbert modular forms over $F$, or rather automorphic representations of $\GL_{2/F}$. By automorphic induction, every algebraic character $\psi$ of a CM extension $\wt{F}/F$ gives rise (via its associated Gr\"o{\ss}encharacter, see e.g. \cite[Theorem 2.43]{gee-mlt}) to an automorphic representation $\pi = \pi(\psi)$ of $\GL_2(\A_F)$. The weights $(k_\sigma)_\sigma$ of $\psi$ satisfy $k_\sigma \neq k_{\bar{\sigma}}$ for all $\sigma$ if and only if $\pi$ is cuspidal and cohomological. Moreover, $\pi$ is unramified at any finite place $v$ which is unramified in $\wt{F}$ and such that $\psi$ is unramified at all places above $v$. At such a place $v$, we denote by $a_v(\pi)$ be the eigenvalue of the Hecke operator $T_v$ (cf.\ Section \ref{subsec:Heckealgebras}). We then have $a_v(\pi)=0$ if $v$ is inert in $\wt{F}$, and 
\[
a_v(\pi) = \psi(\varpi_w) + \psi(\varpi_{\bar{w}})
\]
if $v$ splits as $v=w\bar{w}$, where $\varpi_w$ and $\varpi_{\bar{w}}$ are uniformizers in $\wt{F}_w$ and $\wt{F}_{\bar{w}}$, respectively. In particular, we obtain the following the result, which is what we wanted to establish in this subsection.

\begin{proposition}\label{behaviour of av at split primes}
Let $\psi$ be an algebraic character as above with weights $(k_\sigma)_\sigma$ of $\psi$ satisfying $k_\sigma \neq k_{\bar{\sigma}}$ for all $\sigma$ . Let $v|p$ and assume that $v$ splits as $v=w\bar{w}$ in $\wt{F}$, and that $\psi$ is unramified at $w$ and $\bar{w}$. Then:
\begin{enumerate}
\item If $\sum_{\sigma\in \Sigma_w} k_{\sigma} \neq \sum_{\sigma\in \Sigma_{\bar{w}}} k_{\sigma}$, then $a_v(\pi) \neq 0$. In particular, if $F_v= \Qp$, then $a_v(\pi) \neq 0$.

\smallskip

\item It can happen that $a_v(\pi)=0$.
\end{enumerate}  
\end{proposition}

\begin{proof}
We have $a_v(\pi)=0$ if and only if $\psi(\varpi_w)= - \psi(\varpi_{\bar{w}})$, so the first part follows from Corollary \ref{condition on char for av=0} and the second part follows from the example in Corollary \ref{hecke character with av=0}, choosing $a\neq b$.
\end{proof}

\subsection{Some quotient singularities}

Here we recall some facts about a certain type of quotient singularity. Let $H=\Z / 2$, where the non-trivial element acts on $A = E [[ x_1, \dots, x_n ]]$ by sending each $x_i$ to $-x_i$. The ring of invariants $B=A^H$ is a complete local domain. Its geometry is summarized by the following proposition. 

\begin{proposition}\label{quotientsingularities}
$B$ is normal and Cohen-Macaulay. Moreover:

\begin{enumerate}
\item If $n=1$, then $B$ is regular.

\item If $n=2$, then $B$ is a complete intersection but not regular.

\item If $n$ is odd and $\geq 3$, then $B$ is $2$-Gorenstein (but not Gorenstein).

\item If $n$ is even and $\geq 4$, then $B$ is Gorenstein but not a complete intersection.
\end{enumerate}

\end{proposition} 

\begin{proof}
We may equivalently work with the affine form $R=E[x_1,\dots,x_n] \sub A$ and $S=R^H$ instead. That $S$ is Cohen-Macaulay and normal is a general property of quotient singularities (indeed of rational singularities). When $n=1$, $S= E[x_1^2]$ and part (1) follows. Part (2) follows from the description $S = E[x_1^2,x_1 x_2, x_2^2] \cong E[u,v,w]/(uw-v^2)$ when $n=2$.

\medskip

From now on assume $n\geq 3$. Let $U \sub \Spec S = X$ be the complement of the origin and write $j$ for the inclusion. Then $\omega_X = j_\ast(\omega_U)$. In particular, $H^0(X,\omega_X) = H^0(U, \omega_U)$. Set $U_i = X \setminus \{ x_i = 0 \}$, then $U$ is the union of the $U_i$ and $H^0(U_i,\omega_{U})$ consists of all top forms
\[
\frac{f(x_1,\dots,x_n)}{x_i^b} \, dx_1 \wedge \dots \wedge dx_n,
\]
where the degrees of all terms in $f\in R$ have the same parity $d$, and $d-b \equiv n$ modulo $2$. Gluing, we see that $H^0(X,\omega_X)$ consists of the top forms
\[
g(x_1,\dots,x_n) \, dx_1 \wedge \dots \wedge dx_n
\]
where all terms in $g \in R$ have the same parity as $n$. In particular, $\omega_X$ is free of rank $1$ when $n$ is even, so $S$ is Gorenstein, but when $n$ is odd the top forms $x_i \, dx_1 \wedge \dots \wedge dx_n$, $i=1,\dots,n$, form a minimal set of generators of $\omega_X$ in any neighbourhood of $0$, so $S$ is then not Gorenstein. A similar calculation replacing $\omega_X$ by $\omega_X^{\otimes 2}$ shows that $S$ is $2$-Gorenstein when $n$ is odd.

\medskip

It remains to show that $B$ is not a complete intersection when $n$ is even. Let $T$ be the polynomial ring $T=E[u_{i},v_{jk}]$ with $i=1,\dots,n$ and $1\leq j < k \leq n$. $T$ surjects onto $S$ by sending $u_i$ to $x_i^2$ and $v_{jk}$ to $x_j x_k$. Let $I$ be the kernel of this surjection. Consider the elements $f_{ij} = u_i u_j - v_{ij}^2 \in I$. We claim that the $f_{ij}$ form a regular sequence, but that they do not generate $I$. Comparing the dimensions of $S$ and $T$ and the length of that sequence, this would imply that $S$ is not a complete intersection (by \cite[Tag 09PZ]{sta}). Set $Z_{ij} = \{ f_{ij} = 0 \} \sub \Spec T$. To show that the $f_{ij}$ form a regular sequence, it suffices to show, for each $i < j$, that no irreducible component of 
\[
Y_{ij} = \bigcap_{(a,b) \neq (i,j)} Z_{ab}
\]
is contained in $Z_{ij}$. Let $x \in Z_{ij}\cap Y_{ij}$ be any point. By deforming the $v_{ij}$-coordinate of $x$ but not the other coordinates, we see that we stay in $Y_{ij}$ but move out of $Z_{ij}$. This shows that no component of $Y_{ij}$ is contained in $Z_{ij}$, as desired, and finishes the proof that the $f_{ij}$ form a regular sequence. It remains to show that they do not generate $I$. To see this, consider, for example, the element $v_{12}v_{23} - u_2 v_{13} \in I$. If 
\[
v_{12}v_{23} - u_2 v_{13} = \sum_{i < j} g_{ij}f_{ij}
\] 
for some $g_{ij} \in T$, then we would obtain a contradiction by setting $u_i =1$ for all $i$ and $v_{jk} = -1$ for all $j < k$. This finishes the proof.
\end{proof}

\subsection{$p$-adic geometry}  Here we record a well known result about the local geometry of rigid analytic spaces, which we have not found in the literature in the form we need. To start, let us explicitly record the following (even more well known) proposition; see e.g. \cite[Chapter 4, Proposition 2]{bosch}.

\begin{proposition}
Let $X = \Spa(A)$ be an affinoid rigid space and let $x\in X$ be a point corresponding to a maximal ideal $\m$. Then the natural map $A \to \oo_{X,x}$ induces an isomorphism $\wh{A}_{\m} \to \wh{\oo}_{X,x}$.
\end{proposition} 

\medskip

We will make use of this in the paper without further mention, and we also use the following result.

\begin{proposition}\label{local isom}
Let $f : X \to Y$ be a morphism of rigid analytic varieties over a non-archimedean field $K$, and let $x\in X$ be a $K$-point with image $f(x)=y \in Y$. Assume that $f$ is locally quasi-finite at $x$. Then the following are equivalent:
\begin{enumerate}
\item $f$ is a local isomorphism at $x$.

\item $f$ is \'etale at $x$.

\item $f$ induces an isomorphism $\wh{\oo}_{Y,y} \cong \wh{\oo}_{X,x}$.
\end{enumerate}
\end{proposition}

\begin{proof}
That (1) implies (2) is trivial, and that (2) implies (3) is clear since $f$ induces an isomorphism on residue fields (since $x$ is defined over the base field $K$). It remains to prove that (3) implies (1). 

\medskip

First, note that $\oo_{Y,y} \to \oo_{X,x}$ is flat since it induces an isomorphism on completions and the local rings $\oo_{X,x}$ and $\oo_{Y,y}$ are Noetherian (using that completion of Noetherian local rings is faithfully flat). In other words, $f$ is flat at $x$. Since flatness is an open property, we may shrink $X$ to ensure that $f$ is flat. Second, using \cite[Proposition 1.5.4]{huber-book}, we may assume that $X$ and $Y$ are affinoid and $f$ is finite. Since the subset of rank $1$ points in $X$ is Hausdorff and $f^{-1}(y)$ is finite, a standard topology argument allows us to shrink $X$ and $Y$ so that $f^{-1}(y) = \{ x \}$. Let $A$ and $B$ be the affinoid $K$-algebras so that $X=\Spa B$ and $Y=\Spa A$, and consider $f^\ast : A \to B$. Then
\[
B \otimes_A \wh{\oo}_{Y,y} \cong \prod_{x^\prime \in f^{-1}(y)} \wh{\oo}_{X,x^\prime} = \wh{\oo}_{X,x},
\]
so $f^\ast$ has rank $1$ (since it has so after tensoring with $\wh{\oo}_{Y,y}$) and is therefore an isomorphism, as desired. This finishes the proof.
\end{proof}

\bibliography{endoscopybib}

\renewcommand{\MR}[1]{}\newcommand{\arxiv}[1]{\href{http://www.arxiv.org/abs/#1}{arXiv:#1}}
\begin{thebibliography}{{Lud}18b}

\bibitem[All19]{allen2}
Patrick~B. Allen.
\newblock On automorphic points in polarized deformation rings.
\newblock {\em Amer. J. Math.}, 141(1):119--167, 2019.

\bibitem[AS07]{ash-stevens}
Avner Ash and Glenn Stevens.
\newblock $p$-adic deformations of arithmetic cohomology.
\newblock preprint, 2007.

\bibitem[AT68]{artin-tate}
E.~Artin and J.~Tate.
\newblock {\em Class field theory}.
\newblock W. A. Benjamin, Inc., New York-Amsterdam, 1968.

\bibitem[BC09]{bellaiche-chenevier}
Jo\"el {Bella\"{\i}che} and Ga\"etan {Chenevier}.
\newblock {\em {Families of Galois representations and Selmer groups}}, volume
  324.
\newblock Paris: Soci\'et\'e Math\'ematique de France, 2009.

\bibitem[{Bel}08]{bellaiche}
Jo\"el {Bella\"{\i}che}.
\newblock {Nonsmooth classical points on eigenvarieties}.
\newblock {\em {Duke Math. J.}}, 145(1):71--90, 2008.

\bibitem[Bel12]{bellaiche-crit}
Jo\"{e}l Bella\"{\i}che.
\newblock Critical {$p$}-adic {$L$}-functions.
\newblock {\em Invent. Math.}, 189(1):1--60, 2012.

\bibitem[Bel15]{bellaiche-isobaric}
Jo\"{e}l Bella\"{\i}che.
\newblock Unitary eigenvarieties at isobaric points.
\newblock {\em Canad. J. Math.}, 67(2):315--329, 2015.

\bibitem[Bel21]{bellaiche-eigenbook}
Jo{\"e}l Bella{\"{\i}}che.
\newblock {\em The eigenbook. {Eigenvarieties}, families of {Galois}
  representations, {{\(p\)}}-adic {{\(L\)}}-functions}.
\newblock Pathw. Math. Cham: Birkh{\"a}user, 2021.

\bibitem[Ber08]{berger}
Laurent Berger.
\newblock {{\(p\)}}-adic differential equations and filtered {{\((\varphi,
  N)\)}}-modules.
\newblock In {\em Repr\'esentation \(p\)-adiques de groupes \(p\)-adiques I.
  Repr\'esentations galoisiennes et \((\varphi, \Gamma)\)-modules}, pages
  13--38. Paris: Soci{\'e}t{\'e} Math{\'e}matique de France, 2008.

\bibitem[Ber17]{bergdall-paraboline}
John Bergdall.
\newblock Paraboline variation over {$p$}-adic families of
  {$(\phi,\Gamma)$}-modules.
\newblock {\em Compos. Math.}, 153(1):132--174, 2017.

\bibitem[Ber19]{bergdall-upper-bounds}
John Bergdall.
\newblock Upper bounds for constant slope {$p$}-adic families of modular forms.
\newblock {\em Selecta Math. (N.S.)}, 25(4):Paper No. 59, 24, 2019.

\bibitem[Ber20]{bergdall-smooth}
John Bergdall.
\newblock Smoothness of definite unitary eigenvarieties at critical points.
\newblock {\em J. Reine Angew. Math.}, 759:29--60, 2020.

\bibitem[BH17]{bh}
John Bergdall and David Hansen.
\newblock On $p$-adic {L}-functions for {H}ilbert {M}odular {F}orms.
\newblock 2017.
\newblock To appear in Memoirs of the A.M.S.

\bibitem[BHKT19]{bhkt}
Gebhard B\"{o}ckle, Michael Harris, Chandrashekhar Khare, and Jack~A. Thorne.
\newblock {$\hat G$}-local systems on smooth projective curves are potentially
  automorphic.
\newblock {\em Acta Math.}, 223(1):1--111, 2019.

\bibitem[BHS17]{bhs2}
Christophe Breuil, Eugen Hellmann, and Benjamin Schraen.
\newblock Smoothness and classicality on eigenvarieties.
\newblock {\em Invent. Math.}, 209(1):197--274, 2017.

\bibitem[BHS19]{bhs3}
Christophe Breuil, Eugen Hellmann, and Benjamin Schraen.
\newblock A local model for the trianguline variety and applications.
\newblock {\em Publ. Math. Inst. Hautes \'{E}tudes Sci.}, 130:299--412, 2019.

\bibitem[Bos14]{bosch}
Siegfried Bosch.
\newblock {\em Lectures on formal and rigid geometry}, volume 2105 of {\em
  Lecture Notes in Mathematics}.
\newblock Springer, Cham, 2014.

\bibitem[CE98]{coleman-edixhoven}
Robert~F. Coleman and Bas Edixhoven.
\newblock On the semi-simplicity of the {$U_p$}-operator on modular forms.
\newblock {\em Math. Ann.}, 310(1):119--127, 1998.

\bibitem[Che51]{chevalley}
Claude Chevalley.
\newblock Deux th\'{e}or\`emes d'arithm\'{e}tique.
\newblock {\em J. Math. Soc. Japan}, 3:36--44, 1951.

\bibitem[Che05]{chenevier-jl}
Ga{\"e}tan Chenevier.
\newblock A {{\(p\)}}-adic {Jacquet}-{Langlands} correspondence.
\newblock {\em Duke Math. J.}, 126(1):161--194, 2005.

\bibitem[Che14]{chenevier-determinants}
Ga\"{e}tan Chenevier.
\newblock The {$p$}-adic analytic space of pseudocharacters of a profinite
  group and pseudorepresentations over arbitrary rings.
\newblock In {\em Automorphic forms and {G}alois representations. {V}ol. 1},
  volume 414 of {\em London Math. Soc. Lecture Note Ser.}, pages 221--285.
  Cambridge Univ. Press, Cambridge, 2014.

\bibitem[CT23]{caraiani-tamiozzo}
Ana Caraiani and Matteo Tamiozzo.
\newblock On the \'{e}tale cohomology of {H}ilbert modular varieties with
  torsion coefficients.
\newblock {\em Compos. Math.}, 159(11):2279--2325, 2023.

\bibitem[EGH23]{egh}
Matthew Emerton, Toby Gee, and Eugen Hellmann.
\newblock An introduction to the categorical $p$-adic {L}anglands program,
  2023.
\newblock arXiv:2210.01404v2.

\bibitem[EH14]{emerton-helm}
Matthew Emerton and David Helm.
\newblock The local {L}anglands correspondence for {${\rm GL}_n$} in families.
\newblock {\em Ann. Sci. \'{E}c. Norm. Sup\'{e}r. (4)}, 47(4):655--722, 2014.

\bibitem[Eis95]{eisenbud}
David Eisenbud.
\newblock {\em Commutative algebra. {With} a view toward algebraic geometry},
  volume 150 of {\em Grad. Texts Math.}
\newblock Berlin: Springer-Verlag, 1995.

\bibitem[Eme18]{emerson}
Kathleen Emerson.
\newblock {\em Comparison of Different Definitions of Pseudocharacter}.
\newblock PhD thesis, 2018.
\newblock ProQuest LLC.

\bibitem[FS24]{fargues-scholze}
Laurent Fargues and Peter Scholze.
\newblock Geometrization of the local {L}anglands correspondence, 2024.
\newblock arXiv:2102.13459v3.

\bibitem[{Fu}22]{fu}
Weibo {Fu}.
\newblock A derived construction of eigenvarieties, 2022.
\newblock arXiv:2110.04797v2.

\bibitem[Gee22]{gee-mlt}
Toby Gee.
\newblock Modularity lifting theorems.
\newblock {\em Essent. Number Theory}, 1(1):73--126, 2022.

\bibitem[GK82]{gelbart-knapp}
S.~S. {Gelbart} and A.~W. {Knapp}.
\newblock {L-indistinguishability and R groups for the special linear group}.
\newblock {\em {Adv. Math.}}, 43:101--121, 1982.

\bibitem[Har87]{harder}
G.~Harder.
\newblock Eisenstein cohomology of arithmetic groups. {T}he case {${\rm
  GL}_2$}.
\newblock {\em Invent. Math.}, 89(1):37--118, 1987.

\bibitem[Hel23]{hellmann2021derived}
Eugen Hellmann.
\newblock On the derived category of the {I}wahori-{H}ecke algebra.
\newblock {\em Compos. Math.}, 159(5):1042--1110, 2023.

\bibitem[HN17]{hansen}
David {Hansen} and James {Newton}.
\newblock {Universal eigenvarieties, trianguline Galois representations, and
  \(p\)-adic Langlands functoriality}.
\newblock {\em {J. Reine Angew. Math.}}, 730:1--64, 2017.

\bibitem[Hub96]{huber-book}
Roland Huber.
\newblock {\em \'{E}tale cohomology of rigid analytic varieties and adic
  spaces}.
\newblock Aspects of Mathematics, E30. Friedr. Vieweg \& Sohn, Braunschweig,
  1996.

\bibitem[JL70]{jacquet-langlands}
H.~{Jacquet} and R.~P. {Langlands}.
\newblock {\em {Automorphic forms on GL (2)}}, volume 114.
\newblock Springer, Cham, 1970.

\bibitem[JN19a]{JN1}
Christian {Johansson} and James {Newton}.
\newblock {Extended eigenvarieties for overconvergent cohomology.}
\newblock {\em {Algebra Number Theory}}, 13(1):93--158, 2019.

\bibitem[JN19b]{JN2}
Christian {Johansson} and James {Newton}.
\newblock {Irreducible components of extended eigenvarieties and interpolating
  Langlands functoriality.}
\newblock {\em {Math. Res. Lett.}}, 26(1):159--201, 2019.

\bibitem[Kis03]{kisin-overconvmfs}
Mark Kisin.
\newblock Overconvergent modular forms and the {F}ontaine-{M}azur conjecture.
\newblock {\em Invent. Math.}, 153(2):373--454, 2003.

\bibitem[Kis04]{kisin-geometricdef}
Mark Kisin.
\newblock Geometric deformations of modular {G}alois representations.
\newblock {\em Invent. Math.}, 157(2):275--328, 2004.

\bibitem[KPX14]{kpx}
Kiran~S. Kedlaya, Jonathan Pottharst, and Liang Xiao.
\newblock Cohomology of arithmetic families of {$(\varphi,\Gamma)$}-modules.
\newblock {\em J. Amer. Math. Soc.}, 27(4):1043--1115, 2014.

\bibitem[KW09]{khare-wintenberger}
Chandrashekhar Khare and Jean-Pierre Wintenberger.
\newblock Serre's modularity conjecture. {II}.
\newblock {\em Invent. Math.}, 178(3):505--586, 2009.

\bibitem[Laf18]{lafforgue}
Vincent Lafforgue.
\newblock Chtoucas pour les groupes r\'{e}ductifs et param\'{e}trisation de
  {L}anglands globale.
\newblock {\em J. Amer. Math. Soc.}, 31(3):719--891, 2018.

\bibitem[LL79]{labesse-langlands}
J.-P. {Labesse} and R.~P. {Langlands}.
\newblock {L-indistinguishability for SL(2)}.
\newblock {\em {Can. J. Math.}}, 31:726--785, 1979.

\bibitem[{LMF}23]{lmfdb}
The {LMFDB Collaboration}.
\newblock The {L}-functions and modular forms database.
\newblock \url{https://www.lmfdb.org}, 2023.
\newblock [Online; accessed 14 December 2023].

\bibitem[Loe11]{loeffler}
David Loeffler.
\newblock Overconvergent algebraic automorphic forms.
\newblock {\em Proc. Lond. Math. Soc. (3)}, 102(2):193--228, 2011.

\bibitem[LR07]{lansky-raghuram}
Joshua~M. {Lansky} and A.~{Raghuram}.
\newblock {Conductors and newforms for \(SL(2)\)}.
\newblock {\em {Pac. J. Math.}}, 231(1):127--153, 2007.

\bibitem[{Lud}18a]{L2}
Judith {Ludwig}.
\newblock {\(L\)-indistinguishability on eigenvarieties}.
\newblock {\em {J. Inst. Math. Jussieu}}, 17(2):425--440, 2018.

\bibitem[{Lud}18b]{L1}
Judith {Ludwig}.
\newblock {On endoscopic \(p\)-adic automorphic forms for \(\mathrm{SL}_2\)}.
\newblock {\em {Doc. Math.}}, 23:383--406, 2018.

\bibitem[Lun73]{luna}
Domingo Luna.
\newblock Slices \'{e}tales.
\newblock In {\em Sur les groupes alg\'{e}briques}, pages 81--105. Bull. Soc.
  Math. France, Paris, M\'{e}moire 33. 1973.

\bibitem[LZ19]{lafforgue-zhu}
Vincent Lafforgue and Xinwen Zhu.
\newblock D\'ecomposition au-dessus des param\`etres de {L}anglands
  elliptiques, 2019.
\newblock arXiv:1811.07976v2.

\bibitem[NT23]{newton-thorne}
James Newton and Jack~A. Thorne.
\newblock Adjoint {S}elmer groups of automorphic {G}alois representations of
  unitary type.
\newblock {\em J. Eur. Math. Soc. (JEMS)}, 25(5):1919--1967, 2023.

\bibitem[Pot13]{pottharst}
Jonathan Pottharst.
\newblock Analytic families of finite-slope {S}elmer groups.
\newblock {\em Algebra Number Theory}, 7(7):1571--1612, 2013.

\bibitem[Pro76]{procesi}
C.~Procesi.
\newblock The invariant theory of {$n\times n$} matrices.
\newblock {\em Advances in Math.}, 19(3):306--381, 1976.

\bibitem[{Ram}00]{ramakrishnan}
Dinakar {Ramakrishnan}.
\newblock {Modularity of the Rankin-Selberg \(L\)-series, and multiplicity one
  for \(\mathrm{SL}(2)\)}.
\newblock {\em {Ann. Math. (2)}}, 152(1):45--111, 2000.

\bibitem[Sah17]{saha}
Jyoti~Prakash Saha.
\newblock Conductors in {$p$}-adic families.
\newblock {\em Ramanujan J.}, 44(2):359--366, 2017.

\bibitem[{Sta}]{sta}
The {Stacks Project authors}.
\newblock {\em The {S}tacks {P}roject}.

\bibitem[SW20]{berkeley}
Peter Scholze and Jared Weinstein.
\newblock {\em Berkeley lectures on {$p$}-adic geometry}, volume 207 of {\em
  Annals of Mathematics Studies}.
\newblock Princeton University Press, Princeton, NJ, 2020.

\bibitem[Urb11]{urban}
Eric Urban.
\newblock Eigenvarieties for reductive groups.
\newblock {\em Ann. of Math. (2)}, 174(3):1685--1784, 2011.

\bibitem[Wei20]{weidner}
M.~Weidner.
\newblock Pseudocharacters of homomorphisms into classical groups.
\newblock {\em Transform. Groups}, 25(4):1345--1370, 2020.

\end{thebibliography}
\bibliographystyle{alpha}

\end{document}